\documentclass[12pt]{amsart}
\usepackage[all]{xy} 
\usepackage[usenames,dvipsnames]{pstricks}
\usepackage{amssymb,mathrsfs,euscript,enumitem,amsmath,bbold}
\usepackage{amssymb}
\usepackage[hypertex]{hyperref}
\usepackage[backgroundcolor=green,shadow,colorinlistoftodos]{todonotes}
\textheight9in  \def\DATE{November 26, 2016}
\allowdisplaybreaks

\newtheorem{theorem}{Theorem}

\newtheorem{lemma}[theorem]{Lemma}
\newtheorem{sublemma}[theorem]{Sublemma}
\newtheorem{proposition}[theorem]{Proposition}

\theoremstyle{definition}
\newtheorem{example}[theorem]{Example}
\newtheorem{remark}[theorem]{Remark}
\newtheorem{definition}[theorem]{Definition}

\def\velkyzvedak{{\rule{0pt}{.7em}}}
\def\KaPe{{KP}}
\long\def\comment#1\endcomment{\relax}
\def\envelope{completion}
\def\Lin{{\mbox {\it Lin\/}}}
\long\def\fatvert{\hbox{
\psscalebox{1 1} 
{
\begin{pspicture}(0,-.16)(0,0.098557666)
\psdots[linecolor=black, dotsize=0.34](0,0)
\end{pspicture}
}\hskip .3em}}
\def\choose{\atopwithdelims []}
\def\broucek{{brou\v cek}}

\def\oB{{\EuScript B}}
\def\squeezeddots{{\mbox{$\cdot \hskip -2pt \cdot \hskip -2pt \cdot$}}}
\def\podpera{{\rule{0em}{.56em}}}\def\mini{{\rule{0em}{.4em}}}
\def\oH{{\EuScript H}}
\def\ModHybFun{{\underline{\rm M}{\rm od}}}
\def\ModPreFun{{\underline{\rm Mo}{\rm d}}}
\def\CycHyb{\tt CycHyb}\def\ModHyb{\tt ModHyb}\def\PreHyb{\tt PreHyb}
\def\ColHyb{\tt ColHyb}
\def\nsoMod{\underline{\rm Mod}}

\def\korekce{{\rule{0pt}{.55em}}}
\def\Aut{{\rm Aut}}\def\oModgg{{\rm Mod}_{\rm gg}}
\def\gg{genus-graded}
\def\sfS{{\mathbb S}}

\def\zvedak{{\rule{0pt}{.6em}}}

\makeatletter
\newcommand{\oMod}{\@ifnextchar*{\oModNoArg}{\oModArg}}
\def\oModNoArg*{\mathrm{Mod}}
\def\oModArg#1{\mathrm{Mod}\left( #1 \right)}
\makeatother

\def\CycOp{{\tt CycOp}}
\def\CycOpgg{{\tt CycOp_{\rm gg}}}
\def\nsCycOp{\underline{\tt C}{\tt y}\underline{\tt cO}{\tt p}}
\def\nsCycOpgg{\underline{\tt C}{\tt y}\underline{\tt cO}{\tt p}_{\rm gg}}

\def\ModOp{{\tt ModOp}}
\def\nsModOp{\underline{\tt ModO}{\tt p}}

\def\preModOp{{\tt pre}\underline {\tt ModO}{\tt p}}
\def\pre{{\rm pre}}

\newcommand{\sfO}{{\sf O}}
\newcommand{\sfo}{{\sf o}}

\makeatletter
\newcommand{\sur}[2]{\@ifnextchar[{\surfacewithclosed{#1}{#2}}{\surfacewithoutclosed{#1}{#2}}}
\def\surfacewithclosed#1#2[#3]{ \left[ \substack{#1 \\ #2 \\ #3} \right] }
\def\surfacewithoutclosed#1#2{ \left[ \substack{#1 \\ #2} \right] }
\makeatother

\makeatletter
\newcommand{\surbig}[2]{\@ifnextchar[{\surfacewithclosedb{#1}{#2}}{\surfacewithoutclosedb{#1}{#2}}}
\def\surfacewithclosedb#1#2[#3]{ \Big[ \substack{#1 \\ #2 \\ #3} \Big] }
\def\surfacewithoutclosedb#1#2{\Big[\substack{\raisebox{.2em}{\scriptsize$#1$} \\ #2} \Big] }
\makeatother

\makeatletter
\newcommand{\ssur}[1]{\@ifnextchar[{\subsurface{#1}}{\onlydecomp{#1}}}
\def\onlydecomp#1{ \{#1\} }
\def\subsurface#1[#2]{ \substack{\{#1\} \\ #2} }
\makeatother

\newcommand{\oAss}{{\EuScript{A}\mathit{ss}}}
\newcommand{\oAssst}{{\EuScript{A}\mathit{ss}_{\rm st}}}
\newcommand{\nsoAss}{\underline{\EuScript{A}\mathit{ss}}}
\newcommand{\nsoAssst}{{\underline{\EuScript{A}\mathit{ss}}_{\rm st}}}
\newcommand{\oCom}{{\EuScript{C}\mathit{om}}}
\newcommand{\oComst}{{\EuScript{C}\mathit{om}_{\rm st}}}
\newcommand{\oQO}{{\EuScript{QO}}}
\newcommand{\oQOst}{{\EuScript{QO}_{\rm st}}}
\newcommand{\oQnsO}{{{\EuScript Q}\underline{\EuScript O}}}
\newcommand{\oQnsOst}{{\EuScript{Q}\underline{\EuScript O}_{\rm st}}}
\newcommand{\oQC}{{\EuScript{QC}}}
\newcommand{\oQCst}{{\EuScript{QC}_{\rm st}}}
\newcommand{\oQOC}{{\EuScript{QOC}}}
\newcommand{\oQOCst}{{\EuScript{QOC}_{\rm st}}}
\newcommand{\oQnsOC}{{\EuScript{Q}\underline{\EuScript{O}}\EuScript{C}}} 
\newcommand{\oQnsOCst}{{\EuScript{Q}\underline{\EuScript{O}}\EuScript{C}_{\rm st}}} 
\newcommand{\oOC}{{\EuScript{OC}}} 
\newcommand{\onsOC}{\underline{\EuScript{O}}{\EuScript{C}}} 
\newcommand{\onsOCst}{\underline{\EuScript{O}}{\EuScript{C}_{\rm st}}} 
\newcommand{\onsOCKP}{\underline{\EuScript{O}}{\EuScript{C}_{\rm KP}}} 
\newcommand{\oQnsOCKP}{\EuScript{Q}\underline{\EuScript{O}}{\EuScript{C}_{\rm KP}}} 
\newcommand{\onsOCpre}{\underline{\EuScript{O}}{\EuScript{C}^{\rm pre}}}

\newcommand{\oOCst}{{\EuScript{OC}_{\rm st}}} 
\newcommand{\oOMB}{\underline{\EuScript{O}}}

\newcommand{\KP}{\mathrm{KP}}

\newcommand{\ns}{non-$\Sigma$ }
\newcommand{\Ns}{Non-$\Sigma$ }
\def\MultCyc{{\tt MultCyc}}

\def\Set{{\tt Set}}\def\id{\hbox {$1 \!\! 1$}}

\def\({(}    \def\){)}

\def\and{\ \hbox { and } \ }

\def\({(\hskip -.15em(}    \def\){)\hskip -.15em)}
    
\def\Des{{\rm Des}}

\def\R#1{\boxed{\hbox{\scriptsize $#1$}}\,}
\def\Rrom#1{\boxed{\hbox{\scriptsize #1}}\,}

\catcode`\@=11

\def\@evenfoot{\rule{0pt}{20pt}[\DATE] \hfill [{\tt \jobname.tex}]}
\def\@oddfoot{\rule{0pt}{20pt}{[\tt \jobname.tex}]\hfill [\DATE]}
\catcode`\@=13

\def\rada#1#2{{#1,\ldots,#2}}

\def\nsM{{\underline{\EuScript M}}}

\def\Free{{{\mathbb M}}}

\def\Fin{{\tt Fin}}
\def\bfk{{\mathbb k}}

\def\Com{\mbox{{$\mathcal C$}\hskip -.3mm {\it om}}}

\def\Set{{\tt Set}}
\newcommand{\oP}{{\mathcal{P}}}
\newcommand{\oM}{{\EuScript M}}
\newcommand{\ooo}[2]{\sideset{_{#1}}{_{#2}}{\mathop{\circ}}}
\def\stt#1{{\{#1\}}}
\def\bbN{{\mathbb N}}\def\ot{\otimes}
\def\k-Mod{\hbox{\tt $\bfk$-Mod}}
\def\MultCyc{{\tt MultCyc}}
\def\Sym{{\rm Sym}}\def\Mod{{\rm Mod}}

\def\xxi{{\circ}}
\def\oP{{\mathcal P}}

\catcode`\@=11

\def\@evenfoot{\rule{0pt}{20pt}[\DATE] \hfill [{\tt \jobname.tex}]}
\def\@oddfoot{\rule{0pt}{20pt}{[\tt \jobname.tex}]\hfill [\DATE]}
\catcode`\@=13

\newcommand{\oxi}[1]{\xxi_{#1}}
\newcommand{\oC}{{\EuScript{C}}}

\newcommand{\cyc}[1]{\left(\hspace{-.6ex}\left( #1 \right)\hspace{-.6ex}\right)}	
\newcommand{\tF}{\tilde{F}}
\newcommand{\pF}{F'}
\newcommand{\ppF}{F''}
\newcommand{\cycs}[1]{\left(\hspace{-.3ex}\left( #1 \right)\hspace{-.3ex}\right)}

\newcommand{\St}{{\mathrm{st}}}

\def\End{\hbox{${\mathcal E}\hskip -.1em {\it nd}$}}

\makeatletter
\newcommand{\Fr}{\@ifnextchar[{\FrSscipt}{\FrNoSscipt}}
\def\FrSscipt[#1]#2{\mathbb{F}_{\!\rm #1}(#2)}
\def\FrNoSscipt#1{\mathbb{F}(#1)}
\makeatother

\makeatletter
\providecommand\@dotsep{5}
\def\listtodoname{List of Todos}
\def\listoftodos{\@starttoc{tdo}\listtodoname}
\makeatother

\title[Modular operads, Cardy condition and SFT]{Open-closed modular
  operads, Cardy condition 
\\ 
and string field theory}

\author[M.\ Doubek \& M.\ Markl]{Martin Doubek and Martin Markl}

\catcode`\@=11
\address{Mathematical Institute of the Academy, {\v Z}itn{\'a} 25,
         115 67 Prague 1, The Czech Republic}
\address{Faculty of Mathematics and Physics, Charles University,
186 75 Sokolovsk\'a 83, Prague~8, The Czech Republic}
\email{markl@math.cas.cz}

\catcode`\@=13

\thanks{The first author was supported by GA\v{C}R P201/13/27340P.
The second author was supported by the Eduard \v Cech Institute
  P201/12/G028 and RVO: 67985840.}

\begin{document}
\bibliographystyle{plain}

\parskip3pt plus 1pt minus .5pt
\baselineskip 17pt

\begin{abstract}
We prove that the modular operad of diffeomorphism
classes of Riemann surfaces with both `open' and `closed' boundary components,
in the sense of string field theory, is the modular \envelope\ of its
genus~$0$ part quotiented by the Cardy condition. We also provide a
finitary presentation of a version
of this modular two-colored operad and
characterize its algebras via morphisms of Frobenius algebras,
recovering some previously known results of Kaufmann, Penner and
others. 
As an important auxiliary tool we
characterize inclusions of cyclic operads that induce inclusions of
their modular \envelope{s}.    
\end{abstract}


\renewcommand*{\thefootnote}{\fnsymbol{footnote}}
\setcounter{footnote}{0}

\maketitle

\begin{center}
  {\em Dedicated to the memory of Martin Doubek who died in a traffic
  accident
\\ 
on 29th of August~2016, in the middle of work on the present paper.\/}
\end{center}

\setcounter{tocdepth}{1}
\tableofcontents

\section*{Introduction}

\paragraph{\bf History of the subject}
Barton Zwiebach constructed in~\cite{zwiebach:NuclPh93} 
`string products' on the
Hilbert space of closed string field theory satisfying the `master
equation' which reflected the structure of the set $\oQC$ of diffeomorphism
classes of Riemann surfaces of arbitrary genera with labelled holes. As we
proved in \cite{markl:la}, the master equation expresses 
that the string
products form an algebra over the Feynman transform of $\oQC$. We
moreover proved that  $\oQC$ is the modular \envelope\footnote{Often
  called a modular {\em envelope\/} in recent literature. We take
  the liberty to keep our original terminology.} of its cyclic
suboperad $\oCom \subset \oQC$ 
consisting of Riemann surfaces of genus~$0$, that is
\begin{subequations}
\begin{equation}
\label{zitra_Hanka_u_Pakousu}
\oQC \cong \Mod(\oCom).
\end{equation}   

Later in~\cite{doubek:modass} we proved a similar statement for open
strings. Namely, we identified the modular operad $\oQO$ of
diffeomorphism classes of Riemann
surfaces with marked `open' boundaries with the  modular \envelope\ of
its genus zero part $\oAss$, i.e.\ we established an isomorphism
\[
\oQO \cong \Mod(\oAss).
\]
As a follow-up to~\cite{doubek:modass} we argued in~\cite{nsmod}
that $\oQO$ is the symmetrization of a more elementary object $\oQnsO$
bearing the structure of a non-$\Sigma$ modular operad. The previous
isomorphism then follows from a more elementary
\begin{equation}
\label{zitra_Hanka_u_Pakousu_a_Jarka}
\oQnsO \cong \nsoMod(\nsoAss),
\end{equation}
where $\nsoAss$ is the non-$\Sigma$ version of the associative cyclic
operad and $\nsoMod(-)$ the non-$\Sigma$ modular \envelope\ functor.
  
\paragraph{\bf Aims} 
We complete the story and establish analogs
of the results mentioned above
for the combined theory of open and closed
strings. The central object will be 
the set $\oQnsOC$ of diffeomorphism classes of
Riemann surfaces with both open and closed inputs. It behaves as a
non-$\Sigma$-modular operad in the open and as an ordinary modular
operad in the closed inputs; we call these structures {\em modular
  hybrids\/}. 
Contrary to expectations, it turns out that $\oQnsOC$ is \underline{not} the
modular \envelope\ of its genus~$0$ part $\onsOC$, but the quotient of
this \envelope\
by the {\em Cardy conditions\/} known to
physicists~\cite{cardy-lewellen,lauda} that is, symbolically,
\begin{equation}
\label{jsem_zavislak}
\oQnsOC \cong \ModHybFun(\onsOC)/\hbox{Cardy}.
\end{equation}
\end{subequations}
The unusual feature of the Cardy conditions is that they involve both the
open and closed interactions.
The above isomorphism restricted to closed resp.~open parts
gives~(\ref{zitra_Hanka_u_Pakousu})
resp.~(\ref{zitra_Hanka_u_Pakousu_a_Jarka}), so it is indeed the
culmination of the development described in the previous
paragraph. As a~bonus, we obtain a purely combinatorial proof of 
a~result of~\cite{kaufmann-penner:NP06} 
characterizing algebras over a version of $\oQnsOC$ in terms of
morphisms of Frobenius algebras.

\paragraph{\bf Why is the paper so unbearably long?} 
It is so because we establish three different versions
of~(\ref{jsem_zavislak}): 
the `ordinary,' stable and the
Kaufmann-Penner version, each having its own merit -- `ordinary'
version involves everything that makes sense,
stability
prevents the combinatorial explosion of the Feynman
transform~\cite[II.5.4]{markl-shnider-stasheff:book}, 
while the Kaufmann-Penner version\footnote{Abbreviated
``\KaPe'' at some places in the sequel.} 
admits a nice finitary presentation so its algebras can be described easily. 
In more detail, the cyclic hybrid $\onsOC$ contains the stable and
\KaPe\ subhybrids $\onsOCst$ and  $\onsOCKP$ such that
\[
\onsOC \supset\onsOCst\supset\onsOCKP.
\] 
Since, as demonstrated in
Example~\ref{Chtel_jsem_jet_na_bruslicky_ale_prsi.},  
the modular \envelope\ functor need not preserve inclusions, it
is not a priory clear whether
\[
\ModHybFun(\onsOC) \supset \ModHybFun(\onsOCst) \supset
\ModHybFun(\onsOCKP).
\] 
A substantial 
part of this paper is devoted to the proof that it is indeed the case,
therefore the stable and \KaPe\ cases can be treated as the
restricted versions of the ordinary one. 

\paragraph{\bf Our approach}
There are two approaches to the structures of (topological) string field theory.
The classical one of~\cite{Atiyah} interprets
surfaces as cobordisms, with the corresponding combinatorial
structure 
being that of a PROP. 
The second one does not discriminate between `outputs' and `inputs,'
and the relevant 
combinatorial structure is a~modular
operad. The difference on the algebra level is that, while in the first
approach the bilinear form on the underlying space forms a part of the
structure, in the second approach, adopted
e.g.\ in~\cite{kaufmann-penner:NP06} and this article, the bilinear form
is absorbed in the definition of the modular endomorphism operad.

\paragraph{\bf Main results}
The central technical result is
{\em Proposition~\ref{THMModOCus}\/}  which, 
together with {\em Propositions~\ref{THMModOCStable}\/}
and~{\em\ref{leze_na_mne_chripka}\/}, provides a 
combinatorial description of the
modular \envelope\ of the cyclic hybrid $\onsOC$ and its versions.
Our description enables one to interpret, in
Remark~\ref{za_10_dni_asi_naposledy_do_Sydney},  elements of
this completion as diffeomorphism classes of certain Riemann surfaces with
embedded~loops.

Using the above propositions we obtain 
the main results of this paper -- three versions
of the isomorphism~(\ref{jsem_zavislak}): the `ordinary' one 
in {\em Theorem~\ref{PROCUnivProp}\/}, stable one in
{\em Theorem~\ref{PROOCUnivPropST}\/} 
and the Kaufmann-Penner in {\em
  Theorem~\ref{THMOCUnivPropKP}\/}. 
An interesting byproduct 
is {\em Theorem~\ref{THMPresentationOfQOCKP}\/} describing  the \KaPe\ version  
$\oQnsOCKP$ of $\oQnsOC$ in terms
of generators and relations, together with a 
characterization of algebras over $\oQnsOCKP$ as couples
of Frobenius algebras connected by a morphism satisfying the Cardy
and centrality conditions
given in {\em Theorem~\ref{omylem_jsem_si_koupil_bio_banany}\/}. 
Variants of these results are known, 
see e.g.~\cite{kaufmann-penner:NP06,lauda}, but our approach provides a 
purely combinatorial~proof.   

It turns out that $\onsOC$  and its versions are more than
just cyclic hybrids as they admit a partial
modular hybrid structure - they are closed under contractions of open inputs
belonging to the same boundary component since this operation does not
change the geometric genus. We 
call structures of this
type {\em premodular hybrids\/} and denote $\onsOC$ with this extended
structure by $\onsOCpre$. 
In {\em Theorem~\ref{THMPositiveModEnvOC}\/} we prove that $\oQnsOC$ 
can be alternatively described as the modular \envelope\
of this premodular hybrid. Since the Cardy condition already 
lives in $\onsOCpre$,
no quotienting is necessary. Finally, 
{\em Proposition~\ref{THMSuffCondOnInjectivity}\/}
characterizing inclusions
of cyclic operads inducing inclusions of their modular
\envelope{s} is interesting in its own~right.

\begin{center}
{\bf Plan of the paper}
\end{center}

{\em Section~\ref{sec:participants-game}\/} begins with recalling the necessary
facts about cyclic and modular operads, and their
non-$\Sigma$ versions. We then introduce various versions of hybrids
as structures that combine ordinary and non-$\Sigma$
operads. This section also contains definitions of concrete operads
featuring in this article. We believe that Table~\ref{analogies} helps to
navigate through~them.

{\em Section~\ref{SSECModOCus}\/} is devoted to modular \envelope{s}
of cyclic hybrids and to their quotients by the Cardy condition. It
contains the main technical results of this paper,
Proposition~\ref{THMModOCus} and Theorems~\ref{PROCUnivProp}
and~\ref{THMPositiveModEnvOC}.   

{\em Section~\ref{Dnes_se_pojedu_podivat_do_Ribeauville}\/} 
has an auxiliary character. Its
Proposition~\ref{THMSuffCondOnInjectivity} describes inclusions $\oB
\hookrightarrow \oC$ of cyclic operads that induce inclusions $\Mod(\oB)
\hookrightarrow \Mod(\oC)$ of their   
modular \envelope{s}.

In {\em Section~\ref{sec:stable-kaufm-penn}\/}
we use the results of
Section~\ref{Dnes_se_pojedu_podivat_do_Ribeauville} to derive
the stable and
Kaufmann-Penner versions of the theorems in Section~\ref{SSECModOCus}.

Theorem~\ref{THMPresentationOfQOCKP} of 
{\em Section~\ref{dnes_prodlouzit_prukaz na_LAA}\/} provides
a finitary presentation of 
the \KaPe\ modular hybrid $\oQnsOCKP$. As its application 
we obtain a result of Kaufmann and Penner
describing its algebras in terms of morphisms of
Frobenius algebras.

\paragraph{\bf Acknowledgment} 
We are indebted to Ralph Kaufmann for
explaining to us what the Cardy condition is.

\noindent
{\bf Conventions.}   
We will assume working knowledge of operads and their versions.  Suitable
references are
monographs~\cite{loday-vallette,markl-shnider-stasheff:book} 
complemented
with~\cite{markl:handbook} and the original sources~\cite{getzler-kapranov:CPLNGT95,getzler-kapranov:CompM98,ginzburg-kapranov:DMJ94}. 
Modular \envelope{s} were introduced in~\cite{markl:la} 
and non-$\Sigma$ modular operads in~\cite{nsmod}. Sundry facts about
operads relevant for string field theory can be found
e.g.~in~\cite{braun,burke,chuang-lazarev:dual-fey,gaberdiel-zwiebach:NP97,harrelson,jurco-munster,kaufmann09,kaufmann-livernet-penner,kaufmann-penner:NP06,munster-sachs,saadi-zwiebach,zwiebach:NuclPh93}. 
Operads in this article will, with few exceptions, live in
the symmetric monoidal category $\Set$ of sets. We will denote by
$\id_X$ or simply by $\id$ when $X$ is understood, 
the identity automorphism of an object
$X$ (set, operad, vector space, \&c.).

We will denote by $\bbN_+$ the set $\{1,2,\ldots\}$ of positive
integers, by $\bbN$ the abelian semigroup $\{0,1,2,\ldots\}$ of
non-negative integers,  and by $\frac{1}{2}\bbN$ the semigroup $\{n/2\
|\ n\in\bbN\}$ of half-integers. By $\Set$ we denote the category of
sets, by $\Fin$ the category of 
finite sets; $|S| \in \bbN$ will denote the cardinality of $S \in
\Fin$.

Operads considered in this article may have `inputs' of two types --
open and closed. We will tend to use $O$ as the default notation for
open inputs, and $C$ for the closed ones. The
operations $\ooo uv$ in cyclic operads will be termed
`$\circ$-operations,' while the operations 
$\oxi {uv}$\footnote{In ancient times denoted $\xi_{uv}$. Notation due
to R.~Kaufmann.} 
in modular operads will be called `contractions.'

\section{Participants in the game}
\label{sec:participants-game}

\begin{table}
\def\arraystretch{1.3}
\begin{center}
  \begin{tabular}{|l|l|l|l|} 
\hline \multicolumn{1}{|c|}{\textbf{Symbol}} 
& \multicolumn{1}{c|}{\textbf{Name}}
& \multicolumn{1}{c|}{\textbf{Type}} &
\multicolumn{1}{c|}{\textbf{Found in:}}
\\
\hline \hline 
$\oAss$ & associative operad & cyclic operad &
Example~\ref{poleti_se_zitra?}
\\
\hline 
$\oAssst$ & stable associative operad & cyclic operad &
Example~\ref{poleti_se_zitra_stable}
\\
\hline 
$\oQO$ & quantum open operad & modular operad &
Example~\ref{jsem_na_zavodech_a_prsi}
\\
\hline 
$\oQOst$ & stable quantum open operad & modular operad &
Example~\ref{jsem_na_zavodech_a_prsi_stable}
\\
\hline 
$\oCom$ & commutative operad & \gg\ cyclic operad &
Example~\ref{ze_ja_vul_jsem_zacal_psate_ten_grant}
\\
\hline 
$\oComst$ & stable commutative operad & \gg\ cyclic operad &
Example~\ref{ze_ja_vul_jsem_zacal_psate_ten_grant_stable}
\\
\hline 
$\oQC$ & quantum closed operad & modular operad & 
Example~\ref{za_chvili_musim psat ten zatracenej_grant}
\\
\hline 
$\oQCst$ & stable quantum closed operad & modular operad & 
Example~\ref{za_chvili_musim psat ten zatracenej_grant_stable}
\\
\hline 
$\oQOC$ & quantum open-closed operad & modular operad & 
Example~\ref{DEFQOC}
\\
\hline 
$\oQOCst$ & stable quantum open-closed operad & modular operad & 
Example~\ref{DEFQOC_stable}
\\
\hline 
$\oOC$ & open-closed operad & \gg\ cyclic operad & 
Example~\ref{na_obed_k_Japoncum_s_Jarcou}
\\
\hline 
$\oOCst$ & stable open-closed operad & \gg\ cyclic operad & 
Example~\ref{Dnes_naposledy_v_Myluzach_na_varhanky.}
\\
\hline 
$\nsoAss$ & \ns associative  operad & \ns cyclic operad & 
Example~\ref{prvni_holubinky}
\\
\hline 
$\nsoAssst$ & stable \ns associative  operad & \ns cyclic operad & 
Example~\ref{prvni_holubinky_stable}
\\
\hline 
$\oQnsO$ & \ns quantum open  operad & \ns modular operad & 
Example~\ref{ani_v_patek_se_neleti}
\\
\hline 
$\oQnsOst$ & stable \ns quantum open  operad & \ns modular operad & 
Example~\ref{ani_v_patek_se_neleti_stable}
\\
\hline 
$\oOMB$ &   {multiple boundary operad}    & premodular operad & 
Subsect.~\ref{premodular}
\\
\hline 
$\onsOC$ & open-closed hybrid & cyclic hybrid & 
Example~\ref{h1}
\\
\hline 
$\oQnsOC$ & quantum open-closed hybrid  & modular hybrid & 
Example~\ref{hm}
\\
\hline 
$\onsOCpre$ & open-closed hybrid   & premodular hybrid & 
Example~\ref{h2}
\\
%
\hline 
$\onsOCst$ &stable  open-closed hybrid & cyclic hybrid & 
Example~\ref{Dnes_naposledy_v_Myluzach_na_varhanky.}
\\
\hline 
$\onsOCKP$ &K.-P.\ open-closed hybrid & cyclic hybrid & 
Example~\ref{KP}
\\
\hline 
$\oQnsOCKP$ &K.-P.\ open-closed hybrid & modular hybrid & 
Definition~\ref{za_chvili_s_Nunykem_na_veceri}
\\
\hline 
$\oQnsOCst$ &stable  quantum open-closed hybrid  & modular hybrid & 
Example~\ref{DEFQOC_stable}
\\
\hline
\end{tabular}
\\
\rule{0em}{.8em}
\end{center}
\caption{Operads and operad-like structures featuring in this article; see also
  diagrams~(\ref{d}) and~(\ref{dst}).}
\label{analogies}
\end{table}

Most of the material recalled in this section already appeared in the
literature or is an harmless modification of the existing notions. The
only novel concept is that of premodular
operads and premodular hybrids introduced in Definition~\ref{zitra_do_Brna}.

\subsection{Standard versions.}
We start with the following innocuous generalization 
of cyclic operads.

\begin{definition}
A {\em genus-graded\/}  cyclic operad is a cyclic operad 
with an additional grading by the `operadic genus' (or simply 
the genus) belonging to an abelian unital semigroup $\sfS$. 
\end{definition}

In other words, \gg\ cyclic operads 
are cyclic operads in the cartesian monoidal
category of $\sfS$-graded sets.
The components $\oC(S)$, $S \in
\Fin$,  of a genus-graded cyclic operad $\oC$ are thus disjoint unions
\[
\oC(S) ={\hbox {\LARGE $\sqcup$}}_{G\in \sfS}\oC(S;G)
\]
such that the structure maps
$
\ooo ab:\oC\big(S_1 \sqcup \stt a)
\times \oC\big(S_2\sqcup \stt b)
\to \oC ( S_1\sqcup S_2)
$
restrict to
\[
\ooo ab:\oC\big(S_1 \sqcup \stt a; G_1)
\times \oC\big(S_2\sqcup \stt b; G_2)
\longrightarrow \oC ( S_1\sqcup S_2; G_1 + G_2)
\]
for arbitrary $S_1,S_2 \in \Fin$ and  $G_1,G_2 \in \sfS$. 
In this article, $\sfS$ will either be $\bbN$ or
$\frac{1}{2}\bbN$.  

A morphism of genus-graded cyclic operads is a 
morphism of the underlying cyclic operads preserving the genus.
We let $\CycOp$ to denote the category of ordinary cyclic operads (no
genus grading) and
$\CycOpgg$ the category of genus-graded ones.
Taking the genus $0$ part and ignoring the remaining ones leads to the forgetful functor 
\[
\CycOpgg \longrightarrow \CycOp.
\]
On the other hand, every cyclic operad can be viewed as a
genus-graded cyclic one concentrated in genus $0\in \sfS$ with
empty components in higher genera. This gives an inclusion of
categories
\[
\iota:
\CycOp \hookrightarrow \CycOpgg
\]
which is the left adjoint of the forgetful functor above.

We will modify  
the standard definition of {\em modular operads\/} as well, 
by allowing the operadic genus $G$ to belong to a
semigroup $\sfS$ containing $\bbN$.
This is necessary since we want our theory to accommodate the quantum
open-closed operad $\oQOC$ recalled later in this section, 
cf.~Example~\ref{DEFQOC} and Remark~\ref{Je vedro takze na bruslicky
  pujdu zitra vecer.}. 
Let $\ModOp$ denote the category of these $\sfS$-graded modular
operads. The concrete $\sfS$ will always be clear from the context.

Forgetting the operadic contractions 
\[
\oxi{uv}:\oM(S \sqcup \stt {u,v};G)\to\oM(S,G+1),\  S\in \Fin,\ G
\in \sfS,
\]
every modular operads $\oM$ becomes a \gg\ cyclic operad. This 
gives rise to the forgetful functor 
\[
\square_{\rm gg} : \ModOp \to \CycOpgg.
\]
Its left adjoint will be denoted
$\oModgg:\CycOpgg \to \ModOp$.  One also has the `standard' forgetful
functor $\square : \ModOp \to \CycOp$, which replaces everything
outside genus $0$ by the empty set $\emptyset$.  Its left adjoint $\oMod* :
\CycOp \to \ModOp$ is the standard modular \envelope\ functor
introduced in~\cite[page~382]{markl:la}.

The above functors fit into the 
following diagram of adjunctions in which
the top arrows are the left adjoints to the bottom ones:
\begin{equation}
\label{je_to_tady_pruda}
\raisebox{-6em}{}
\xymatrix@C=4em@R=4em{&\CycOpgg \ar@/^.7pc/@<.1ex>[dr]^{\oModgg} \ar@/^.7pc/@<.1ex>[dl]  &
\\
\CycOp \ar@/^.7pc/@<.1ex>[rr]^{\oMod*} 
\ar@/^.7pc/@<.1ex>@{^{(}->}[ur]^\iota
&& \ar@<.1ex>@/^.7pc/[ll]_{\square}\ \ModOp\ . 
\ar@<.1ex>@/^.7pc/[ul]_{\square_{\rm gg}}
}
\end{equation}

\paragraph{\bf Orders}
At this point we need to recall some notions of~\cite{nsmod} used
later in this article in the definition of non-$\Sigma$ modular and
quantum open-closed operads.

\begin{definition}
\label{zase_jsem_podlehl_uz_po_trech_dnech}
A \emph{cycle} on a finite set 
$O=\{o_1,\ldots,o_n\}\in \Fin$ is an equivalence class of total orders on $O$
modulo the equivalence generated by
\[
(o_1,o_2,\ldots,o_n) \equiv (o_2,\ldots,o_n,o_1).
\]
A cycle represented by $(o_1,\ldots,o_n)$ will be denoted by
$\cyc{o_1,\ldots,o_n}$. The empty set has a unique cycle on it denoted
$\cyc{}$. As
in~\cite{nsmod}  we will sometimes call
cycles the {\em pancakes\/}, imagining them placed in the plane and
oriented anticlockwise.
\end{definition}

To save the space, we will sometimes leave out the commas, i.e.\ write
e.g.~$\cycs{o_1 o_2 o_3}$ instead of ~$\cycs{o_1, o_2, o_3}$. The same
simplification will be used also for finite sets, i.e.\ we will write
e.g.~$\stt{abc}$ instead of $\stt{a,b,c}$.

\begin{definition}
\label{o_vikendu_za_Bernasem} 
A {\em multicycle\/} $\sfO$ on a finite set $O\in \Fin$ consists of
\begin{enumerate}
\item[(i)] 
a disjoint unordered decomposition 
$O = o_1 \sqcup \cdots \sqcup o_b$ of the underlying set
$O$ into $b \geq 0$ possibly empty sets, and
\item[(ii)] 
a cycle $\sfo_i$ on each $o_i$, $1 \leq i \leq b$.
\end{enumerate}
In the above situation we write $\sfO=\sfo_1\cdots\sfo_b$.
\end{definition}

V.~Nov{\' a}k in~\cite{novak} introduced (partial or total) 
cyclic orders on a set. It has the property that the
disjoint union of cyclically ordered sets bears an
induced cyclic order. 
In Nov{\' a}k's terminology, 
a cycle on a finite set $O$ is the same as a total cyclic order of $O$
while a multicycle on $O$ determines a partial cyclic order on $O$
induced from the total cyclic orders of its components.
Since we allowed the sets $o_i$ in (ii) to be empty, his cyclic order on $O$
does not determine the multicycle uniquely as it cannot detect the trivial
$\sfo_i$'s which are part of the structure.
Notice that $b=0$ in Definition~\ref{o_vikendu_za_Bernasem} 
is possible only for $O=\emptyset$. We denote the corresponding
trivial multicycle by $\varnothing$.

In~\cite{nsmod} we introduced two operations on cycles.
The \emph{merging} of pancakes $\cyc{o'_1,\ldots,o'_m}$ 
and $\cyc{o''_1,\ldots,o''_n}$  is defined as
\begin{equation}
\label{eq:1merging}
\cyc{o'_1,\ldots,o'_m} \ooo{o'_m}{o''_1} \cyc{o''_1,\ldots,o''_n} :=
\cyc{o'_1,\ldots,o'_{m-1},o''_2,\ldots,o''_{n}}.
\end{equation}
Invoking the invariance of cycles
under cyclic permutations we see that~(\ref{eq:1merging}) in fact determines 
the merging
\[
\cyc{o'_1,\ldots,o'_m} \ooo{o'_i}{o''_j} 
\cyc{o''_1,\ldots,o''_n}
\]
for arbitrary $1 \leq i \leq m$, $1 \leq j \leq n$. The following picture
explains the terminology:
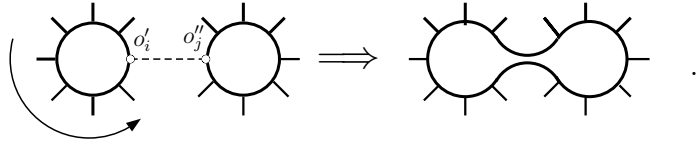
\[
\raisebox{-3em}{}
\psscalebox{.5.5}
{
\begin{pspicture}(0,-.4)(18.316029,1.84)
\pscircle[linecolor=black, linewidth=0.08, dimen=outer](2.4460278,0.020668488){1.0}
\pscircle[linecolor=black, linewidth=0.08, dimen=outer](6.4460278,0.020668488){1.0}
\psline[linecolor=black, linewidth=0.04, linestyle=dashed, dash=0.17638889cm 0.10583334cm](3.4460278,0.020668488)(5.4460278,0.020668488)(5.4460278,0.020668488)(5.4460278,0.020668488)
\psline[linecolor=black, linewidth=0.08](3.1460278,0.7206685)(3.546028,1.1206685)(3.546028,1.1206685)
\psline[linecolor=black, linewidth=0.08](1.7460278,0.7206685)(1.3460279,1.1206685)(1.3460279,1.1206685)
\psline[linecolor=black, linewidth=0.08](0.9460278,0.020668488)(1.4460279,0.020668488)(1.4460279,0.020668488)
\psline[linecolor=black, linewidth=0.08](2.4460278,-0.9793315)(2.4460278,-1.4793315)(2.4460278,-1.4793315)
\psline[linecolor=black, linewidth=0.06](3.1460278,-0.67933154)(3.546028,-1.0793315)(3.546028,-1.0793315)
\psline[linecolor=black, linewidth=0.06](1.7460278,-0.67933154)(1.3460279,-1.0793315)(1.3460279,-1.0793315)
\psline[linecolor=black, linewidth=0.08](6.4460278,1.5206685)(6.4460278,1.5206685)(6.4460278,1.0206685)
\psline[linecolor=black, linewidth=0.06](6.4460278,-0.9793315)(6.4460278,-1.4793315)(6.4460278,-1.4793315)
\psline[linecolor=black, linewidth=0.08](2.4460278,1.5206685)(2.4460278,1.0206685)(2.4460278,1.0206685)
\psline[linecolor=black, linewidth=0.06](5.246028,-0.9793315)(5.246028,-0.9793315)(5.646028,-0.5793315)
\psline[linecolor=black, linewidth=0.08](7.146028,0.7206685)(7.5460277,1.1206685)(7.5460277,1.1206685)
\psline[linecolor=black, linewidth=0.08](5.746028,0.7206685)(5.346028,1.1206685)(5.346028,1.1206685)
\psline[linecolor=black, linewidth=0.06](7.646028,-0.9793315)(7.246028,-0.5793315)(7.246028,-0.5793315)
\psline[linecolor=black, linewidth=0.08](13.046028,0.7206685)(13.446028,1.1206685)(13.446028,1.1206685)
\psline[linecolor=black, linewidth=0.08](11.646028,0.7206685)(11.246028,1.1206685)(11.246028,1.1206685)
\psline[linecolor=black, linewidth=0.06](10.846027,0.020668488)(11.346027,0.020668488)(11.346027,0.020668488)
\psline[linecolor=black, linewidth=0.06](12.346027,-0.9793315)(12.346027,-1.4793315)(12.346027,-1.4793315)
\psline[linecolor=black, linewidth=0.06](13.046028,-0.67933154)(13.446028,-1.0793315)(13.446028,-1.0793315)
\psline[linecolor=black, linewidth=0.06](11.646028,-0.67933154)(11.246028,-1.0793315)(11.246028,-1.0793315)
\psline[linecolor=black, linewidth=0.08](15.546028,1.5206685)(15.546028,1.5206685)(15.546028,1.0206685)
\psline[linecolor=black, linewidth=0.06](15.546028,-0.9793315)(15.546028,-1.4793315)(15.546028,-1.4793315)
\psline[linecolor=black, linewidth=0.08](12.346027,1.5206685)(12.346027,0.9206685)(12.346027,0.9206685)
\psline[linecolor=black, linewidth=0.08](14.846027,0.6206685)(14.446028,1.1206685)(14.846027,0.6206685)
\psline[linecolor=black, linewidth=0.06](16.746027,-0.9793315)(16.446028,-0.67933154)(16.446028,-0.67933154)
\psarc[linecolor=black, linewidth=0.08, dimen=outer](13.996028,1.0706685){0.95}{207.51837}{327.6011}
\psarc[linecolor=black, linewidth=0.08, dimen=outer](15.646028,0.020668488){1.0}{208}{147}
\psarc[linecolor=black, linewidth=0.08, dimen=outer](12.346027,0.020668488){1.0}{34.346157}{328.253}
\psarc[linecolor=black, linewidth=0.08, dimen=outer](13.996028,-1.0293316){0.95}{28.095282}{150.83333}
\psarc[linecolor=black, linewidth=0.04, dimen=outer, arrowsize=.2cm 2.0,arrowlength=1.4,arrowinset=0.0]{->}(2.3460279,0.020668488){2.1}{165}{311}
\rput[bl](8.4,-.23){\psscalebox{2.5 2.5}{$\Longrightarrow$}}
\rput[br](4,0.2206685){\psscalebox{1.5 1.5}{$o'_i$}}
\rput[bl](4.85,0.2206685){\psscalebox{1.5 1.5}{$o''_j$}}
\psline[linecolor=black, linewidth=0.08](16.346027,0.7206685)(16.346027,0.7206685)(16.746027,1.1206685)
\psline[linecolor=black, linewidth=0.06](14.446028,-0.9793315)(14.446028,-0.9793315)(14.846027,-0.5793315)
\psline[linecolor=black, linewidth=0.06](16.646029,0.020668488)(17.246027,0.020668488)(17.246027,0.020668488)
\psline[linecolor=black, linewidth=0.06](7.4460278,0.020668488)(7.9460278,0.020668488)(7.9460278,0.020668488)
\psdots[linecolor=black, dotstyle=o, dotsize=0.23125](5.4460278,0.020668488)
\rput{92.38133}(3.6098614,-3.4215245){\psdots[linecolor=black, dotstyle=o, dotsize=0.23125](3.4460278,0.020668488)}
\end{pspicture}
}.
\]

The second operation on cycles is the \emph{pancake cutting},  defined
by the formula
\[
\oxi{o_1o_i} \cyc{o_1,\ldots,o_i,\ldots,o_n} := 
\cyc{o_2,\ldots,o_{i-1}} \cyc{o_{i+1},\ldots,o_n}, \ 1 < i \leq n,
\]
whose result is a multicycle with two cycles. The invariance under
the cyclic group action determines
$\oxi{o_io_j} \cyc{o_1,\ldots,o_i,\ldots,o_j,\ldots,o_n}$ for
each $1 \leq  i \leq j \leq n$. The intuition behind this operation
is explained by the picture below.

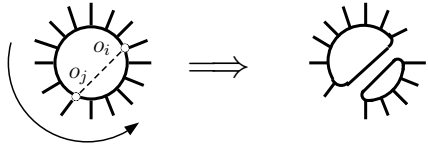
\begin{center}
\psscalebox{.5.5} 
{
\begin{pspicture}(0,-2.3344152)(11.346027,2.3344152)
\pscircle[linecolor=black, linewidth=0.08, dimen=outer](2.4460278,-0.029330902){1.0}
\psline[linecolor=black, linewidth=0.08](3.1460278,0.6706691)(3.546028,1.070669)(3.546028,1.070669)
\psline[linecolor=black, linewidth=0.08](1.7460278,0.6706691)(1.3460279,1.070669)(1.3460279,1.070669)
\psline[linecolor=black, linewidth=0.08](0.9460278,-0.029330902)(1.4460279,-0.029330902)(1.4460279,-0.029330902)
\psline[linecolor=black, linewidth=0.08](2.4460278,-1.0293308)(2.4460278,-1.5293308)(2.4460278,-1.5293308)
\psline[linecolor=black, linewidth=0.08](3.1460278,-0.7293309)(3.546028,-1.1293309)(3.546028,-1.1293309)
\psline[linecolor=black, linewidth=0.08](2.4460278,1.570669)(2.4460278,0.9706691)(2.4460278,0.9706691)
\psarc[linecolor=black, linewidth=0.04, dimen=outer, arrowsize=0.2cm 2.0,arrowlength=1.4,arrowinset=0.0]{->}(2.3460279,-0.029330902){2.1}{165}{311}
\rput[bl](2.5,0.1706691){\psscalebox{1.5 1.5}{$o_i$}}
\rput[bl](1.85,-0.6){\psscalebox{1.5 1.5}{$o_j$}}
\rput[bl](5,-.4){\psscalebox{2.5 2.5}{$\Longrightarrow$}}
\psdots[linecolor=black, fillstyle=solid, dotstyle=o, dotsize=0.23125](2.046028,-0.9293309)
\psdots[linecolor=black, fillstyle=solid, dotstyle=o, dotsize=0.23125](3.3460279,0.3706691)
\psline[linecolor=black, linewidth=0.08](3.3460279,-0.42933092)(3.8460279,-0.6293309)(3.8460279,-0.6293309)
\psline[linecolor=black, linewidth=0.08](2.8460279,-0.9293309)(3.046028,-1.429331)(3.046028,-1.429331)
\psline[linecolor=black, linewidth=0.08](2.046028,0.8706691)(1.8460279,1.4706692)(1.8460279,1.4706692)
\psline[linecolor=black, linewidth=0.08](2.8460279,0.8706691)(3.046028,1.3706691)(3.046028,1.3706691)
\psline[linecolor=black, linewidth=0.08](10.446028,0.6706691)(10.846027,0.8706691)(10.846027,0.8706691)
\psline[linecolor=black, linewidth=0.08](9.046028,0.6706691)(8.646028,1.070669)(8.646028,1.070669)
\psline[linecolor=black, linewidth=0.08](8.246028,-0.029330902)(8.746028,-0.029330902)(8.746028,-0.029330902)
\psline[linecolor=black, linewidth=0.08](9.746028,-1.0293308)(9.746028,-1.5293308)(9.746028,-1.5293308)
\psline[linecolor=black, linewidth=0.06](10.446028,-0.7293309)(10.846027,-1.1293309)(10.846027,-1.1293309)
\psline[linecolor=black, linewidth=0.08](9.746028,1.570669)(9.746028,0.9706691)(9.746028,0.9706691)

\psline[linecolor=black, linewidth=0.08](10.646028,-0.42933092)(11.146028,-0.6293309)(11.146028,-0.6293309)
\psline[linecolor=black, linewidth=0.08](10.146028,-0.9293309)(10.346027,-1.429331)(10.346027,-1.429331)
\psline[linecolor=black, linewidth=0.08](8.846027,0.3706691)(8.346027,0.6706691)(8.346027,0.6706691)
\psline[linecolor=black, linewidth=0.08](9.346027,0.8706691)(9.146028,1.4706692)(9.146028,1.4706692)
\psline[linecolor=black, linewidth=0.08](10.146028,0.8706691)(10.346027,1.3706691)(10.346027,1.3706691)
\psline[linecolor=black, linewidth=0.08](3.4460278,0.3706691)(3.9460278,0.5706691)(3.9460278,0.5706691)
\psline[linecolor=black, linewidth=0.04, linestyle=dashed, dash=0.17638889cm 0.10583334cm](2.1460278,-0.8293309)(3.246028,0.2706691)(3.246028,0.2706691)
\psline[linecolor=black, linewidth=0.08](3.4460278,-0.029330902)(4.0460277,-0.029330902)(4.0460277,-0.029330902)
\psline[linecolor=black, linewidth=0.08](0.9460278,0.5706691)(1.5460278,0.3706691)(1.5460278,0.3706691)
\psline[linecolor=black, linewidth=0.08](10.746028,-0.029330902)(11.346027,-0.029330902)(11.346027,-0.029330902)
\psline[linecolor=black, linewidth=0.08](1.6460278,-0.5293309)(1.1460278,-0.8293309)
\psline[linecolor=black, linewidth=0.08](8.946028,-0.5293309)(8.646028,-0.9293309)
\psline[linecolor=black, linewidth=0.08](10.446028,-0.7293309)(10.846027,-1.1293309)
\psline[linecolor=black, linewidth=0.08](1.9460279,-1.0293308)(1.6460278,-1.429331)
\psarc[linecolor=black, linewidth=0.08, dimen=outer](9.696028,0.020669097){0.95}{38.997017}{210.56715}
\psarc[linecolor=black, linewidth=0.08, dimen=outer](9.796028,-0.0793309){0.95}{274.397}{355.44293}
\psarc[linecolor=black, linewidth=0.08, dimen=outer](10.246028,0.5706691){0.2}{299}{60}
\psarc[linecolor=black, linewidth=0.08, dimen=outer](9.096027,-0.3793309){0.25}{199.6999}{319.36636}
\psarc[linecolor=black, linewidth=0.08, dimen=outer](10.596027,-0.0793309){0.15}{341}{145}
\psarc[linecolor=black, linewidth=0.08, dimen=outer](9.796028,-0.8793309){0.15}{129.16226}{281.66156}
\psline[linecolor=black, linewidth=0.08](10.346027,0.3706691)(9.246028,-0.6293309)
\psline[linecolor=black, linewidth=0.08](10.546028,0.0706691)(9.646028,-0.8293309)
\psline[linecolor=black, linewidth=0.08](9.746028,-1.0293308)(9.946028,-1.0293308)(9.946028,-1.0293308)
\psline[linecolor=black, linewidth=0.08](10.746028,-0.029330902)(10.746028,-0.2293309)
\end{pspicture}
}
\end{center}

The pancake merging of cycles can be extended to multicycles as:
\begin{subequations}
\begin{equation}
\label{1}
\left( \sfo'_1\cdots\sfo'_{i'}\cdots\sfo'_{b'} \right) \ooo{a}{b}
\left( \sfo''_1\cdots\sfo''_{i''}\cdots\sfo''_{b''} \right): =
\sfo'_1\cdots\widehat{\sfo}'_{i'}\cdots\sfo'_{b'}\
\sfo''_1\cdots\widehat{\sfo}''_{i''}\cdots\sfo''_{b''}\
(\sfo'_{i'}\ooo{a}{b}\sfo''_{i''}),
\end{equation}
where $a$ and $b$ belong to the underlying sets of $\sfo'_{i'}$
resp.~$\sfo''_{i''}$ for some $1 \leq i' \leq b'$, $1 \leq i'' \leq
b''$, and the hat indicates the omission as usual.
To extend the pancake cutting, i.e.\ to define $\oxi{uv}
\left( \sfo_1\cdots\sfo_b \right)$, 
 we need to distinguish two cases.
It might happen that $u$ and $v$ belong to 
the same pancake, i.e. $u,v\in\sfo_{i}$
for some $1 \leq i \leq b$. Then we put
\begin{equation}
\label{2}
\oxi{uv} \left( \sfo_1\cdots\sfo_i\cdots\sfo_b \right) :=
\sfo_1\cdots (\oxi{uv}\sfo_i) \cdots\sfo_b.
\end{equation}
The second possibility is that $u\in\sfo_{i'}$ and $v\in\sfo_{i''}$ 
for $1 \leq i'\neq i'' \leq b$. Then
\begin{equation}
\label{3}
\oxi{uv} \left( \sfo_1\cdots\sfo_{i'}\cdots\sfo_{i''}\cdots\sfo_b \right) 
:= \sfo_1\cdots\widehat{\sfo}_{i'}\cdots\widehat{\sfo}_{i''}\cdots\sfo_b \
(\sfo_{i'}\ooo uv\sfo_{i''}).
\end{equation}
\end{subequations}
We have therefore defined the multicycles
\[
\sfO'  \ooo{a}{b} \sfO'' \hbox { resp. } \oxi{uv} \sfO
\]
for arbitrary multicycles $\sfO',\sfO''$ resp.\ $\sfO$ on finite sets $O',O''$
resp.\ $O$, with elements $a\in O'$, $b\in O''$ resp.\ $u,v \in O$.
Pancake  merging offers an
effective definition of the cyclic associative operad $\oAss$:

\begin{example}
\label{poleti_se_zitra?}
The component $\oAss(O)$ of the cyclic associative operad $\oAss$
is, for $O\in \Fin$, the set cycles on $O$, that is
\[
\oAss(O) := \left\{ \sfo \ |\ \sfo\textrm{ is a cycle on }O
\right\}.
\]
Clearly
$|\oAss(O)|= (|O|-1)!$\,. 
The structure operations are given by the pancake merging.
An automorphism $\rho \in \Aut(O)$ acts on the set $\oAss(O)$ by 
\[
\rho\cyc{o_1,\ldots,o_n} := \cyc{\rho(o_1),\ldots,\rho(o_n)}.
\]

We will denote by $\oAss$
both the cyclic operad $\oAss$  and its  \gg\
version $\iota(\oAss)$. The meaning will always be clear from the context.
\end{example}

\begin{example}
\label{jsem_na_zavodech_a_prsi}
The $(O;G)$-component of the quantum open modular operad $\oQO$ is,
for \hbox{$O \in \Fin$} and $G \in \bbN$
defined as
\begin{equation}
\label{dnes_jsem_vstal_brzo}
\oQO(O;G) := \left\{\rule{0em}{1.4em} \surbig{\sfo_1\cdots\sfo_b}{g} \
  |\ b\in \bbN_+,
  g\in\bbN, G=2g\!+\! b\!-\! 1, \sfo_1\cdots\sfo_b\textrm{ is a multicycle
    on }O \right\}
\end{equation}
while the other components are empty.
In~(\ref{dnes_jsem_vstal_brzo}), $\sur{\sfo_1\cdots\sfo_b}{g}$ is a
formal symbol depending on the multicycle $\sfo_1\cdots\sfo_b$ and on
a non-negative integer $g$ determined by the operadic genus by the
formula $G=2g+b-1$.  Less formally, $\sur{\sfo_1\cdots\sfo_b}{g}$
specifies the diffeomorphism class of a two-dimensional oriented genus
$g$ surface $\Sigma$ with $b$ `open' boundaries with teeth labelled by
elements of $O$ on boundaries portrayed in Figure~\ref{zubata_plocha}
taken from~\cite{nsmod}.  For this reason we call the number $g$ the
\emph{geometric genus} of $\sur{\sfo_1\cdots\sfo_b}{g}\in \oQO$.  The
operadic structure is given by connecting the teeth with ribbons so
that the orientability is not violated. 
\begin{figure}
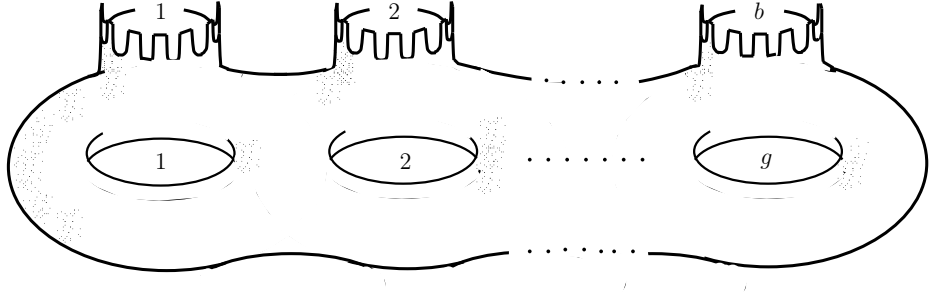

\begin{center}
\psscalebox{.25.25}
{

}
\end{center}
\caption{\label{zubata_plocha}An oriented surface $\Sigma$ with $b$ toothed 
boundaries and
  genus $g$.}
\end{figure}
Notice that we assume that $b
\geq 1$, so $\Sigma$ has at least one open boundary.
The operadic genus $G$ is related to the Euler characteristics $\chi$
of $\Sigma$
by $G = 1-\chi$.

Notice that not all combinations of $G$ and $b$ are allowed, for
instance $G=b=1$ would imply $g = \frac12$. The assumption that $g \in
\bbN$ is precisely the {\em geometricity\/}
of~\cite{nsmod}.\footnote{Notice however that in~\cite{nsmod} the
  symbols $g$ and $G$ are interchanged.}

\begin{subequations}
The $ \ooo{a}{b}$-operations are given by the pancake merging 
as
\begin{equation}
\label{a}
\surbig{ \sfo'_1\cdots\sfo'_{b'}}{g'} \ooo{a}{b}
\surbig{ \sfo''_1\cdots\sfo''_{b''}}{g''}: =
\surbig{(\sfo'_1\cdots\sfo'_{b'})  \ooo{a}{b}
(\sfo''_1\cdots\sfo''_{b''})}{g' + g''},
\end{equation}
where $g',g''\in \bbN$ and the meaning of the remaining symbols is the
same as in~(\ref{1}). The contractions are given by the pancake
cutting, i.e. in the notation of~(\ref{2}) resp.~(\ref{3}),
\begin{equation}
\label{b}
\oxi{uv} \surbig{ \sfo_1\cdots\sfo_b }g :=
\begin{cases}{
\surbig{\oxi{uv}(\sfo_1\cdots\sfo_b)}g} & \hbox{if $u$ and $v$ belong to the
same pancake, and} 
\\
\rule{0pt}{1.9em}
\surbig{\oxi{uv}(\sfo_1\cdots\sfo_b)}{g+1} &
\hbox{if they belong to different pancakes.}
\end{cases}
\end{equation}
\end{subequations}
Automorphisms $\rho \in \Aut(O)$ act according to the formula
\[
\rho\surbig{\cycs{o_1^1\cdots o^1_{n_1}}\cdots\cycs{o_1^b \cdots
    o^b_{n_b}}}{g} 
:= 
\surbig{\cycs{\rho(o_1^1)\cdots \rho(o^1_{n_1})}
\cdots\cycs{\rho(o_1^b) \cdots
    \rho(o^b_{n_b})}}{g}.
\]
\end{example}

The only solution of $G=2g\!+\! b\!-\! 1$ with  $G=0$ for
$b\in \bbN_+$ and $g\in\bbN$ is $b=1$, $g=0$. Therefore the genus $0$
component of $\oQO$ equals the associative operad $\oAss$, so 
the injection
\[
\oAss \hookrightarrow \oQO,\ 
\sfo \mapsto \surbig{\ \sfo\ }{0}
\]
defines an isomorphism $\oAss \cong\square(\oQO)$.

We proved in \cite{doubek:modass} that $\oMod*(\oAss) \cong
\oQO$. By~(\ref{je_to_tady_pruda}), $\oMod*(\oAss) \cong
\oModgg\big(\iota(\oAss)\big)$, therefore, under the identification
$\iota(\oAss) = \oAss$, one has
\[
\oModgg(\oAss) \cong \oQO.
\]

In the following example we introduce a \gg\ version of the cyclic
operad describing commutative Frobenius algebras. Its standard
definition is modified in such a way that it forms
a \gg\ cyclic suboperad of the quantum open-closed operad  $\oQOC$
recalled in Example~\ref{DEFQOC} below.

\begin{example}
\label{ze_ja_vul_jsem_zacal_psate_ten_grant}
The component $\oCom(C;G)$ of the cyclic \gg\ operad $\oCom$ is, for a
finite nonempty
set $C \in \Fin$  and a non-negative half-integer 
$G \in \frac 12\bbN$ satisfying
\begin{equation}
\label{dam_si_s_Jarcou_pivo_u_Zabacka}
G=-1+|C|/2\footnote{Notice that $G \in \frac 12\bbN$ implies $|C| 
  \geq 2$.}
\end{equation}
defined as
\[
\oCom(C;G) := 
 \left\{ C \right\},
\]
while $\oCom(C;G)$ is empty for other pairs $(C,G)$. So all
non-empty components of $\oCom$ are one-point sets indexed by $C \in \Fin$. 
The operadic composition is the 
only possible one and automorphisms from $\Aut(C)$ act trivially.

It is useful in some situations to represent the unique element of
$\oCom(C;G)$ as  the diffeomorphism class of genus-$0$ compact oriented
surfaces with holes indexed by $C$. In this visualization, the
operadic composition is given by connecting these holes by tubes,
as indicated in Figure~\ref{Jarka_stale_bez_prace} taken from~\cite{nsmod}.  
\begin{figure}[t]
\begin{center}
\psscalebox{.35.35}
{
\begin{pspicture}(0,-3.45)(6.9,3.45)
\pscircle[linecolor=black, linewidth=0.1, dimen=outer](3.45,0.0){3.45}
\rput{-5.0}(-0.1949039,0.33596677){\psellipse[linecolor=black, linewidth=0.04, dimen=outer](3.75,2.4)(0.75,0.45)}
\psellipse[linecolor=black, linewidth=0.04, linestyle=dashed, dash=0.17638889cm 0.10583334cm, dimen=outer](5.95,-0.05)(0.45,0.7)
\rput{15.0}(-0.42458665,-1.1749442){\psellipse[linecolor=black, linewidth=0.04, dimen=outer](4.25,-2.2)(0.75,0.55)}
\rput{30.0}(-0.30753174,-0.7522759){\psellipse[linecolor=black, linewidth=0.04, dimen=outer](1.25,-0.95)(0.45,0.7)}
\rput{15.0}(0.19349791,-0.8697625){\psellipse[linecolor=black, linewidth=0.04, dimen=outer](3.4,0.3)(0.7,0.65)}
\rput[l](3.5,2.45){\psscalebox{2.8}{\scriptsize $a$}}
\rput[l](1.05,-0.9){\psscalebox{2.8}{\scriptsize $b$}}
\rput[l](4.1,-2.19){\psscalebox{2.8}{\scriptsize $c$}}
\rput[l](5.75,-0.05){\psscalebox{2.8}{\scriptsize $d$}}
\rput[l](3.25,0.3){\psscalebox{2.8}{\scriptsize $e$}}
\end{pspicture}
}  
\psscalebox{.35.35}{
\begin{pspicture}(0,-4.1714683)(22.653332,4.1714683)
\rput{-5.0}(-0.11898602,0.91770875){\psellipse[linecolor=black, linewidth=0.04, dimen=outer](10.45,1.8214682)(0.75,0.45)}
\psellipse[linecolor=black, linewidth=0.04, linestyle=dashed, dash=0.17638889cm 0.10583334cm, dimen=outer](12.65,-0.6285317)(0.45,0.7)
\rput{15.0}(-0.34602472,-2.9287448){\psellipse[linecolor=black, linewidth=0.04, dimen=outer](10.95,-2.7785318)(0.75,0.55)}
\rput{30.0}(0.30083218,-4.1797843){\psellipse[linecolor=black, linewidth=0.04, dimen=outer](7.95,-1.5285317)(0.45,0.7)}
\rput{15.0}(0.27205983,-2.623563){\psellipse[linecolor=black, linewidth=0.04, linestyle=dashed, dash=0.17638889cm 0.10583334cm, dimen=outer](10.1,-0.27853173)(0.7,0.65)}
\rput[l](12.45,-0.6285317){\psscalebox{2.8}{\scriptsize $a$}}
\rput[l](7.75,-1.58){\psscalebox{2.8}{\scriptsize $x$}}
\rput[l](10.75,-2.8){\psscalebox{2.8}{\scriptsize $c$}}
\rput[l](10.2,1.8){\psscalebox{2.8}{\scriptsize $d$}}
\rput[l](9.9,-0.3){\psscalebox{2.8}{\scriptsize $e$}}
\pscustom[linecolor=black, linewidth=0.04]
{
\newpath
\moveto(14.8,2.9714682)
}
\pscustom[linecolor=black, linewidth=0.04]
{
\newpath
\moveto(15.9,2.7714682)
}
\pscustom[linecolor=black, linewidth=0.04]
{
\newpath
\moveto(12.9,-0.028531723)
\lineto(12.937313,-0.072281726)
\curveto(12.955969,-0.09415672)(13.000746,-0.13165672)(13.026865,-0.14728172)
\curveto(13.052984,-0.16290672)(13.131343,-0.19415672)(13.183581,-0.20978172)
\curveto(13.235821,-0.22540672)(13.392538,-0.25665674)(13.497015,-0.27228174)
\curveto(13.601493,-0.28790674)(13.978359,-0.30040672)(14.250747,-0.2972817)
\curveto(14.523135,-0.29415673)(14.821642,-0.29103172)(14.9,-0.29103172)
}
\pscustom[linecolor=black, linewidth=0.04]
{
\newpath
\moveto(16.8,-0.028531723)
\lineto(16.764551,-0.072281726)
\curveto(16.746828,-0.09415672)(16.70429,-0.13165672)(16.679478,-0.14728172)
\curveto(16.654665,-0.16290672)(16.580225,-0.19415672)(16.530598,-0.20978172)
\curveto(16.480968,-0.22540672)(16.33209,-0.25665674)(16.232836,-0.27228174)
\curveto(16.133581,-0.28790674)(15.77556,-0.30040672)(15.516792,-0.2972817)
\curveto(15.258022,-0.29415673)(14.974441,-0.29103172)(14.9,-0.29103172)
}
\pscustom[linecolor=black, linewidth=0.04]
{
\newpath
\moveto(12.9,-1.2285317)
\lineto(12.937313,-1.1847817)
\curveto(12.955969,-1.1629068)(13.000746,-1.1254067)(13.026865,-1.1097817)
\curveto(13.052984,-1.0941567)(13.131343,-1.0629067)(13.183581,-1.0472817)
\curveto(13.235821,-1.0316567)(13.392538,-1.0004067)(13.497015,-0.98478174)
\curveto(13.601493,-0.96915674)(13.978359,-0.9566567)(14.250747,-0.9597817)
\curveto(14.523135,-0.9629067)(14.821642,-0.96603173)(14.9,-0.96603173)
}
\pscustom[linecolor=black, linewidth=0.04]
{
\newpath
\moveto(16.9,-1.3285317)
\lineto(16.862686,-1.270198)
\curveto(16.84403,-1.2410318)(16.799253,-1.1910317)(16.773134,-1.170198)
\curveto(16.747015,-1.1493648)(16.668657,-1.1076986)(16.616419,-1.0868654)
\curveto(16.56418,-1.0660317)(16.407463,-1.0243654)(16.302984,-1.0035317)
\curveto(16.198507,-0.98269796)(15.821641,-0.96603173)(15.5492525,-0.970198)
\curveto(15.276865,-0.9743649)(14.978358,-0.9785317)(14.9,-0.9785317)
}
\psellipse[linecolor=black, linewidth=0.04, linestyle=dashed, dash=0.17638889cm 0.10583334cm, dimen=outer](17.05,-0.6285317)(0.45,0.7)
\psarc[linecolor=black, linewidth=0.08, dimen=outer](10.3,-0.6285317){3.3}{10.768175}{345.85382}
\psarc[linecolor=black, linewidth=0.08, dimen=outer](19.3,-0.6285317){3.3}{189.59221}{169.51924}
\rput{-5.0}(0.3317066,1.6402671){\psellipse[linecolor=black, linewidth=0.04, linestyle=dashed, dash=0.17638889cm 0.10583334cm, dimen=outer](18.95,-2.9785318)(0.75,0.45)}
\rput{30.0}(3.2594893,-10.621643){\psellipse[linecolor=black, linewidth=0.04, dimen=outer](21.45,0.7714683)(0.45,0.7)}
\rput{15.0}(0.4731698,-5.151145){\psellipse[linecolor=black, linewidth=0.04, dimen=outer](19.8,-0.77853173)(0.7,0.65)}
\rput(17.0,-0.6285317){\psscalebox{2.8}{\scriptsize $b$}}
\rput(19.8,-0.82853174){\psscalebox{2.8}{\scriptsize $y$}}
\rput(18.9,-3.0285318){\psscalebox{2.8}{\scriptsize $z$}}
\rput(21.5,0.7714683){\psscalebox{2.8}{\scriptsize $u$}}
\end{pspicture}
}
\end{center}
\caption{\label{Jarka_stale_bez_prace}
Left: a surface representing the element of
$\oCom\big(\{a,b,c,d,e\};3\big)$. Right: the glueing $\ooo{a}{b}$.}
\end{figure}
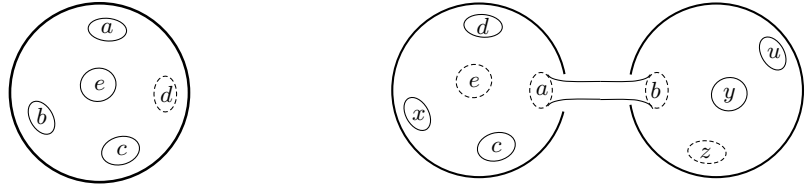
\end{example}

\begin{example}
\label{za_chvili_musim psat ten zatracenej_grant}
The $(C;G)$-component of the quantum closed modular operad $\oQC$
is, for $C \in \Fin$ and $G \in \frac12 \bbN$ 
given by
\[
\oQC(C;G) := \left\{\rule{0em}{1.4em} \surbig{g}{\zvedak C} \ | \  g\in\bbN, \
 G=2g-1+|C|/2     \right\},
\]
while the other components are empty. Since all non-empty components
of $\oQC(C;G)$ are one-point sets, the modular operad structure
is the unique one with  $\Aut(C)$ acting trivially. 
The symbol
$\sur{g}C$ 
represents the diffeomorphism class of closed oriented surfaces
of genus $g$ with holes indexed by the elements of $C$. The modular operadic
structure is in this representation given 
by connecting the holes by tubes as in 
Figure~\ref{Jarka_stale_bez_prace}.  

There is an obvious injection
\begin{align}
\label{ic}
i_C :\oCom \hookrightarrow \oQC \ , \
C \mapsto \sur{}{0}[C]
\end{align}
which identifies  $\oCom$ with the cyclic \gg\ suboperad of $\oQC$
consisting of elements with $g=0$. It is easy to verify directly
that 
\[
\oModgg(\oCom) \cong \oQC.
\]
A non-\gg\ version of this result appeared in
\cite[page~382]{markl:la}.
\end{example}

Finally, we recall a two-colored modular operad $\oQOC$  
containing $\oQO$ in the `open' color and $\oQC$ in the `closed' one.

\begin{example} 
\label{DEFQOC}
The $(O,C;G)$-component $\oQOC(O,C;G)$ 
of the quantum open-closed modular operad 
is, for $O,C \in \Fin$ and $G \in \frac12\bbN$,
defined as the set
\begin{gather}
\label{Podari_se_Jarusce_dostat_novou_praci?}
\left\{ \surbig{\sfo_1\cdots\sfo_b}{g}[\zvedak C] \ |\ b\in \bbN,
g\in\bbN,
\ G=2g+b-1+|C|/2,\ \sfo_1\cdots\sfo_b\textrm{ is a multicycle on }O \right\}.
\end{gather}
Other components of $\oQOC$ are empty.
The operadic compositions are
\[
\sur{ \sfo'_1\cdots\sfo'_{b'}}{g'}[C'] \ooo{a}{b}
\sur{ \sfo''_1\cdots\sfo''_{b''}}{g''}[C'']: =
\sur{(\sfo'_1\cdots\sfo'_{b'})  \ooo{a}{b}
(\sfo''_1\cdots\sfo''_{b''})}{g' + g''}[C'\sqcup\, C''],
\]
if $a,b$ are open inputs, and
\[
\sur{\sfo'_1\cdots\sfo'_b}{g'}[C'] \ooo{a}{b}
\sur{\sfo''_1\cdots\sfo''_{b''}}{g''}[C''] :=
\sur{\sfo'_1\cdots\sfo'_{b'}\ \sfo''_1\cdots\sfo''_{b''}}{g'+g''}[C'\sqcup\,
  C''\setminus\{a,b\}]
\]
if $a \in C'$, $b \in C''$ are closed inputs. We believe that the
meaning of the notation, analogous to the one used in previous
examples, is clear.
The contractions are, for open inputs $u$ and~$v$, given by
\[
\oxi{uv} \sur{ \sfo_1\cdots\sfo_b }g[\korekce C] :=
\begin{cases}{
\surbig{\oxi{uv}(\sfo_1\cdots\sfo_b)}g[\korekce C]} 
& \hbox{if $u$ and $v$ belong to the
same pancake, and} 
\\
\rule{0pt}{1.9em}
\surbig{\oxi{uv}(\sfo_1\cdots\sfo_b)}{g+1}[\korekce C] &
\hbox{if they belong to different pancakes.}
\end{cases}
\]
The contractions for closed inputs $u,v \in C$ are defined by
\[
\oxi{uv} \sur{ \sfo_1\cdots\sfo_b }g[\korekce C] :=
\sur{ \sfo_1\cdots\sfo_b }{g+1}[C \setminus \{u,v\}].
\]
There is an obvious action of the group $\Aut(O) \times \Aut(C)$ on
the set $\oQOC(O,C;G)$ extending the action of $\Aut(O)$ described in
Example~\ref{jsem_na_zavodech_a_prsi} and the trivial action of $\Aut(C)$. 
The symbol
\[
\surbig{\sfo_1\cdots\sfo_b}{g}[\korekce C]
\]
in~(\ref{Podari_se_Jarusce_dostat_novou_praci?})
represents the diffeomorphism class of closed oriented surfaces of genus $g$
with $b$ `open'
holes with teeth labelled by elements of $O$ as
portrayed in Figure~\ref{zubata_plocha},  and `closed' holes labelled
by elements of $C$. The operad structure in the open color is given by
connecting the teeth via ribbons, 
the structure in the closed color by
connecting the holes via tubes.
\end{example}

There are natural injections $\oQO \hookrightarrow \oQOC$ and 
$\oQC \hookrightarrow \oQOC$ given by
\[
\sur{\sfo_1\cdots\sfo_b}{g}  \mapsto  \sur{\sfo_1\cdots
\sfo_b}{g}[\korekce\emptyset]
\hbox { and }
\sur{}{g}[\korekce C]  \mapsto  
\sur{\varnothing}{g}[\korekce C]
\]
which identify $\oQO$ with the \gg\ suboperad of $\oQOC$ consisting of
elements with no
closed inputs, and $\oQC$ with the \gg\ suboperad of $\oQOC$ of
elements with $b=0$ open boundaries.

The \gg\ cyclic operad $\oCom$ from
Example~\ref{ze_ja_vul_jsem_zacal_psate_ten_grant} 
clearly coincides with the
suboperad of $\oQOC$ consisting of elements with $g = b =0$, 
i.e.\ elements of the form
\[
\sur{\varnothing}{0}[\korekce C] , \ C \in \Fin.
\] 
Such an element lives in  the operadic genus $G ={|C|}/2 -1$.
This shall explain the necessity of introducing  \gg\ cyclic operads in
this article. Notice that the stability 
assumption $|C| \geq 3$ implies that $G
\geq \frac 12$.

\begin{remark}
\label{Je vedro takze na bruslicky pujdu zitra vecer.}
One easily verifies that the Cardy condition
\begin{equation}
\label{stale_jeste_v_Creswicku}
\oxi{uv} \left( \sur{\cycs{uqa}}{0}[\emptyset] \ooo{a}{b} 
\sur{\cycs{bvr}}{0}[\emptyset] \right) = \sur{\cycs{q}}{0}[\{c\}]
\ooo{c}{d} \sur{\cycs{r}}{0}[\{d\}]
\end{equation}
visualized in Figure~\ref{Uz_mam_Martinovo_parte.} holds in $\oQOC$.
\begin{figure}
\begin{center}
\psscalebox{1.0 1.0} 
{
\begin{pspicture}(0,-0.763125)(12.37,0.763125)
\definecolor{colour0}{rgb}{0.9647059,0.9647059,0.9647059}
\psarc[linecolor=black, linewidth=0.04, dimen=outer](2.6845312,-0.035097655){0.32}{86.42367}{275.07962}
\psarc[linecolor=black, linewidth=0.04, dimen=outer](2.6645312,-0.055097654){0.66}{87.0549}{151.5804}
\psarc[linecolor=black, linewidth=0.04, dimen=outer](2.6845312,-0.035097655){0.64}{200.75534}{270.0}
\psarc[linecolor=black, linewidth=0.04, dimen=outer](1.8845313,0.40490234){0.24}{264.0578}{326.88867}
\psarc[linecolor=black, linewidth=0.04, dimen=outer](1.8845313,-0.39509764){0.24}{32.005383}{97.43141}
\psline[linecolor=black, linewidth=0.04](1.8845313,0.16490234)(1.5245312,0.16490234)(1.5245312,-0.15509766)(1.8845313,-0.15509766)(1.8845313,-0.15509766)
\psline[linecolor=black, linewidth=0.04](2.6845312,0.6049023)(2.6845312,0.28490233)(2.6845312,0.28490233)
\psline[linecolor=black, linewidth=0.04](2.6845312,-0.35509765)(2.6845312,-0.67509764)(2.6845312,-0.67509764)
\rput(.4,0){
\rput{-180}(6.4891067,-0.06661772){\psarc[linecolor=black, linewidth=0.04, dimen=outer](3.2445314,-0.03543906){0.32}{86.42367}{275.07962}}
\rput{-180}(6.529028,-0.026539017){\psarc[linecolor=black, linewidth=0.04, dimen=outer](3.2645051,-0.015412816){0.66}{87.0549}{151.5804}}
\rput{-180}(6.4891067,-0.06661772){\psarc[linecolor=black, linewidth=0.04, dimen=outer](3.2445314,-0.03543906){0.64}{200.75534}{270.0}}
\rput{-180}(8.090837,-0.9434644){\psarc[linecolor=black, linewidth=0.04, dimen=outer](4.0451083,-0.4743882){0.24}{264.0578}{326.88867}}
\rput{-180}(8.087685,0.65653217){\psarc[linecolor=black, linewidth=0.04, dimen=outer](4.044058,0.3256111){0.24}{32.005383}{97.43141}}
\psline[linecolor=black, linewidth=0.04](4.0447936,-0.23438841)(4.4047933,-0.2339157)(4.4043727,0.08608402)(4.044373,0.08561131)(4.044373,0.08561131)
\psline[linecolor=black, linewidth=0.04](3.2453718,-0.6754385)(3.2449517,-0.3554388)(3.2449517,-0.3554388)
\psline[linecolor=black,
linewidth=0.04](3.2441113,0.28456065)(3.2436912,0.6045604)(3.2436912,0.6045604)
}
\psellipse[linecolor=black, linewidth=0.04, dimen=outer](7.324531,-0.09509765)(0.28,0.66)
\psellipse[linecolor=black, linewidth=0.04, dimen=outer](8.604531,-0.09509765)(0.28,0.66)
\psellipse[linecolor=black, linewidth=0.04, dimen=outer](9.804531,-0.09509765)(0.28,0.66)
\psellipse[linecolor=black, linewidth=0.04, dimen=outer](10.964531,-0.09509765)(0.28,0.66)
\psline[linecolor=black, linewidth=0.04](7.324531,0.56490237)(8.604531,0.56490237)
\psline[linecolor=black, linewidth=0.04](7.364531,-0.71509767)(8.604531,-0.71509767)
\psline[linecolor=black, linewidth=0.04](9.804531,-0.71509767)(10.924531,-0.71509767)
\psline[linecolor=black, linewidth=0.04](9.844531,0.56490237)(10.964531,0.56490237)
\psarc[linecolor=black, linewidth=0.04, dimen=outer](6.884531,0.32490236){0.24}{264.0578}{326.88867}
\psline[linecolor=black, linewidth=0.04](6.884531,0.084902346)(6.5245314,0.084902346)(6.5245314,-0.23509766)(6.884531,-0.23509766)(6.884531,-0.23509766)
\psarc[linecolor=black, linewidth=0.04, dimen=outer](6.884531,-0.47509766){0.24}{32.005383}{97.43141}
\psline[linecolor=colour0, linewidth=0.06](8.4045315,0.40490234)(8.524531,0.36490235)
\psline[linecolor=colour0, linewidth=0.06](8.284532,0.20490235)(8.444531,0.20490235)
\psline[linecolor=colour0, linewidth=0.06](8.284532,0.0049023437)(8.4045315,0.0049023437)
\psline[linecolor=colour0, linewidth=0.06](8.284532,-0.19509766)(8.4045315,-0.19509766)
\psline[linecolor=colour0, linewidth=0.06](8.284532,-0.43509766)(8.444531,-0.39509764)
\psline[linecolor=colour0, linewidth=0.06](8.3645315,-0.55509764)(8.324532,-0.5150977)
\psline[linecolor=colour0, linewidth=0.06](8.4045315,-0.59509766)(8.524531,-0.5150977)
\psline[linecolor=colour0, linewidth=0.06](7.0845313,0.124902345)(7.0845313,-0.07509766)(7.0845313,-0.27509767)(7.0445313,-0.035097655)(7.0445313,0.084902346)(7.1245313,0.16490234)
\psline[linecolor=black, linewidth=0.04](11.044887,-0.2743884)(11.484793,-0.2739157)(11.48428,0.046084024)(11.0443735,0.04561131)(11.0443735,0.04561131)
\psarc[linecolor=black, linewidth=0.04, dimen=outer](11.064531,0.30490234){0.26}{139.94589}{275.07962}
\psarc[linecolor=black, linewidth=0.04, dimen=outer](11.104531,-0.53509766){0.26}{86.42367}{232.64528}
\psline[linecolor=colour0, linewidth=0.06](10.764531,0.28490233)(10.724531,0.124902345)(10.724531,0.0049023437)(10.684531,0.044902343)(10.724531,-0.19509766)(10.684531,-0.15509766)(10.724531,-0.27509767)(10.724531,-0.47509766)(10.764531,-0.39509764)(10.764531,-0.55509764)(10.804531,-0.67509764)(10.724531,-0.35509765)
\psline[linecolor=colour0, linewidth=0.06](10.724531,0.48490235)(10.564531,0.32490236)
\psline[linecolor=colour0, linewidth=0.06](10.764531,0.48490235)(10.724531,0.124902345)(10.724531,0.124902345)
\psline[linecolor=colour0, linewidth=0.06](11.244532,0.24490234)(11.284532,-0.035097655)(11.244532,-0.39509764)(11.244532,-0.39509764)
\psline[linecolor=colour0, linewidth=0.06](11.164532,0.124902345)(11.164532,-0.31509766)(11.164532,-0.31509766)
\psline[linecolor=colour0, linewidth=0.06](10.764531,-0.55509764)(10.924531,-0.55509764)
\psline[linecolor=colour0, linewidth=0.06](10.764531,-0.35509765)(10.924531,-0.43509766)(10.924531,-0.43509766)
\psline[linecolor=colour0, linewidth=0.06](10.964531,-0.23509766)(11.004531,-0.35509765)(10.924531,-0.07509766)
\psline[linecolor=colour0, linewidth=0.06](10.764531,0.44490233)(10.884531,0.40490234)
\psline[linecolor=colour0, linewidth=0.06](10.724531,0.24490234)(10.924531,0.28490233)(10.924531,0.28490233)
\psline[linecolor=colour0, linewidth=0.06](10.844531,0.044902343)(10.844531,-0.035097655)
\psline[linecolor=colour0, linewidth=0.06](10.884531,0.084902346)(11.004531,0.20490235)(11.004531,0.20490235)
\rput(0.7,-0.1){$\oxi{uv}$}
\rput(9.2,-0.1){$\ooo cd$}
\rput(3.15,-0.1){$\ooo a{\,b}$}
\rput(8.844531,0.68490237){\scriptsize $c$}
\rput(9.564531,0.68490237){\scriptsize $d$}
\rput(1.3,-.1){\scriptsize $q$}
\rput(6.3,-.1){\scriptsize $q$}
\rput(5,-0.1){\scriptsize $r$}
\rput(11.7,-0.1){\scriptsize $r$}
\rput(2.92,0.44){\scriptsize $a$}
\rput(2.92,-0.55){\scriptsize $u$}
\rput(3.4,-0.52){\scriptsize $b$}
\rput(3.4,0.44){\scriptsize $v$}
\rput(1.1645312,0.044902343){$\left(\rule{0em}{1.9em} \right.$}
\rput(5.164531,-0.11509766){$\left. \rule{0em}{1.9em}\right)$}
\rput(5.8,-0.1){$=$}
\end{pspicture}
}
\end{center}
\caption{\label{Uz_mam_Martinovo_parte.} A pictorial form of the
  Cardy condition. A better picture can be
found in \cite[(3.44)]{lauda}.}
\end{figure}
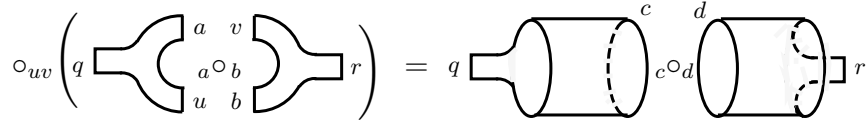
Notice that the Cardy condition involves both `open' and `closed'
structure operations.
Let us explain how it forces the operadic genus of $\oQOC$ to be
half-integral.
The terms
\[
\sur{\cycs{uqa}}{0}[\emptyset] \hbox { and }
\sur{\cycs{bvr}}{0}[\emptyset]
\] 
in the left hand side
are related by  a bijection, hence they have the same
operadic genus, say $G'$.
The  terms
\[
 \sur{\cycs{q}}{0}[\{c\}] \hbox { and }
\sur{\cycs{r}}{0}[\{d\}]
\] 
in the right hand side also have
the same operadic genus, say $G''$, by the same reason.
Since $\oxi{uv}$ is required to raise the operadic genus by $1$, 
we have
\[
2G'+1 = 2G''
\]
which shows that $G'$ and $G''$ cannot be simultaneously integral.
\end{remark}

The modular operad $\oQOC$ contains the following important
\gg\ cyclic suboperad.

\begin{example}
\label{na_obed_k_Japoncum_s_Jarcou}
The open-closed cyclic operad $\oOC$ 
is the \gg\ cyclic suboperad of $\oQOC$ consisting of diffeomorphism
classes of all surfaces of geometric
genus $0$, i.e. 
\begin{equation}
\label{Ars_Nova}
\oOC(O,C;G) := \left\{ \sur{\sfo_1\cdots\sfo_b}{g}[\korekce C] \in 
\oQOC(O,C;G)  \ | \ g=0
\right\}.
\end{equation}
\end{example}
The structures encountered so far and their inclusions are organized in
the following diagram:
\begin{equation}
\label{d}
\xymatrix@R1.2em{\oCom \ar@{^{(}->}[r]  \ar@{^{(}->}[d]   & \oQC  \ar@{^{(}->}[dr] &
\\
\oOC \ar@{^{(}->}[rr]  && \ \oQOC\, .
\\
\oAss \ar@{^{(}->}[r]  \ar@{^{(}->}[u]   & \oQO \ar@{^{(}->}[ur]  &
}
\end{equation}

\subsection{Non-$\Sigma$ versions.}
It turns out that the operad $\oQOC$ from Example~\ref{DEFQOC} 
is generated, in the sense
specified below, by a simpler object $\oQnsOC$ that behaves in the
open color as a \ns (non-symmetric) modular operad. 
Let us briefly recall what these objects are.

\Ns cyclic operads were introduced 
in \cite[p.\ 257]{markl-shnider-stasheff:book}. 
The \gg\ extension of their definition is so obvious that we can safely
leave the details to the reader. We denote their category $\nsCycOp$
resp.\  $\nsCycOpgg$ in the genus-graded case.\footnote{As usual, we
  indicate the \ns versions by underlining.} The left adjoint 
\[
\Sym: \nsCycOpgg \to \CycOpgg
\]
to the forgetful functor $ \Des : \CycOpgg \to \nsCycOpgg $ is
called the {\em symmetrization\/}.

\begin{example}
\label{prvni_holubinky}
The \gg\ cyclic operad 
$\oAss$ recalled in Example~\ref{poleti_se_zitra?} is 
the symmetrization of the \ns cyclic \gg\ operad $\nsoAss$
whose $(\sfo,G)$-component, for a cycle $\sfo$ on a finite  set
$O$ and $G\in\sfS$,  
is defined by
\[
\nsoAss(\sfo;G) := 
\begin{cases} \left\{ \sfo \right\}, & \hbox { if } 
  G=0, \hbox { and} 
\\ 
\emptyset, & \textrm{otherwise.} 
\end{cases}
\]
Since $\nsoAss(\sfo;G)$ is either empty or a one-point set, the
operadic composition is
defined in the only possible manner, and  
the subgroup $\Aut(\sfo) \subset \Aut(O)$ of automorphisms  preserving the
cyclic order acts trivially.
The isomorphism
$
\oAss \cong \Sym(\nsoAss)
$
is easy to check.
\end{example}

Since \ns modular operads were introduced only very recently
\cite{nsmod}, we recall their basic features in more detail.
While the components of \ns cyclic \gg\ 
operads are indexed by cycles  and genera, components of modular \ns
operads are indexed by {\em multi\/}cycles and genera.
We start by recalling a suitable groupoid of multicycles.

\begin{definition}
  A isomorphism $\sfO' =\sfo'_1\cdots\sfo'_{b'} \longrightarrow \sfO'' = 
\sfo''_1\cdots\sfo''_{b''}$
  of multicycles consists of a bijection
  $u:\{1,\ldots,b'\}\to \{1,\ldots,b''\}$ and of a cyclic order-preserving
  bijection $\sigma_i:\sfo'_i\to\sfo''_{u(i)}$ specified for each
  $1\leq i\leq b'$.  The groupoid of all multicycles and their
  isomorphisms will be denoted by $\MultCyc$.
\end{definition}

\begin{definition}
A {\em \ns modular operad\/} is a functor 
\[
\nsM:\MultCyc\times\textstyle{\frac{1}{2}}\bbN \longrightarrow \Set,
\]
where $\frac{1}{2}\bbN$\footnote{In~\cite{nsmod}, $\nsM$ was a
  functor $\MultCyc\times \bbN \longrightarrow \Set$, but the difference is
  irrelevant.} 
is viewed as a discrete category, together
with operadic compositions
\begin{equation}
\label{eq:7}
\ooo{u}{v} : \nsM(\sfO';G') \otimes \nsM(\sfO'';G'')  \longrightarrow \nsM( \sfO'\ooo{u}{v}\sfO'';G'+G'')
\end{equation}
defined for arbitrary disjoint multicycles $\sfO'$ and $\sfO''$ with
elements $u\in O'$, $v\in O''$ of their underlying sets, and contractions
\begin{equation}
\label{eq:8}
\oxi{uv} = \oxi{vu} : \nsM(\sfO;G) \longrightarrow \nsM(\oxi{uv}\sfO;G+1)
\end{equation}
given for any multicycle $\sfO$ and distinct elements $u,v \in O$ 
of its underlying set.
These data are required to satisfy the 
expected associativity and equivariance axioms
for which we refer the reader to~\cite{nsmod}.
\end{definition}

We will denote the category of  \ns modular operads  by
$\nsModOp$. As for cyclic operads, we have the forgetful functor $\Des
: \ModOp \to \nsModOp$ and its left adjoint $\Sym: \nsModOp \to \ModOp$.

\begin{example}
\label{ani_v_patek_se_neleti}
The quantum open modular operad $\oQO$ recalled in
Example~\ref{jsem_na_zavodech_a_prsi} is the symmetrization  
of the \ns modular operad $\oQnsO \in \nsModOp$ whose
$(\sfO;G)$-component is,
for \hbox{$\sfO =\sfo_1\cdots\sfo_b \in \MultCyc$} and $G \in \bbN$
defined as
\begin{equation*}
\oQnsO(\sfO;G) := \left\{\rule{0em}{1.1em} \sur{\sfO}{g} \
  |\
  g\in\bbN, G=2g\!+\! b\!-\! 1 \right\}
\end{equation*}
while the other components are empty. 
The structure operations are given by formulas~(\ref{a})
and~(\ref{b}).
There is a natural inclusion $i_O : \nsoAss \hookrightarrow \oQnsO$
given by 
\begin{equation}
  \label{eq:1}
i_O(\sfo) : =    \sur{\sfo}{\zvedak 0}.
\end{equation}

The symbol $\sur{\sfO}{g}$ represents, for $\sfO
=\sfo_1\cdots\sfo_b$, the diffeomorphism class of surfaces $\Sigma$ as in
Figure~\ref{zubata_plocha} with $b$ boundary components and teeth indexed by the
underlying set of $\sfO$ such that 
each $\sfo_i$, $1\leq i \leq b$, corresponds to a
specific boundary component of $\Sigma$, 
with its total cyclic order induced by
the orientation of $\Sigma$.   
\end{example}

The forgetful functor $\nsModOp\to\nsCycOpgg$ 
has the left adjoint  
\[
\nsoMod : \nsCycOpgg \to \nsModOp.
\]
The second author proved in \cite{nsmod}
that 
\[
\nsoMod(\nsoAss)\cong\oQnsO.
\]

\subsection{Premodular operads}
\label{premodular}
Consider the subcollection  $\oOMB$  of the \ns version of the
quantum open operad $\oQnsO$ of Example~\ref{ani_v_patek_se_neleti}
consisting of symbols $\sur{\sfO}{0}$ representing surfaces with
geometric genus $g=0$.  Inspecting formulas~(\ref{a}) and~(\ref{b})
defining the operadic structure of $\oQnsO$ we see that, while $\oOMB$ is closed
under the compositions $\ooo ab$, i.e.\ it is a \gg\ \ns cyclic
suboperad of $\oQnsO$, $\oOMB(\sfO)$ is closed under the contraction 
$\oxi{uv}$ only if $u$ and
$v$ belong to the same cycle of the multicycle
$\sfO=\sfo_1\cdots\sfo_b$.
The collection  $\oOMB$ is an example of the following structure.

\begin{definition}
\label{zitra_do_Brna}
A \emph{premodular operad} $\oP$ is a functor
\[
\oP : \MultCyc \times \textstyle{\frac{1}{2}}\bbN \longrightarrow \Set
\]
together with operadic compositions
\[
\ooo{u}{v} : \oP(\sfO';G') \otimes \oP(\sfO'';G'')  \longrightarrow \oP(
\sfO'\ooo{u}{v}\sfO'';G'+G'')
\]
defined for arbitrary disjoint multicycles $\sfO'$ and $\sfO''$ 
and 
elements $u\in O'$, $v\in O''$ of their underlying sets, and contractions
\[
\oxi{uv} = \oxi{vu} : \oP(\sfO;G) \longrightarrow \oP(\oxi{uv}\sfO;G+1)
\]
defined only for  distinct elements $u,v$ 
belonging to the {\em same\/} underlying set of a cycle of the multicycle $\sfO$.
These data are required to satisfy the axioms analogous 
to those for \ns modular operads.
Premodular operads and their morphisms form the category 
$\preModOp$.
\end{definition}

Premodular operads are thus specific partial \ns modular operads.
Notice that there is no analogue of premodular operads for the
standard ($\Sigma$-) modular operads.

\subsection{Hybrids}
\label{SSECOpVariants}
We will also need  operad-like structures with two colors,
`open' and `closed,' which behave differently in each of its
colors. Namely, we consider structures that are 
\begin{itemize}
\item [(i)]
\gg\  \ns cyclic  operads in the open color and ordinary \gg\ cyclic
  operads in the closed one, 
\item [(ii)] 
\ns modular operads in the open color and ordinary modular
  operads in the closed one, and
\item [(iii)]
premodular operads
in the open color and ordinary modular
  operads in the closed~one.
\end{itemize}

\begin{definition}
We call a structure of type (i),~(ii) resp.~(iii) 
a {\em cyclic, modular\/} 
resp.~{\em premodular hybrid\/}. Their categories will be denoted
$\CycHyb$, $\ModHyb$ and $\PreHyb$, respectively.
\end{definition}

Modular hybrids should be compared with another
formalization of the combinatorial structure of surfaces with open and
closed boundaries -- c/o-structures 
of~\cite[Appendix~A]{kaufmann-penner:NP06}.  

\begin{example}
\label{DG_307_v_Akropoli}
Let $\ColHyb$ denote the category of hybrid collections which are, by
definition, functors
\[
E:\MultCyc\times \Fin \times \textstyle{\frac{1}{2}}\bbN \longrightarrow \Set
\]
where $\frac{1}{2}\bbN$ is viewed as a discrete category. Informally,
objects of $\ColHyb$ are what remains from cyclic (or modular) hybrid
when one forgets all $\circ$-operations (and contractions). We
therefore have a commuting diagram
\[
\xymatrix{ & \CycHyb  \ar[rd] & 
\\ 
\ModHyb\ar[rr]\ar[ru] && \ColHyb
}
\] 
of forgetful functors and the associated commutative diagram
\[
\xymatrix{ & \CycHyb  \ar[ld]_{\ModHybFun(-)} & 
\\ 
\ModHyb  && \ColHyb \ar[ll]_{\Fr[\underline{m}od]{-}}
\ar[lu]_{\Fr[\underline{c}yc]{-}}
}
\]  
of their left adjoints. The functors $\Fr[\underline{c}yc]{-}$,
$\Fr[\underline{m}od]{-}$ and $\ModHybFun(-)$ are the free cyclic, free modular
and modular hybrid \envelope\ functors, respectively. 
\end{example}

\begin{example}
\label{dnes_vecer_s_Jarkou_a_Hankou_k_Pakousum}
Let $\bfk$ be a field, not necessarily of characteristic zero. 
Each (graded) 
$\bfk$-vector space $A$ equipped with a 
non-degenerate symmetric bilinear form $\beta_A$ has
its modular endomorphism operad 
$\End_A$, see e.g.~\cite[Example~II.5.43]{markl-shnider-stasheff:book},
\cite[Example~52]{markl:handbook} or the original 
source~\cite[(2.25)]{getzler-kapranov:CompM98}. Given another
vector space $B$ with a non-degenerate symmetric bilinear form $\beta_B$, one
can, in the obvious manner, extend the construction of $\End_A$ and 
create a~modular endomorphism hybrid $\End_{A,B}$ with components
\begin{equation}
\label{dnes_vecer_s_Nunykem_do_Matu}
\textstyle
\End_{A,B}(\sfO,C;G) := \Lin\big(\bigotimes_{o\, \in O}A_o \ot \bigotimes_{c\, \in
  C} B_c\, ,\bfk \big).
\end{equation}
In the above display, $O$ is the underlying set 
of the multicycle $\sfO$, $A_o$
resp.~$B_c$ are the identical copies of the space $A$ resp.~$B$, and
$\bigotimes$  is the unordered product in the symmetric monoidal
category of graded vector 
spaces~\cite[Definition~II.1.58]{markl-shnider-stasheff:book}. 
Due to the presence of non-degenerate bilinear forms, both $A$ and $B$
are finite-dimensional, canonically isomorphic to their duals via
raising and lowering indexes. This allows for several formally
different but equivalent definitions of the endomorphism modular
hybrid. For instance,~(\ref{dnes_vecer_s_Nunykem_do_Matu}) can be
replaced by
\[
\textstyle
\End_{A,B}(\sfO,C;G) := \bigotimes_{o\, \in O}A_o \ot \bigotimes_{c\, \in
  C} B_c
\] 
which is more in the spirit of~\cite{getzler-kapranov:CompM98}.

Having endomorphism hybrids, one can speak about algebras; 
an {\em algebra\/} for a
modular hybrid is, by definition, a morphism $\alpha : \oH \to
\End_{A,B}$. Since $\End_{A,B}$ is at the same time also a cyclic
hybrid, we define in the same way algebras for cyclic hybrids.
\end{example}

\begin{example}
\label{hm}
The operad $\oQOC$ of Example~\ref{DEFQOC} is the symmetrization, in the open
color, of the modular hybrid $\oQnsOC$ whose $(\sfO,C;G)$-component 
 is, for $\sfO=\sfo_1\cdots\sfo_b \in \MultCyc$, $C
\in \Fin$ and $G \in \frac12 \bbN$ defined as the set of symbols
\begin{gather}
\label{dnes_s_Jarcou_na_vyhlidku}
\oQnsOC(\sfO,C;G) : =\left\{ \surbig{\sfO}{g}[\rule{0pt}{.55em} C] \ |\
g\in\bbN,
\ G=2g+b-1+|C|/2 \right\}.
\end{gather}
\end{example}

\begin{example}
\label{h1}
There two-colored \gg\ cyclic operad $\oOC$ from
Example~\ref{na_obed_k_Japoncum_s_Jarcou} is the symmetrization, in the open
color, of the cyclic hybrid $\onsOC$ defined as the subcollection of
$\oQnsOC$ consisting of symbols~(\ref{dnes_s_Jarcou_na_vyhlidku}) with
$g=0$.
The hybrid $\onsOC$ clearly contains both $\Com$ and $\nsoAss$ as
graded cyclic (resp.~\ns cyclic) suboperads.
\end{example}

\begin{example}
\label{h2}
The cyclic hybrid $\onsOC$ from Example~\ref{h1} is obviously 
stable under the contractions  $\oxi{uv}$ for $u$ and $v$ belonging to
the same pancake. 
It therefore forms a premodular hybrid which we denote by 
$\onsOCpre$.
\end{example}

One clearly has the following non-$\Sigma$ version of the diagram
in~(\ref{d}) composed of ordinary and  non-$\Sigma$ operads, and
cyclic and modular hybrids:
\begin{equation}
\label{dst}
\xymatrix@R1.2em{\oCom \ar@{^{(}->}[r]  \ar@{^{(}->}[d]   & \oQC  \ar@{^{(}->}[dr] &
\\
\onsOC \ar@{^{(}->}[rr]  && \ \oQnsOC\, .
\\
\nsoAss \ar@{^{(}->}[r]  \ar@{^{(}->}[u]   & \oQnsO \ar@{^{(}->}[ur]  &
}
\end{equation}

\subsection{Stable versions} 
Let us slightly 
generalize the stability condition for modular operads
introduced in~\cite{getzler-kapranov:CompM98}.

\begin{definition}
\label{s_Jarkou_k_Zabackovi}
The \emph{stable part} of a cyclic or modular operad $\oP$ is the
collection 
defined as
\[
\oP_\St(S;G) := 
\oP(S;G)
\]
if the {\em stability}
\begin{equation}
\label{st}
2(G-1)+|S|>0
\end{equation}
is satisfied, while $\oP_\St(S;G) := 
\emptyset$ for the remaining $(S;G)$.
The operad $\oP$ is  \emph{stable} if $\oP=\oP_\St$.
\end{definition}

Inequality~(\ref{st}) is equivalent to the absence of continuous
families of automorphisms of a Riemann surface of genus $G$ with $|C|$
distinct marked points; whence its name.
Notice, that for  \gg\ cyclic operads concentrated in 
genus~$0$,~(\ref{st}) says that
$|S| \geq 3$.

Definition~\ref{s_Jarkou_k_Zabackovi} is easily modified to the
non-$\Sigma$ cases, while for hybrids we replace~(\ref{st}) by 
\begin{equation}
\label{Vcera_jsem_hodinu_bezel.}
2(G-1)+|O| +
|C|>0.
\end{equation}
The statements in the  following lemma can be verified directly.

\begin{lemma}
\label{Prijedu_zase_pristi_rok?}
Inequalities~(\ref{st}) and~(\ref{Vcera_jsem_hodinu_bezel.}) are
preserved by the $\circ$-operations and contractions. If a contraction
of $x$ belongs
to the stable part of a (non-$\Sigma$) modular operad or of a modular
hybrid, then  $x$ belongs to the stable part as well.
\end{lemma}

Thus the stable part of a (\ns) cyclic, (\ns) modular, 
or premodular operad, or of a 
hybrid, is the structure of the same
type, with the  operations given by the restrictions of the original ones.

\begin{example}
\label{poleti_se_zitra_stable}
The stable version  $\oAssst$ of the associative cyclic operad $\oAss$
of Example~\ref{poleti_se_zitra?} is obtained by requiring that
$\oAssst(O) = \emptyset$ if $|O| \le 2$, i.e.\
\[
 \oAssst(O) :=  
\begin{cases}
\oAss(O) & \hbox { if $|O| \geq 3$ and}
\\
\emptyset & \hbox { otherwise. }
\end{cases}
\]
The operad $\oAssst$ 
governs associative algebras with a non-degenerate invariant
bilinear 
form.\footnote{I.e.~non-commutative Frobenius algebras.}
\end{example}

\begin{example}
\label{jsem_na_zavodech_a_prsi_stable}
The stable version of the quantum open modular operad $\oQO$ from
Example~\ref{jsem_na_zavodech_a_prsi} is
defined by
\[
 \oQOst(O;G) :=  
\begin{cases}
\oQO(O;G) & \hbox { if $2(G-1)+|O|>0$ and}
\\
\emptyset & \hbox { otherwise. }
\end{cases}
\]
The operadic structure is defined by the same formulas as for $\oQO$.
\end{example}

\begin{example}
\label{ze_ja_vul_jsem_zacal_psate_ten_grant_stable}
The stable version of the \gg\ cyclic commutative operad $\oCom$ from
Example~\ref{ze_ja_vul_jsem_zacal_psate_ten_grant} is
defined by
\[
\oComst(C;G) :=  
\begin{cases}
\oCom(C;G) & \hbox { if $|C|\geq 3$ and}
\\
\emptyset & \hbox { otherwise. }
\end{cases}
\]
It is easy to check that, since $G$ is determined by $|C|$
via~(\ref{dam_si_s_Jarcou_pivo_u_Zabacka}), the 
condition $|C|\geq 3$ is equivalent
to the stability  $2(G-1)+|C|>0$ as expected. Algebras over
$\oComst$ are commutative Frobenius algebras.
\end{example}

\begin{example}
\label{za_chvili_musim psat ten zatracenej_grant_stable}
The stable variant $\oQCst$ of the quantum closed operad $\oQC$
recalled  in
Example~\ref{za_chvili_musim psat ten zatracenej_grant} is
defined by
\[
\oQCst(C;G) :=  
\begin{cases}
\oQC(C;G) & \hbox { if $2(G-1)+|C|>0$ and}
\\
\emptyset & \hbox { otherwise. }
\end{cases}
\]   
\end{example}

\begin{example} 
\label{DEFQOC_stable}
The stable version $\oQOCst$ of the quantum open-closed operad $\oQOC$ from
Example~\ref{DEFQOC} is defined by
\[
\oQOCst(O,C;G) :=  
\begin{cases}
\oQOC(O,C;G) & \hbox { if $2(G-1)+|O| + |C|>0$ and}
\\
\emptyset & \hbox { otherwise. }
\end{cases}
\] 
The stability condition for the symbol
in~(\ref{Podari_se_Jarusce_dostat_novou_praci?}) in $\oQOC(O,C;G)$ 
expressed in terms of
its geometric genus and number of boundaries reads
\[
 4g+2b+2|C| +|O| >4.
\]
The stable subhybrid  $\oQnsOCst$ of the modular hybrid  $\oQnsOC$
from Example~\ref{hm} is defined similarly.
\end{example}

\begin{example}
\label{Dnes_naposledy_v_Myluzach_na_varhanky.}
The stable version $\oOCst$ of the open-closed cyclic operad $\oOC$ from
Example~\ref{na_obed_k_Japoncum_s_Jarcou} is the \gg\ cyclic
suboperad of $\oQOCst$ consisting of all symbols as
in~(\ref{Ars_Nova}) satisfying
\[
2b+2|C| +|O| >4.
\]
Likewise,  the stable cyclic hybrid $\onsOCst$ consists of
all symbols~(\ref{dnes_s_Jarcou_na_vyhlidku}) with $g=0$ satisfying
the same inequality. 
There are six unstable elements of $\onsOC$, i.e.\
elements of  $\onsOC \setminus \onsOCst$, namely
\begin{gather}  
\label{REMListOfUnstables}
\R1:= \sur{\cycs{}}{0}[\emptyset]\!,\
\R2:= \sur{\cycs{p}}{0}[\emptyset]\!,\
\R3 :=  \sur{\cycs{pq}}{0}[\emptyset]\!,\
\R4 := \sur{\cycs{}\cycs{}}{0}[\emptyset]\!,\
\R5 := \sur{\cycs{}}{0}[\stt d]\!,\
\R6 := \sur{\varnothing}{0}[\stt{d,e}].
\end{gather}
Their operadic genera, number of boundaries, 
and the cardinalities of $O$ and
$C$ are listed in Table~\ref{Avignon}.
\begin{table}
\def\arraystretch{1.3}
\begin{center}
  \begin{tabular}{|c|c|c|c|c|c|c|c|c|c|} 
\hline 
\multicolumn{1}{|c|}{{element}} 
& \multicolumn{1}{c|}{\textbf{\hskip .5em $b$ \hskip .5em}}
& \multicolumn{1}{c|}{\textbf{\hskip .3em $G$ \hskip .3em}} &
\multicolumn{1}{c|}{\textbf{$|O|$}}&\multicolumn{1}{c|}{\textbf{$|C|$}}&
\multicolumn{1}{|c|}{{element}} & 
\multicolumn{1}{c|}{\textbf{\hskip .5em $b$ \hskip .5em}}
& \multicolumn{1}{c|}{\textbf{\hskip .3em $G$ \hskip .3em}} &
\multicolumn{1}{c|}{\textbf{$|O|$}}&\multicolumn{1}{c|}{\textbf{$|C|$}}
\\
\hline \hline 
$\R1$ & 1 & 0 & 0 &0 &$\R2$ & 1 & 0 & 1 &0
\\
\hline 
$\R3$ & 1 & 0 & 2 &0 &$\R4$ & 2 & 1 & 0 &0
\\
\hline 
$\R5$ & 1 & $\frac12$ & 0 &1 &$\R6$ & 0 & 0 & 0 &2
\\
\hline
\end{tabular}
\\
\rule{0pt}{.8em}
\end{center}
\caption{Unstable elements in $\onsOC$.\label{Avignon}}
\end{table}
The symbol $\sur{\varnothing}{0}[\stt d]$ is excluded since its operadic
genus equals  $G=-1/2$.
\end{example}

The stable version of the diagram in~(\ref{d}) can be obtained by
decorating everything by the \hbox{subscript ``st.''}

\begin{example}
\label{KP}
We will consider also the
Kaufmann-Penner cyclic subhybrid $\onsOCKP$ 
of the stable cyclic hybrid  $\onsOCst$ obtained by discarding the
following types of elements of $\onsOCst$:
\[
\hbox {type (i): }
\sur{\cycs{}\cdots\cycs{}}{0}[\emptyset], \  b\geq 3;\
\hbox {type (ii): }
\sur{\cycs{}\cdots\cycs{} \, \sfo\,}{0}[\emptyset],\ b\geq 2, \ |\sfo|
\geq 1; \  \hbox {type (iii): }
\sur{\cycs{}\cdots\cycs{}}{0}[\stt {d}],\ b\geq 2.
\]
The check that  $\onsOCKP$ is  closed under $\circ$-operations
is routine.
In the proof of Theorem~\ref{THMPresentaionOfOC} we will need an explicit list of elements
\begin{equation}
\label{zitra_na_houby}
\sur{\sfo_1 \, \cdots \sfo_b}{0}[\zvedak C] \in \onsOC
\end{equation} 
that do belong to $\onsOCKP$. We distinguish three cases depending on
the number of boundaries.
\hfill\break 
\noindent 
-- If $b \geq 2$, then~(\ref{zitra_na_houby})   belongs
  to  $\onsOCKP$ if and only if \hfill\break 
\rule{2em}{0em} -- $|C| \geq 2$, or \hfill\break 
\rule{2em}{0em} -- $|C| =1$ and at least one of $\sfo_1,\ldots,\sfo_b$ is
  not empty, or\hfill\break 
\rule{2em}{0em} -- $C = \emptyset$ and at least two of $\sfo_1,\ldots,\sfo_b$ are
  not empty,
\hfill\break 
\noindent 
-- if $b = 1$, then~(\ref{zitra_na_houby})   belongs
  to  $\onsOCKP$ if and only if $|O| + 2|C| > 2$, and
\hfill\break 
\noindent 
-- if $b = 0$, then~(\ref{zitra_na_houby})   belongs
  to  $\onsOCKP$ if and only if $|C| \geq 3$.
\end{example}

\begin{example}
\label{prvni_holubinky_stable}
The stable version $\nsoAssst$ of the \ns associative cyclic operad
$\nsoAss$ from Example~\ref{prvni_holubinky} is defined by
\[
\nsoAssst(\sfo;G) :=  
\begin{cases}
\nsoAss(\sfo;G) & \hbox {if the cardinality of the underlying set of
  $\sfo$ is $\geq 3$, and}
\\
\emptyset & \hbox { otherwise. }
\end{cases}
\]  
\end{example}

\begin{example}
\label{ani_v_patek_se_neleti_stable}
The stable version $\oQnsOst$ of the \ns associative modular operad
$\oQnsO$ from Example~\ref{ani_v_patek_se_neleti} is defined by
\[
\oQnsOst(\sfO;G) :=  
\begin{cases}
\oQnsO(\sfO;G) & \hbox {if $2(G-1)+|O|>0$, and}
\\
\emptyset & \hbox { otherwise. }
\end{cases}
\]  
\end{example}

\section{Modular \envelope\  and Cardy
  condition} 
\label{SSECModOCus}

This section forms the core of this
article. Proposition~\ref{THMModOCus} below explicitly describes the modular
\envelope\ of the cyclic hybrid $\onsOC$ and identifies it with the
set of diffeomorphism classes of suitable Riemann surfaces with embedded
circles. Its proof occupies nearly seven
pages. Theorem~\ref{PROCUnivProp} is the central result
of this paper.  It describes $\oQnsOC$ as the quotient of
$\ModHybFun(\onsOC)$ by the Cardy
condition. Theorem~\ref{THMPositiveModEnvOC} in the last subsection
characterizes $\oQnsOC$ as the modular \envelope\ of the premodular hybrid
$\onsOCpre$.

\subsection{Modular \envelope\ of cyclic hybrids}
In Example~\ref{DG_307_v_Akropoli} we introduced the
modular  hybrid \envelope\ functor  $\ModHybFun :  \CycHyb \to \ModHyb$ as the
left adjoint to the forgetful functor $\ModHyb \to \CycHyb$. It 
is clearly a combination of
the non-$\Sigma$-modular \envelope\ functor
$\nsoMod$~\cite[Section~5]{nsmod} 
in the open
color and the ordinary modular \envelope\
$\oMod*$~\cite[page~382]{markl:la} 
in the closed
one, as indicated by underlying only the first letter of ``Mod.'' 
The aim of this subsection is to describe   its value on the
open-closed cyclic hybrid $\onsOC$ from Example~\ref{h1}. This auxiliary technical
result is the main step in proving Theorem~\ref{PROCUnivProp}.

We will need the following terminology.
Let $\sfO = \sfo_1 \cdots \sfo_b$, $b \geq 1$, be a nontrivial
multicycle. 
A {\em decomposition\/} of $\sfO$ is a disjoint decomposition 
\[
\{\rada 1b\} = B_1 \cup \cdots \cup B_a
\]
of the set indexing the cycles of $\sfO$ into nonempty subsets. 
When necessary, we will
identify it with a
choice $\sfO_1,\ldots \sfO_a$ of multicycles 
$\sfO_i := \{\sfo_j \ | \ j\in B_i\}$, $1 \leq i \leq a$. In this
situation we denote $b_i := |B_i|$; clearly $b = b_1 + \cdots + b_a$.

\begin{proposition} 
\label{THMModOCus}
The component  $\ModHybFun(\onsOC)(\sfO,C;G)$ of the
modular  hybrid \envelope\ of~$\onsOC$ is, for $\sfO \in \MultCyc$, $C\in \Fin$ and $G\in \frac12
\bbN$,  the set of all symbols
\begin{equation}
\label{v_pondeli_do_Leicesteru}
\sur{V(\sfO_1;G_1) \cdots V(\sfO_a;G_a)}{g}[\rule{0pt}{.6em}  C], 
\hbox { abbreviated as }
\sur{ V_1\cdots V_a }{ g}[\rule{0pt}{.6em}C],
\end{equation} 
\begin{tabular}{lrl}
where  &(i)&\hskip -.5em  
  $g \in \bbN$,
\\
&(ii)&\hskip -.5em  
$\sfO_1,\ldots \sfO_a$  is a decomposition of $\sfO$,
\\
&(iii)&\hskip -.5em  $G_i \in
\bbN$, $1\leq i \leq a$, are such that  $\oQnsO(\sfO_i;G_i)$ is
non-empty,
and
\\
&(iv)&\hskip -.5em 
 $V(\sfO_i;G_i)$ is the unique nontrivial element of
 $\oQnsO(\sfO_i;G_i)$, $1\leq i \leq a$.
\end{tabular} \hfill\break 
We moreover assume that
\[
\textstyle
G=\sum_{i=1}^a G_i + 2g+a-1+|C|/2.
\]
For $g \in \bbN$ satisfying $G=2g-1+|C|/2$
we complete the definition by
\[
\ModHybFun(\onsOC)(\varnothing  ,C;G) :=  \left\{ \sur{ \varnothing
  }{g}[\zvedak C]
\right\},
\]
while $\ModHybFun(\onsOC)(\sfO,C;G)$ is empty in all remaining cases.
\end{proposition}

The modular operad compositions are defined as follows.
If $u$ is an open input of $V'_i$, $1 \leq i \leq a'$, 
and $v$ is an open input of $V''_j$, $1 \leq i \leq a''$, then
\[
\sur{ V'_1 \cdots V'_{a'} }{g'}[C'] \ooo{u}{v} \sur{V''_1  \cdots
  V''_{a''} }{g''}[C''] := \sur{ (V'_i\ooo{u}{v} V''_j) V'_1 \cdots
\widehat{V}'_i \cdots
 V'_{a'} V''_1 \cdots
\widehat{V}''_j \cdots
 V''_{a''}  }{g'+g''}[C' \sqcup C''].
\]
If $u \in C'$ and $v\in C''$ are closed inputs, then
\[
\sur{ V'_1 \cdots V'_{a'} }{g'}[C'] \ooo{u}{v} \sur{V''_1  \cdots
  V''_{a''} }{g''}[C'']
:= \sur{V'_1 \cdots V'_{a'} V''_1 \cdots V'_{a''} }{g'+g''}[C' \sqcup
C'' \setminus \stt{u,v}].
\]
If $u$ is an open input of $V_i$ and $v$   an open input of $V_j$, $1
\leq i,j \leq a$, $i\not= j$, then
\[
\oxi{uv} \sur{V_1  \cdots V_a }{g}[\zvedak C] := \sur{(V_i
  \ooo{u}{v} V_j) V_1 \cdots
\widehat{V}_i \cdots
\widehat{V}_j \cdots V_a }{g+1}[\zvedak C].
\]
If both $u,v$ are open inputs of the same $V_i$,  $1
\leq i \leq a$, then
\[
\oxi{uv} \sur{V_1  \cdots V_a }{g}[\zvedak C] 
= \sur{ (\oxi{uv}V_i) V_1 \cdots
\widehat{V}_i
  \cdots V_a }{g}[\zvedak C].
\]
Finally, if $u,v \in C$ are closed inputs, then
\[
\oxi{uv} \sur{ V_1 \cdots V_a }{g}[\zvedak C] : = \sur{V_1 \cdots V_a
}{g+1}
[C \zvedak \setminus \stt{u,v}].
\]
The unit $e: \onsOC \to \ModHybFun(\onsOC)$ of the adjunction
$\xymatrix@1@C0em{\ModHyb\rule{.5em}{0em} \ar@/^1pc/@<.2ex>[r]    &
  \ar@/^1pc/[l] \ \rule{.5em}{0em} \CycHyb
}
$ 
is given by
\begin{align}
\label{v_Koline_behem_soustredeni} 
e\left(\sur{\sfo_1\cdots\sfo_a}{0}[\zvedak C]
\right) :=  \sur{\sur{\sfo_1}{0} \cdots
  \sur{\sfo_a}{0} }{0}[\zvedak C]  \hbox { for }  a\geq 1, \hbox { and }
e\left(\sur{\varnothing}{0}[\zvedak C]\right) := \sur{\varnothing}{0}[C].
\end{align}

We will use the inclusion $e$ of~(\ref{v_Koline_behem_soustredeni}) to
view $\onsOC$ as a cyclic subhybrid of $\ModHybFun(\onsOC)$.
A~combinatorial characterization of pairs $(\sfO_i,G_i)$ for which
the set $\oQnsO(\sfO_i;G_i)$ in (iii) is
non-empty 
was given in
Example~\ref{ani_v_patek_se_neleti}. 
Namely, there must exist $g_i \in \bbN$ such that 
$
G_i = 2g_i + b_i -1
$, $V_i$ is then the symbol $\sur{\sfO_i}{g_i}$, $1 \leq i
\leq a$, and the element in~(\ref{v_pondeli_do_Leicesteru}) takes the form
\[
\sur{ \sur{\sfO_1}{g_1} \cdots \sur{\sfO_a}{g_a}}{\zvedak g}[\zvedak C].
\]
The graphical form of the expression above suggests to call  
$V_i =  \sur{\sfO_i}{g_i} \in\oQnsO$ in~(\ref{v_pondeli_do_Leicesteru}) 
a \emph{nest}.

\begin{remark}
\label{za_10_dni_asi_naposledy_do_Sydney}
Symbols~(\ref{v_pondeli_do_Leicesteru}) can be represented by oriented
surfaces $\Sigma$ with holes indexed by $C$, $b$ teethed boundaries
with teeth indexed by the multicycle~$\sfO$, and an extra data
consisting of $a$ embedded non-intersecting circles dividing $\Sigma$
into $a+1$ regions, say $R_1,\ldots,R_a,R_{a+1}$, such that $R_i$
contains teethed boundaries indexed by $\sfO_i$, $1\leq i \leq a$, and
$R_{a+1}$ all holes indexed by $C$.
\end{remark}

\begin{proof}[Proof of Proposition~\ref{THMModOCus}]
We need to verify the universal property saying that for each
modular hybrid $\oH$ and morphism of cyclic hybrids
$F:\onsOC\to\oH$ there is a unique morphism 
$\tilde{F}:\ModHybFun(\onsOC)\to\oH$ of modular hybrids such that the
following diagram commutes:
\[
\xymatrix@C4em{
\onsOC\ \ar@{^{(}->}^(.4)e[r] \ar_{F}[dr]  & 
\ModHybFun(\onsOC) \ar@{-->}^(.5){\tilde{F}}[d]
\\ &
\ \oH .    
}
\]
\noindent 
\paragraph{\bf Uniqueness.}
Assume that $\tF$ exists and prove its uniqueness. We have
the  diagram:
\begin{equation}
\label{psano_v_Leicesteru}
\xymatrix@C4em@R1em{
&&   \oCom \ar@{^{(}->}[r]\ar@<-.8ex>@{_{(}->}[ld]_{i_C} 
\ar[ddr]^(.6){F_{i_C}}|(.33)\hole & \oQC\ar[dd]^{F''} 
\ar@{^{(}->}[ld]_{\iota_C}
\\
&\ \onsOC \ \ar@{^{(}->}[r]\ar[ddr]^(.6)F|(.33)\hole  &
\ModHybFun(\onsOC)\ar[dd]^{\tilde{F} }&
\\
\nsoAss \ \ar[ddr]^(.6){F_{i_O}} 
\ar@{_{(}->}[r]  \ar@{^{(}->}[ur]^{i_O}&  \ar[dd]^{F'} 
\ar@{^{(}->}[ur]^{\iota_O} \oQnsO&& \oH 
\\
&&\oH\ar@{=}[ur]&
\\
&\oH \ar@{=}[ur]&&
}
\end{equation}
In this diagram, the inclusion  $\iota_O:\oQnsO \hookrightarrow
\ModHybFun(\onsOC)$ is the dotted arrow in 
\[
\xymatrix@C4em{
\nsoAss\ \ar@{^{(}->}[d] \ar@{^{(}->}[r]^{i_O}  &
\ar@{-->}[d]^{\iota_O}   
\oQnsO
\\ 
\ar@{^{(}->}[r]^e \onsOC\ &
\ModHybFun(\onsOC)    
}
\]
where $e$ is the unit of the
adjunction~(\ref{v_Koline_behem_soustredeni}) 
and $i_O$ the inclusion~(\ref{eq:1}).
The dotted arrow exists as $\nsoMod(\nsoAss)\cong\oQnsO$ by~\cite{nsmod}.
Likewise, the inclusion $\iota_C :\oQC \hookrightarrow
\ModHybFun(\onsOC)$ is the dotted arrow~in
\[
\xymatrix@C4em{
\oCom\ \ar@{^{(}->}[d]  \ar@{^{(}->}[r]^{i_C} \ar[r]  &
\ar@{-->}[d]^{\iota_C} 
 \oQC
\\ 
\ar@{^{(}->}[r]^e \onsOC\ &
\ModHybFun(\onsOC)    
}
\]
where $i_C$ is the 
inclusion~(\ref{ic}). It exists since,
by~\cite[page~382]{markl:la}, $\oQC \cong 
\oMod{\oCom}$.
It is easy to show that
\[
\iota_O( V) = \sur{V}{0}[\emptyset]
\hbox {
  resp. }\
\iota_C \sur{g}{\podpera C} =
\sur{\varnothing}{g}[\podpera  C]
\]
so the maps $\iota_O$ and $\iota_C$ are indeed inclusions.
To simplify the notation, 
we use these injections to interpret $\oQC$ and $\oQnsO$ as suboperads
of $\ModHybFun(\onsOC)$.

Further, $F' : \oQnsO \to \oH$ in~(\ref{psano_v_Leicesteru})
is the dotted arrow in 
\[
\xymatrix@C4em{
\nsoAss\ \ar@{^{(}->}[d] \ar@{^{(}->}[r]^{i_O}  &
\ar@{-->}[d]^{F'}  
\oQnsO
\\ 
\ar@{^{(}->}[r]^F \onsOC\ &
\ModHybFun(\onsOC)    
}
\] 
and $F'' : \oQC \to \oH$ the dotted arrow in
\[
\xymatrix@C4em{
\oCom\ \ar@{^{(}->}[d]  \ar@{^{(}->}[r]^{i_C} \ar[r]  &
\ar@{-->}[d]^{F''}  
\oQC
\\ 
\ar@{^{(}->}[r]^F \onsOC\ &
\ModHybFun(\onsOC)    
}
\]
By  the uniqueness of $F'$ resp.\ $F''$,
\begin{gather} \label{EQRestrictionOfTildeFOC}
\tF \circ \iota_O =\tF|_{\oQnsO} = \pF \ \hbox { and } \tF \circ \iota_C =
\tF|_{\oQC} = \ppF.
\end{gather}

Before we proceed, we need to introduce the following terminology.
Let $\sfo$ be a cycle with underlying set $O$ and $p$ an independent 
symbol. We will call,
only for the purposes of this proof, by an {\em extension\/} of $\sfo$
a cycle $p\sfo$ whose underlying set is $\stt p \sqcup O$ such that
the induced cyclic order on $O$ coincides with $\sfo$. It is clear
that extensions exist; if $\sfO = \cyc{o_1,o_2,\ldots,o_n}$, then
$\cyc{p,o_1,o_2,\ldots,o_n}$ is an extension. On the other hand,
extensions are not unique. Although
$
\cyc{o_1,o_2,\ldots,o_n} = \cyc{o_2,\ldots,o_n,o_1},
$
\[
\cyc{p,o_1,o_2,\ldots,o_n} \not= \cyc{p,o_2,\ldots,o_n,o_1}
\]
if $n \geq 2$.
Extensions can be generalized to multicycles. If $\sfO$ is a
multicycle, then an extension of $\sfO$ by $p$ is a multicycle $p\sfO$
some of whose cycles has been extended by $p$.

Using the definition of the hybrid modular structure of
$\ModHybFun(\oQnsOC)$ we get the following expression for its general element:
\begin{equation}
\label{Jaruska_napsala}
\sur{ \sur{\sfO_1}{g_1} \cdots \sur{\sfO_a}{g_a}}{\zvedak g}[C]
=
\sur{ \sur{{p'_1\sfO_1}}{g_1} }{0}[\emptyset] \ooo{p'_1}{p''_1}
\cdots
\sur{ \sur{{p'_a\sfO_a}}{g_a} }{0}[\emptyset] \ooo{p'_a}{p''_a} 
\sur{ \sur{\cycs{p''_1}}{0} \cdots
   \sur{\cycs{p''_a \zvedak} }{0} }{\zvedak g}[C]
\end{equation}
where we may further express
\begin{equation}
\label{zitra_do_Lancasteru}
\sur{ \sur{\cycs{p''_1}}{0} \cdots
   \sur{\cycs{p''_a \zvedak} }{0} }{\zvedak g}[C] = 
\oxi{u'_1u''_1}\cdots\oxi{u'_gu''_g}
\sur{ \sur{\cycs{p''_1}}{0} \cdots
   \sur{\cycs{p''_a \zvedak} }{0} }{\zvedak 0}[C \sqcup 
\stt{u'_1,u''_1,\ldots,u'_g,u''_g}]
\end{equation}
with some independent variables $u'_1,u''_1,\ldots,u'_g,u''_g$.
Notice that the elements
\[
\sur{ \sur{{p'_1\sfO_1}}{g_1} }{0}[\emptyset] ,\ldots,
\sur{ \sur{{p'_a\sfO_a}}{g_a} }{0}[\emptyset] 
\]
in the right hand side
of~(\ref{Jaruska_napsala}) belong to 
the image of $\iota_O$ and therefore are identified with 
\[
\sur{{p'_1\sfO_1}}{\mini g_1}, \ldots , 
\sur{{p'_a\sfO_a}}{\mini g_a} \in \oQnsO,
\]
while the term
\[
\sur{ \sur{\cycs{p''_1}}{0} \cdots
   \sur{\cycs{p''_a \zvedak} }{0} }{\zvedak 0}[C \sqcup 
\stt{u'_1,u''_1,\ldots,u'_g,u''_g}]
\]  
in the right hand side of~(\ref{zitra_do_Lancasteru}) belongs to the
image of $e: \onsOC \to \ModHybFun(\onsOC)$ and therefore is identified with
\[
\sur{\cycs{p''_1}\cdots\cycs{p''_a\zvedak}}{0}[\podpera C\sqcup 
\stt{u'_1,u''_1,\ldots,u'_g,u''_g}]
\in \onsOC.
\] 
Combining these observations we see that
\begin{align}
\label{s1}
\tF
\sur{ \sur{\sfO_1}{g_1} \cdots \sur{\sfO_a}{g_a}}{\zvedak g}[\podpera C]
=
F'\! \sur{{p'_1\sfO_1\! }}{\mini g_1}\! \ooo{p'_1}{p''_1}\!
\squeezeddots F' \sur{{p'_a\sfO_a\! }}{\mini g_a}\! \ooo{p'_a}{p''_a}\!
 \oxi{u'_1u''_1}\cdots\oxi{u'_gu''_g}
F\!\sur{\cycs{p''_1}\cdots\cycs{\podpera p''_a}}{0}[\podpera C\sqcup 
\stt{u'_1,u''_1,\ldots,u'_g,u''_g}] .
\end{align}
Since $F'$ is unique, (\ref{s1}) determines $\tF$ uniquely on
elements as the one in the left hand side
of~(\ref{Jaruska_napsala}). The proof of the uniqueness is finished by
observing that 
\begin{gather} \label{EQDefTildeOCTwo}
\tF\sur{\varnothing}{g}[\podpera C] =
F''\sur{}{\raisebox{.2em}{\scriptsize $g$}}[C].
\end{gather}

\paragraph{\bf Independence on the choices.}
The aim of this part is to show that the value of the
right hand side of~(\ref{s1})
does not depend on the choices of the extensions 
$p_1'\sfO_1,\ldots,p'_a\sfO_a$.
It will be convenient to rewrite it as
\begin{align}
\label{s11}
\nonumber 
\tF
\sur{ \sur{\sfO_1}{g_1} \cdots \sur{\sfO_a}{g_a}}{\zvedak g}[C]
&=  \oxi{u'_1u''_1}\cdots\oxi{u'_gu''_g} 
\left(
F'\! \sur{{p'_1\sfO_1\! }}{\mini g_1}\! \ooo{p'_1}{p''_1}\!
\cdots F' \sur{{p'_a\sfO_a\! }}{\mini  g_a}\! \ooo{p'_a}{p''_a}\!
F\!\sur{\cycs{p''_1}\cdots\cycs{p''_a\zvedak }}{0}[S]  
\right)
\end{align}
with $S := C\sqcup 
\stt{u'_1,u''_1,\ldots,u'_g,u''_g}$. To prove the independence
of the right hand side of~(\ref{s1})  on the choices, 
it clearly suffices to show the
independence of the expression
\begin{equation}
\label{uz_jsem_z_toho_pocitani_zoufaly}
F'\! \sur{{p'_1\sfO_1\! }}{\mini  g_1} \ooo{p'_1}{p''_1}    
\cdots F' \sur{{p'_a\sfO_a\! }}{\mini g_a}     \!
 \ooo{p'_a}{p''_a}\!
F\!\sur{\cycs{p''_1}\cdots\cycs{p''_a\zvedak}}{0}[S] .
\end{equation}
Since~(\ref{uz_jsem_z_toho_pocitani_zoufaly}) does not depend
on the 
order of $\sfO_1,\ldots,\sfO_a$,   it suffices to
prove that it does not depend on the choice of the extension
$p'_1\sfO_1$. 

 Assume that $\sfO_1 = \sfo_1 \sfo_2\cdots \sfo_b$, $p'_1\sfO_1 =
p'_1\sfo_1\, \sfo_2\cdots \sfo_b$,  
and prove that~(\ref{uz_jsem_z_toho_pocitani_zoufaly}) does
not depend on the choice 
of the extension $p'_1\sfo_1$ of the cycle $\sfo_1$. 
One has
\begin{align*}
F'\!& \sur{{p'_1\sfo_1\, \sfo_2\ \cdots\ \sfo_b\! }}{\mini g_1}
\ooo{p'_1}{p''_1}
\cdots F' \sur{{p'_a\sfO_a\! }}{\mini g_a} 
\ooo{p'_a}{p''_a}\!
F\!\sur{\cycs{p''_1}\cdots\cycs{p''_a\zvedak}}{0}[S] =
\\
&= F'\!\left( \sur{{\cycs{r'}\, \sfo_2\ \cdots\ \sfo_b\! }}{\mini g_1}
\! \ooo{r'}{r''}\!
 \sur{{r''p'_1\sfo_1}}{0}\right)\ooo{p'_1}{p''_1}
\cdots F' \sur{{p'_a\sfO_a\! }}{\mini g_a}   
\! \ooo{p'_a}{p''_a}\!
F\!\sur{\cycs{p''_1}\cdots\cycs{p''_a\zvedak}}{0}[S] 
\\
&= F'\!\sur{{\cycs{r'}\, \sfo_2\ \cdots\ \sfo_b\! }}{\mini g_1}
\! \ooo{r'}{r''}\!
F'\! \sur{{r''p'_1\sfo_1}}{0}\ooo{p'_1}{p''_1}
\cdots F' \sur{{p'_a\sfO_a\! }}{\mini g_a}   
\! \ooo{p'_a}{p''_a}\!
F\!\sur{\cycs{p''_1}\cdots\cycs{p''_a\zvedak}}{0}[S] 
\\
&= F'\! \sur{{\cycs{r'}\, \sfo_2\ \cdots\ \sfo_b\! }}{\mini g_1}
\! \ooo{r'}{r''}\!
F \sur{{r''p'_1\sfo_1}}{0}[\emptyset]\ooo{p'_1}{p''_1} \cdots
 F' \sur{{p'_a\sfO_a\! }}{\mini g_a}   
 \ooo{p'_a}{p''_a}\!
F\!\sur{\cycs{p''_1}\cdots\cycs{p''_a\zvedak}}{0}[S] 
\\
&= F'\! \sur{{\cycs{r'}\, \sfo_2\ \cdots\ \sfo_b\! }}{\mini g_1}
\! \ooo{r'}{r''}\!
 F' \sur{{p'_2\sfO_2\! }}{\mini g_2} \ooo{p'_2}{p''_2}
\cdots F' \sur{{p'_a\sfO_a\! }}{\mini g_a}   
\ooo{p'_a}{p''_a}\! F \sur{{r''p'_1\sfo_1}}{0}[\emptyset] \ooo{p'_1}{p''_1}
F\!\sur{\cycs{p''_1}\cdots\cycs{p''_a\zvedak}}{0}[S] 
\\
&= F'\! \sur{{\cycs{r'}\, \sfo_2\ \cdots\ \sfo_b\! }}{\mini g_1}
\! \ooo{r'}{r''}\!
 F' \sur{{p'_2\sfO_2\! }}{\mini g_a}\ooo{p'_2}{p''_2} \cdots F' \sur{{p'_a\sfO_a\! }}{\mini g_a} \ooo{p'_a}{p''_a}\!
F\!\sur{r''\sfo_1\, \cycs{p''_2} \cdots\cycs{p''_a\zvedak}}{0}[S] .
\end{align*}
The expression in the last line is clearly independent of the position at which
$p'_1$ was inserted into the cycle $\sfo_1$. 

It remains to show that~(\ref{uz_jsem_z_toho_pocitani_zoufaly}) is
independent of the order of $\sfo_1,\ldots,\sfo_b$, i.e.\ that,
choosing $p'_1\sfO_1 = p'_1\sfo_i\cdots  \sfo_1 \cdots
\widehat{\sfo}_i \cdots  \sfo_b$, the value of
(\ref{uz_jsem_z_toho_pocitani_zoufaly}) does not depend on $i$, $
1\leq i \leq b$. 
Before we do so, we 
warn the reader that while 
\[
F'\sur{\sfo}{0} = F\sur{\sfo}{0}[\emptyset]
\]
for any cycle
$\sfo$, it is not necessarily true that
\[
F'\sur{\sfO}{0} = F\sur{\sfO}{0}[\emptyset]
\] 
for a {\em multi\/}cycle
$\sfO=\sfo_1\sfo_2\cdots$.\footnote{This become true under some
additional conditions discussed in Proposition \ref{PROCUnivProp}
below.} The reason is that in general
\[
\iota_O \sur{\sfO}{0} \not= e \sur{\sfO}{0}[\emptyset].
\]
One can however still express $F'\sur{\sfO}{0}$ very explicitly as follows.

For a totally ordered finite set $A$ denote, as in
Definition~\ref{zase_jsem_podlehl_uz_po_trech_dnech}, by 
$\cycs{A}$ the induced cycle.  Each cycle
$\sfo$ is of this form for some (non-unique) totally
ordered $A$. So, let $\sfo_1 = \cycs {A_1}$ and  $\sfo_2 = \cycs {A_2}$ be 
cycles, $x',x''$ independent symbols and $\sfo := \cycs{A_1 x' A_2
  x''}$. Then, in $\oQnsO$, one has the identity
$\sur{\, \sfo_1 \sfo_2}{\podpera 0} = \oxi{x'x''}  \sur{\sfo}{\podpera 0}$, 
therefore, since $F': \oQnsO \to \oH$ is a morphism, 
\begin{equation}
  \label{Martin_Doubek_po_smrti}
F' \sur{\, \sfo_1 \sfo_2}{\podpera 0} = \oxi{x'x''}F'
\sur{\sfo}{\podpera 0}  \hbox { in } \oH .
\end{equation}
It is easy to extend~(\ref{Martin_Doubek_po_smrti}) to an arbitrary number of
cycles, i.e.\ to an arbitrary multicycle.

With~(\ref{Martin_Doubek_po_smrti}) at hand, we are ready to 
prove that~(\ref{uz_jsem_z_toho_pocitani_zoufaly}) is
independent of the order of $\sfo_1,\ldots,\sfo_b$. 
To keep the size of formulas within  reasonable limits, we assume that
$a=2$, the general case is analogous. One has
\begin{align*}
F'\! \sur{{p'_1\sfo_1\, \sfo_2\, \sfo_3\, \cdots\, \sfo_b }}{\mini
  g_1}\!& 
\ooo{p'_1}{p''_1}\!
 F' \sur{{p'_2\sfO_2\! }}{\mini g_2}\! \ooo{p'_2}{p''_2}\!
F\!\sur{\cycs{p''_1}\cycs{p''_2}}{0}[S]  =
\\
=&\
F' \sur{\cycs{s'}\cycs{r'}\, \sfo_3 \cdots \, \sfo_b}{\mini g_1}
 \ooo{s'}{s''} \ooo{r'}{r''}
F'\! \sur{{s''p'_1\sfo_1\, r''\sfo_2 }}{0}\! \ooo{p'_1}{p''_1}\!
 F' \sur{{p'_2\sfO_2\! }}{\mini g_2}\! \ooo{p'_2}{p''_2}\!
F\!\sur{\cycs{p''_1}\cycs{p''_2}}{0}[S]  
\\
=&\
F' \sur{\cycs{s'}\cycs{r'}\, \sfo_3 \cdots \, \sfo_b}{\mini g_1}
 \ooo{s'}{s''} \ooo{r'}{r''} \oxi{x'x''}
F'\! \sur{{s''p'_1r''\sfo}}{0}\! \ooo{p'_1}{p''_1}\!
 F' \sur{{p'_2\sfO_2\! }}{\mini g_2}\! \ooo{p'_2}{p''_2}\!
F\!\sur{\cycs{p''_1}\cycs{p''_2}}{0}[S]  
\\
=&\
 \oxi{x'x''}
F' \sur{\cycs{s'}\cycs{r'}\, \sfo_3 \cdots \, \sfo_b}{\mini g_1}
 \ooo{s'}{s''} \ooo{r'}{r''}
F\! \sur{{s''p'_1r''\sfo}}{0}[\emptyset]\! \ooo{p'_1}{p''_1}\!
 F' \sur{{p'_2\sfO_2\! }}{\mini g_2}\! \ooo{p'_2}{p''_2}\!
F\!\sur{\cycs{p''_1}\cycs{p''_2}}{0}[S]  \\
=&\
 \oxi{x'x''}F' \sur{\cycs{s'}\cycs{r'}\, \sfo_3 \cdots \, \sfo_b }{\mini g_1}
 \ooo{s'}{s''} \ooo{r'}{r''}
 F' \sur{{p'_2\sfO_2\! }}{\mini g_2}\! \ooo{p'_2}{p''_2}\!
F\!\sur{s'' r''\sfo \,\cycs{p''_2}}{0}[S] ,
\end{align*}
where the relation between $\sfo_1$, $\sfo_2$ in the second 
and $\sfo$ in the third
line is as in~(\ref{Martin_Doubek_po_smrti}).
The term in the last line clearly does not
see whether $p'_1$ was inserted into 
$\sfo_1$ or $\sfo_2$. This shows
that~(\ref{uz_jsem_z_toho_pocitani_zoufaly}) is invariant under the
transposition $\sfo_1 \leftrightarrow \sfo_2$. The transpositions  
$\sfo_1 \leftrightarrow \sfo_i$ for arbitrary $1< i
\leq b$ can be discussed similarly.

\paragraph{\bf Morphism property}
Let us define $\tilde{F}:\ModHybFun(\onsOC)\to\oH$ by
formulas \eqref{s1} and \eqref{EQDefTildeOCTwo}. 
It is simple to show that such an $\tF$ extends $F$, i.e.\ that $\tF
\circ e = F$; we leave this as an exercise. 
It is also clear that $\tF$ defined in this way
is equivariant with respect of  the automorphisms of the indexing
(multicyclic) sets, and is  genus-preserving. 
To finish the proof of  Proposition~\ref{THMModOCus} we need to show
that this $\tF$ commutes with the structure operations of modular hybrids.
The commutation with the modular operad
structure in the `closed' color is simple and we leave it as an
exercise.

Let us show that $\tF$ commutes with the $\ooo uv$-operations in the
`open' color. To save the space, we prove it in the simplest
nontrivial case. It will be clear that the general case can be
attended analogously. We are therefore going to prove that
\begin{equation}
  \label{dnes zmeskam Brezanku}
\tF \sur{\sur{\sfO_1}{g_1}}{\podpera g'}[\podpera C_1]
\ooo{u}{v} 
\tF \sur{\sur{\sfO_2}{g_1}}{\podpera g''}[\podpera C_2]
=
\tF
\left(
\sur{\sur{\sfO_1}{g_1}}{\podpera g'}[\podpera C_1]
\ooo{u}{v} 
\sur{\sur{\sfO_2}{g_1}}{\podpera g''}[\podpera C_2]
\right)
=
\tF\sur{\sur{\sfO_1 \ooo{u}{v} \sfO_2}{g_1 + g_2}}{\podpera g'+g''}[\podpera C_1
\sqcup \podpera C_2].
\end{equation}
From the definition of $\tF$ we get
\begin{align*}
\left(
F'\! \sur{{p'_1\sfO_1\! }}{\mini g_1} \ooo{p'_1}{p''_1}     
F\!\sur{\cycs{p''_1}}{0}[S_1]
\right)
& \ooo uv  
\left(
F'\! \sur{{p'_2\sfO_2\! }}{\mini g_2} \ooo{p'_2}{p''_2}     
F\!\sur{\cycs{p''_2}}{0}[S_2]
\right) =
\\
&=
F'\! \sur{{p'_1\sfO_1\! }}{\mini g_1} \ooo uv F'\! \sur{{p'_2\sfO_2\! }}{\mini g_2}
\ooo{p'_1}{p''_1} F\!\sur{\cycs{p''_1}}{0}[S_1]   \ooo{p'_2}{p''_2}
F\!\sur{\cycs{p''_2}}{0}[S_2]  
\\
&=
F' \sur{{p'_1p'_2(\sfO_1\ooo uv \sfO_2)  }}{\mini g_1 + g_2}
\ooo{p'_1}{p''_1} F\!\sur{\cycs{p''_1}}{0}[S_1]   \ooo{p'_2}{p''_2}
F\!\sur{\cycs{p''_2}}{0}[S_2] .
\end{align*}
Assume that $\sfO_1\ooo uv \sfO_2 = \sfo_1 \sfo_2 \cdots
\sfo_b$ and
$p'_1p'_2(\sfO_1\ooo uv \sfO_2) = p'_1p'_2 \sfo_1\, \sfo_2 \cdots
\, \sfo_b$. Then
\begin{align*}
F'& \sur{{p'_1p'_2(\sfO_1\ooo uv \sfO_2)  }}{\mini g_1 + g_2}
\ooo{p'_1}{p''_1} F\!\sur{\cycs{p''_1}}{0}[S_1]   \ooo{p'_2}{p''_2}
F\!\sur{\cycs{p''_2}}{0}[S_2] =
\\
&\hphantom{ooooo}=
F' \sur{{ p'_1p'_2 \sfo_1 \, \sfo_2 \cdots
\, \sfo_b }}{\mini g_1 + g_2}
\ooo{p'_1}{p''_1} F\!\sur{\cycs{p''_1}}{0}[S_1]   \ooo{p'_2}{p''_2}
F\!\sur{\cycs{p''_2}}{0}[S_2] 
\\
&\hphantom{ooooo}=
F' 
\sur{\cycs{r'}\, \sfo_2 \cdots\, \sfo_b}{\mini g_1 + g_2} 
\ooo{r'}{r''} 
F' \sur{r''p'_1p'_2 \sfo_1}{\mini 0} 
\ooo{p'_1}{p''_1} F\!\sur{\cycs{p''_1}}{0}[S_1]   \ooo{p'_2}{p''_2}
F\!\sur{\cycs{p''_2}}{0}[S_2]  
\\
&\hphantom{ooooo}=
F' 
\sur{\cycs{r'}\, \sfo_2 \cdots\, \sfo_b}{\mini g_1 + g_2} 
\ooo{r'}{r''} 
F \sur{r''p'_2 \sfo_1}{\mini 0} [S_1]   \ooo{p'_2}{p''_2}
F\!\sur{\cycs{p''_2}}{0}[S_2]  
=
F' 
\sur{\cycs{r'}\, \sfo_2 \cdots\, \sfo_b}{\mini g_1 + g_2} 
\ooo{r'}{r''} 
F \sur{r'' \sfo_1}{\mini 0} [S_1 \sqcup \, S_2]   
\\
&\hphantom{ooooo}=
F' 
\sur{\cycs{r'}\, \sfo_2 \cdots\, \sfo_b}{\mini g_1 + g_2} 
\ooo{r'}{r''} 
F \sur{r''p'{} \sfo_1}{\mini 0} [\emptyset]   \ooo{p'{}}{p''{}}
F\!\sur{\cycs{p''{}}}{0}[S_1 \sqcup \, S_2]  
\\
&\hphantom{ooooo}=
F' 
\sur{\cycs{r'}\, \sfo_2 \cdots\, \sfo_b}{\mini g_1 + g_2} 
\ooo{r'}{r''} 
F' \sur{r''p'{} \sfo_1}{\mini 0}   \ooo{p'{}}{p''{}}
F\!\sur{\cycs{p''{}}}{0}[S_1 \sqcup \, S_2]  
\\
&\hphantom{ooooo}=
F' 
\sur{p'{} \sfo_1 \, \sfo_2 \cdots\, \sfo_b}{\mini g_1 + g_2}   \ooo{p'{}}{p''{}}
F\!\sur{\cycs{p''{}}}{0}[S_1 \sqcup \, S_2]  =
F' 
\sur{p'{}(\sfO_1\ooo uv \sfO_2}{\mini g_1 + g_2}   \ooo{p'{}}{p''{}}
F\!\sur{\cycs{p''{}}}{0}[S_1 \sqcup \, S_2] . 
\end{align*}
The last term of the above display equals  
the right hand side of~(\ref{dnes zmeskam Brezanku}) evaluated
using~(\ref{s1}).

Let us prove that $\tF$ commutes with the `open' contractions
$\oxi{uv}$. Again we discuss only the simplest nontrivial case,
the general one can be treated similarly. We will verify that
\begin{equation}
\label{druhy_den_v_Mulhouse}
\oxi {uv} \tF  \sur{ \sur{\sfO_1}{g_1}   
\sur{\sfO_2}{g_2}}{\podpera g}[\podpera C]
=
\tF\left(  \oxi {uv} \sur{ \sur{\sfO_1}{g_1} 
\sur{\sfO_2}{g_2}}{\podpera g}[\podpera C]\right).
\end{equation}
Assume that $u,v$ belongs to the same multicycle, say to $\sfO_1$. 
Then~(\ref{druhy_den_v_Mulhouse}) boils to
\begin{equation}
\label{airshow}
\oxi {uv} \tF \sur{ \sur{\sfO_1}{g_1}  \sur{\sfO_2}{g_2}}{\podpera g}[\podpera C]
=
\tF \sur{ \oxi {uv} \sur{\sfO_1}{g_1} \sur{\sfO_2}{g_2}}{\podpera g}[\podpera C]
\end{equation}
in which, by definition,
\[
\oxi {uv} \sur{\sfO_1}{g_1} =  \sur{\oxi
  {uv}\sfO_1}{g'_1}
\]
where $g'_1$ equals $g_1$ or $g_1 +1$ depending
on whether $u$ and $v$ belong to the same cycle or the 
different cycles of $\sfO_1$.
We therefore rewrite~(\ref{airshow}) as
\begin{equation}
\label{airshowbis}
\oxi {uv} \tF \sur{ \sur{\sfO_1}{g_1}  \sur{\sfO_2}{g_2}}{\podpera g}[\podpera C]
=
\tF \sur{  \sur{\oxi {uv}\sfO_1}{g'_1} \sur{\sfO_2}{\podpera g_2}}{\podpera g}[\podpera C]
\end{equation}
The left hand side of~(\ref{airshowbis}) equals
\begin{align*}
\oxi{uv}
\left( \rule{0em}{1.6em}
F'\! \sur{{p'_1\sfO_1\! }}{\mini g_1} \ooo{p'_1}{p''_1}     
F' \sur{{p'_2\sfO_2\! }}{\mini g_2} \right.   & \left.
 \ooo{p'_2}{p''_2}\!
F\!\sur{\cycs{p''_1}\cycs{p''_2\zvedak}}{0}[S]
\right) =
\\
&=
\oxi{uv} F'\! \sur{{p'_1\sfO_1\! }}{\mini g_1} \ooo{p'_1}{p''_1}     
F' \sur{{p'_2\sfO_2\! }}{\mini g_2}    
 \ooo{p'_2}{p''_2}\!
F\!\sur{\cycs{p''_1}\cycs{p''_2\zvedak}}{0}[S] 
\\
&=
F' \sur{{p'_1 \oxi{uv}\sfO_1\! }}{\mini g'_1}  \ooo{p'_1}{p''_1}     
F' \sur{{p'_2\sfO_2\! }}{\mini g_2}    
 \ooo{p'_2}{p''_2}\!
F\!\sur{\cycs{p''_1}\cycs{p''_2\zvedak}}{0}[S] 
\end{align*}
which is  the right hand side of~(\ref{airshowbis})
expressed via~\eqref{s1}.
Before we go further,  we need to prove an auxiliary

\begin{sublemma}
\label{Zitra Airshow v Rixheimu}
Let $\sfo_1$ and $\sfo_2$ be cycles, $S$ a finite set and $u,v,p'$ and $p''$
independent symbols. Then
\begin{equation}
\label{varhanky u Svate Marie}
\oxi {uv} F \sur{\sfo_1 \ \sfo_2}{0}[\podpera S \sqcup \stt {u,v}] 
=
\oxi {p'p''} F \sur{p'\sfo_1\  p''\sfo_2}{0}[\podpera S]   .
\end{equation}
\end{sublemma}

\begin{proof}[Proof of the sublemma]
It follows from the axioms of modular operads that
\[
\oxi {uv} \left(
 F \sur{\sfo_1}{0}[\podpera S \sqcup \stt {u}] \ooo {p'}{p''} 
F \sur{\sfo_2}{0}[\podpera \stt{v}]
\right)
=
\oxi {p'p''} \left(
 F \sur{\sfo_1}{0}[\podpera S \sqcup \stt {u}] \ooo {u}{v} 
F \sur{\sfo_2}{0}[\podpera \stt{v}]
\right).
\]
Equation~(\ref{varhanky u Svate Marie}) is then a consequence  of the fact
that $F$ is a morphism of cyclic hybrids and of the definition of
the structure operations in $\onsOC$. 
\end{proof}

If $u \in \sfO_1$ and  $v  \in \sfO_2$,~(\ref{druhy_den_v_Mulhouse}) boils to
\begin{equation}
\label{airshow1}
\oxi {uv} \tF \sur{ \sur{\sfO_1}{g_1}  \sur{\sfO_2}{g_2}}{\podpera g}[\podpera C]
=
\tF \sur{
\sur{\sfO_1 \ooo {u}{v} \sfO_2}{g_1 + g_2}
}{\podpera g+1}[\podpera C].
\end{equation}
The left hand side of the above display equals
\begin{align*}
\oxi{uv}
\left( \rule{0em}{1.6em}
F'\! \sur{{p'_1\sfO_1\! }}{\mini g_1} \ooo{p'_1}{p''_1}     
F' \sur{{p'_2\sfO_2\! }}{\mini g_2} \right.   & \left.
\! \ooo{p'_2}{p''_2}\!
F\!\sur{\cycs{p''_1}\cycs{p''_2\zvedak}}{0}[\podpera S]
\right)=
\\
=&\
\oxi{uv} \big( F'\! \sur{{p'_1\sfO_1\! }}{\mini g_1} \ooo{p'_1}{p''_1}     
F' \sur{{p'_2\sfO_2\! }}{\mini g_2}\big)    
 \ooo{p'_2}{p''_2}\!
F\!\sur{\cycs{p''_1}\cycs{p''_2\zvedak}}{0}[\podpera S] 
\\
=&\
F' \sur{{p'_1p'_2(\sfO_1\ooo uv \sfO_2)  }}{\mini g_1 + g_2} 
\ooo{p'_1}{p''_1}     
 \ooo{p'_2}{p''_2}\!
F\!\sur{\cycs{p''_1}\cycs{p''_2\zvedak}}{0}[\podpera S] .
\end{align*}
Assume that $\sfO_1\ooo uv \sfO_2 =  \sfo_1 \sfo_2 \cdots
\sfo_b$ and  $p'_1p'_2(\sfO_1\ooo uv \sfO_2) = p'_1p'_2 \sfo_1\, \sfo_2 \cdots
\sfo_b$. Then
\begin{align*}
F' &\sur{{p'_1p'_2(\sfO_1\ooo uv \sfO_2)  }}{\mini g_1 + g_2} \!
\ooo{p'_1}{p''_1}     
 \ooo{p'_2}{p''_2}\!
F\!\sur{\cycs{p''_1}\cycs{p''_2\zvedak}}{0}[\podpera S] =
\\
&\hphantom{ooo}= 
F' \sur{  p'_1p'_2 \sfo_1 \, \sfo_2 \cdots
\, \sfo_b}{\mini g_1 + g_2} 
\ooo{p'_1}{p''_1}  
 \ooo{p'_2}{p''_2}\!
F\!\sur{\cycs{p''_1}\cycs{p''_2\zvedak}}{0}[\podpera S]
\\
&\hphantom{ooo}= 
F' 
\sur{\cycs{r'}\, \sfo_2 \cdots\, \sfo_b}{\mini g_1 + g_2} 
\ooo{r'}{r''} 
F' \sur{r''p'_1p'_2 \sfo_1}{\mini 0} 
\ooo{p'_1}{p''_1}  
 \ooo{p'_2}{p''_2}\!
F\!\sur{\cycs{p''_1}\cycs{p''_2\zvedak}}{0}[\podpera S]
\\&\hphantom{ooo}
= 
F' 
\sur{\cycs{r'}\, \sfo_2 \cdots\, \sfo_b}{\mini g_1 + g_2} 
\ooo{r'}{r''} 
 \oxi{p'_2p''_2}
F\!\sur{r''p'_2 \sfo_1 \ \cycs{p''_2\zvedak}}{0}[\podpera S]
= 
F' 
\sur{\cycs{r'}\, \sfo_2 \cdots\, \sfo_b}{\mini g_1 + g_2} 
\ooo{r'}{r''} 
 \oxi{x',x''}
F\!\sur{r''\sfo_1 }{0}[S \sqcup \stt{x',x''}]
\\
&\hphantom{ooo}= 
 \oxi{x',x''}
F' 
\sur{\cycs{r'}\, \sfo_2 \cdots\, \sfo_b}{\mini g_1 + g_2} 
\ooo{r'}{r''} F'\sur { r''p' \sfo_1}{0}\ooo{p'}{p''}  
F\!\sur{\cycs{p''} }{0}[\podpera S \sqcup \stt{x',x''}]
\\
&\hphantom{ooo}= 
 \oxi{x',x''}
 F'\sur {p' \sfo_1\, \sfo_2 \cdots\, \sfo_b}{\mini g_1 + g_2}\ooo{p'}{p''}  
F\!\sur{\cycs{p''} }{0}[\podpera S \sqcup \stt{x',x''}]
=
 \oxi{x',x''}
 F'\sur {p' (\sfO_1 \ooo uv \sfO_2)}{\mini g_1 + g_2}\ooo{p'}{p''}  
F\!\sur{\cycs{p''} }{0}[\podpera S \sqcup \stt{x',x''}],
\end{align*}
where  in the 4th line we used Sublemma~\ref{Zitra Airshow v Rixheimu}.
It is clear that the last term equals the right hand side
of~(\ref{airshow1}) evaluated via~(\ref{s1}). This finishes the proof
of Proposition~\ref{THMModOCus}.
\end{proof}

\subsection{Modular \envelope\ modulo Cardy conditions.}
In this subsection we identify $\oQnsOC$ with the quotient of 
$\ModHybFun(\onsOC)$
by the Cardy conditions.
\setcounter{footnote}{0}

\begin{theorem} 
\label{PROCUnivProp}
Let us consider the ideal\/  $\EuScript{I}$\footnote{The term {\em congruence\/} instead of
  `ideal' might be more appropriate in the context of sets, but we
  take the liberty to stick to the terminology we are
  used to.} 
in the  modular hybrid\/ $\ModHybFun(\onsOC)$ generated by the single~relation
\[
\sur{ \sur{\cycs{q}}{0} \sur{\cycs{r}}{0} }{0}[\emptyset] = \sur{
  \sur{\cycs{q}\cycs{r}}{0} }{0}[\emptyset].
\]
Then
\begin{equation}
\label{Kdy se s tim prestanu trapit?}
\oQnsOC \ \cong\ \ModHybFun(\onsOC) /\, \EuScript{I}.
\end{equation}
Consequently, for any  modular hybrid $\oH$ and any morphism 
$F:\onsOC\to\oH$ of cyclic hybrids satisfying the relation
\begin{gather} 
\label{EQCardyOCStatement}
\oxi{uv}F\sur{\cycs{uqvr}}{0}[\emptyset] 
= F\sur{\cycs{q}\cycs{r}}{0}[\emptyset],
\end{gather}
there is a unique morphism $\hat{F}:\oQnsOC\to\oH$ of modular hybrids
for which the diagram 
\[
\xymatrix@C4em{
\onsOC\ \ar@{^{(}->}[r] \ar_{F}[dr]  & 
\oQnsOC \ar@{-->}^(.5){\hat{F}}[d]
\\ &
\oH 
}
\]
commutes.
\end{theorem}

\begin{remark}
Equation~(\ref{EQCardyOCStatement}) is equivalent to
\[
\oxi{uv} \left( F\sur{\cycs{uqa}}{0}[\podpera\emptyset] \ooo{a}{b} 
F\sur{\cycs{bvr}}{0}[\podpera \emptyset] \right) = 
F\sur{\cycs{q}}{0}[\podpera\stt c] \ooo{c}{d} F\sur{\cycs{r}}{0}[\podpera\stt d]
\]
which says that $F$ preserves the Cardy
condition~(\ref{stale_jeste_v_Creswicku}). To see  it, use
that $F$, as a morphism of cyclic hybrids, commutes with
$\ooo ab$ and $\ooo cd$, and then  invoke the definition of the
$\circ$-operations in~$\oQnsOC$.
\end{remark}

\begin{proof}[Proof of Theorem 
\ref{PROCUnivProp}]
Let us consider a map 
$\alpha: \ModHybFun(\onsOC)/\, \EuScript{I} \to \oQnsOC$ given
by\footnote{We use the same notation for an element of
  $\ModHybFun(\onsOC)$ and its equivalence class. The meaning will
  always be clear from the context.}
\[
\sur{ \sur{\sfO_1}{g_1} \cdots \sur{\sfO_a}{g_a} 
}{\podpera g}[\zvedak C] \longmapsto \sur{\sfO_1 \cdots
  \sfO_a}{\velkyzvedak g+\sum_{i=1}^ag_i}
[\podpera C] \ 
\hbox { and } \
\sur{\varnothing}{g}[\podpera C] \mapsto \sur{\varnothing}{g}[\podpera C],
\]
and a map
$\beta: \oQnsOC \to  \ModHybFun(\onsOC) /\EuScript{I}$ given by
\[
\sur{\sfo_1 \cdots\, \sfo_b}{g}[\podpera C] \longmapsto
\sur{\sur{\sfo_1\cdots \, \sfo_b}{g}}{\podpera 0}[\podpera C]  \ 
\hbox { and } \
\sur{\varnothing}{g}[\podpera C] \longmapsto \sur{\varnothing}{g}[\podpera C].
\]
It is easy to check that $\alpha$ and $\beta$ are well-defined
morphisms of modular hybrids and that $\alpha\beta=\id$.
To verify $\beta\alpha=\id$, we have to check that
\begin{equation} 
\label{EQCardyOCZero}
\sur{\sur{\sfO_1}{g_1} \cdots \sur{\sfO_a}{g_a}}{\podpera g}[\zvedak C] 
= 
\sur{\sur{\sfO_1 \cdots\, \sfO_a}{\velkyzvedak g+\sum_{i=1}^a
    g_i}}{\zvedak 0}[\podpera C]
\ \hbox { in }\ \ModHybFun(\onsOC)/\, \EuScript{I}
\end{equation}
for $a\geq 1$; for $a=0$  is the claim trivial.
We start by showing that
\begin{equation} 
\label{EQCardyOCOne}
\sur{\sur{\cycs{p}}{0} \sur{\cycs{qr}}{0}}{\podpera 0}[\podpera C] 
= \sur{\sur{\cycs{p}\cycs{qr}}{0}}{\podpera 0}[\podpera C].
\end{equation}
To this end, we rewrite the left hand side  as
\begin{subequations}
\begin{equation}
\label{EQCardyOCOne1}
\sur{\sur{\cycs{p}}{0} \sur{\cycs{s'}}{0}}{\podpera 0}[\emptyset] \ooo{s'}{s''}
\sur{\sur{\cycs{qrs''}}{0}}{\podpera 0}[C]
\end{equation}
and apply the generating relation of $\EuScript{I}$ on the first term.
We obtain
\begin{equation}
\label{EQCardyOCOne2}
\sur{\sur{\cycs{p}\cycs{s'}}{0}}{\podpera 0}[\emptyset] \ooo{s'}{s''}
\sur{\sur{\cycs{qrs''}}{\podpera 0}}{0}[C]
\end{equation}
which equals the right hand side of~(\ref{EQCardyOCOne}).
\end{subequations}

As the second step of the proof we verify that
\begin{equation} 
\label{EQCardyOCTwo}
\sur{\sur{\sfO_1}{g_1} \sur{\sfO_2}{g_2} \cdots  
\sur{\sfO_a}{g_a}}{\mini g}[C] 
= \sur{\sur{\sfO_1\, \sfO_2}{g_1+g_2} \cdots \sur{\sfO_a}{g_a}}{\mini g}[C].
\end{equation}
Assume that $\sfO_1 = \sfo'_1 \sfo'_2 \cdots \sfo'_{b'}$,  
$\sfO_2 = \sfo''_1 \sfo''_2 \cdots \sfo''_{b''}$
and rewrite the left hand side as
\begin{equation}
\label{pisu_v_Koline}
\sur{\sur{\, p'\sfo'_1\, \sfo'_2\cdots\, \sfo'_{b'}}{g_1}}{\podpera 0}[\emptyset] 
\ooo{p'}{p''}
\sur{\sur{\, q'\sfo''_1\, \sfo''_2\cdots\, \sfo''_{b''}}{g_2}}{\podpera 0}[\emptyset] 
\ooo{q'}{q''}
\sur{\sur{\cycs{p''}}{\podpera 0} \sur{\cycs{q''r'}}{0}}{\podpera 0}[\emptyset] \ooo{r'}{r''}
\sur{\sur{\cycs{r''}}{\podpera 0}\sur{\sfO_3}{g_3} \cdots  \sur{\sfO_a}{g_a} }{g}[C].
\end{equation}
Applying~\eqref{EQCardyOCOne} to the third term, we get
\[
\sur{\sur{\, p'\sfo'_1\, \sfo'_2\cdots\, \sfo'_{b'}}{g_1}}{\podpera 0}[\emptyset] 
\ooo{p'}{p''}
\sur{\sur{\, q'\sfo''_1\, \sfo''_2\cdots\, \sfo''_{b''}}{g_2}}{\podpera 0}[\emptyset] 
\ooo{q'}{q''}
\sur{\sur{\cycs{p''}\cycs{q''r}}{0}}{\podpera 0}[\emptyset] \ooo{r'}{r''}
\sur{\sur{\cycs{r''}}{0}\sur{\sfO_3}{g_3} \cdots  \sur{\sfO_a}{g_a} }{\mini g}[C]
\]
which is easily seen to be the right hand side of
\eqref{EQCardyOCTwo}. Using~\eqref{EQCardyOCTwo} inductively we conclude
that the left hand side of~(\ref{EQCardyOCZero}) equals
\begin{equation}
\label{Vecer jdu do Pepikovy pokladny s Elizabetou.}
\sur{\sur{\sfO_1 \cdots\, \sfO_a}{\velkyzvedak
\sum_{i=1}^a g_i}}{\mini g}[\podpera C].
\end{equation}
The last step we need to prove that $\beta\alpha=\id$ is the equality
\begin{equation} 
\label{EQCardyOCThree}
\sur{\sur{\sfo_1\cdots\, \sfo_b}{g_1}}{g}[C] 
= \sur{\sur{\sfo_1\cdots\, \sfo_b}{g_1+1}}{g-1}[C].
\end{equation}
By the definition of the contractions in $\ModHybFun(\onsOC)$
its left hand side equals
\begin{equation}
\label{Jitka_mi_psala}
\oxi{p'p''}\sur{\sur{\cycs{p'}}{0}\sur{\, p'\sfo_1\,\sfo_2  \cdots\,
  \sfo_b}{g_1}}{\podpera g-1}[\zvedak C],
\end{equation}
which, by \eqref{EQCardyOCTwo}, is the same as
\[
\oxi{pp'}\sur{\sur{\cycs{p}\, p'\sfo_1\,\sfo_2  \cdots\,
  \sfo_b}{g_1}}{\podpera g-1}[\zvedak C],
\]
which is the right hand side
of~(\ref{EQCardyOCThree}). Applying~(\ref{EQCardyOCThree}) 
inductively, we see that~ (\ref{Vecer jdu do Pepikovy pokladny s
  Elizabetou.}) 
equals the right hand side of~(\ref{EQCardyOCZero}). 
This finishes the proof of $\beta\alpha=\id$ and
establishes~(\ref{Kdy se s tim prestanu trapit?}).

Let us prove the second part of the theorem. 
Denote by $\pi:\ModHybFun(\onsOC) \twoheadrightarrow \oQnsOC$ 
the natural projection and by
$\tF:\ModHybFun(\onsOC) \to \oH$ 
the unique 
extension of $F$ guaranteed by the universal property of the modular \envelope\.
Such an $\tF$ descends to $\hat{F}$ in the diagram
\[
\xymatrix@C4em@R3em{
\onsOC\ \ar@{^{(}->}[r] \ar_{F}[dr]  & 
\ModHybFun(\onsOC) \ar@{-->}^(.5){\tF}[d] \ar@{->>}^\pi[r] &
\oQnsOC  \ar@{-->}^{\hat{F}}[dl]
\\ &
\oH 
}
\]
if and only if  $\tF$ preserves the generating relation of
$\EuScript{I}$. But this is indeed so, since
\begin{align*}
\tF\sur{\sur{\cycs{q}}{0}\sur{\cycs{r}}{0}}{0}[\emptyset] 
 = F\sur{\cycs{q}\cycs{r}}{0}[\emptyset] 
= \oxi{uv}F\sur{\cycs{uqvr}}{0}[\emptyset]  
= \oxi{uv}\tF\sur{\sur{\cycs{uqvr}}{0}}{0}[\emptyset] 
= \tF\sur{\sur{\cycs{q}\cycs{r}}{0}}{0}[\emptyset]
\end{align*}
where the second equality used~(\ref{EQCardyOCStatement}). 
The uniqueness of $\hat{F}$ 
follows from the uniqueness of $\tF$ and the surjectivity of
$\pi$. This finishes the proof of Theorem \ref{PROCUnivProp}.
\end{proof}

\subsection{Modular \envelope\ of premodular hybrids.}

One has the functor 
\begin{equation}
\label{zitra s Martinem Bordermanem na obed}
\square_\pre: \ModHyb \to \PreHyb
\end{equation}
from the category of modular hybrids to the category of premodular hybrids
which
forgets all contractions in the `closed' color and  contractions 
$\oxi{uv}$ in the `open' color for which $u$ and $v$ belong to
{\em different\/} cycles. In this situation there is another version of the
modular \envelope\ functor, namely the left adjoint
\[
\ModPreFun_\pre : \PreHyb \to \ModHyb 
\]
to~(\ref{zitra s Martinem Bordermanem na obed}). We have 

\begin{theorem} 
\label{THMPositiveModEnvOC}
For the premodular hybrid $\onsOCpre$ from Example~\ref{h2} one has
the isomorphism
\begin{equation}
\label{Peter_Cushing}
\ModPreFun_\pre(\onsOCpre) \cong \oQnsOC
\end{equation}
of modular hybrids.
\end{theorem}

\begin{remark}
Notice that the Cardy condition~(\ref{stale_jeste_v_Creswicku}) is
already built in $\onsOCpre$, so we do not need to take
in~(\ref{Peter_Cushing}) the quitient by it.
\end{remark}

\begin{proof}[Proof of Theorem~\ref{THMPositiveModEnvOC}]
We need to verify that for a arbitrary modular hybrid
$\oH$ and for any morphism 
$F:\onsOCpre\to \square_\pre(\oH)$ of premodular hybrids, there is a unique morphism
$\tilde{F}:\oQnsOC\to \oH$ of modular hybrids
such that the diagram
\[
\xymatrix@C4em{
\onsOCpre\ \ar@{^{(}->}[r] \ar_{F}[dr]  & 
\oQnsOC \ar@{-->}^(.5){\hat{F}}[d]
\\ &
\oH 
}
\]
commutes.
Since $F$ is a morphism of premodular hybrids, 
it automatically satisfies relation~(\ref{EQCardyOCStatement}), because
\[
\oxi{uv}F\sur{\cycs{uqvr}}{0}[\emptyset] =
F\left(\oxi{pp'}\sur{\cycs{uqvr}}{0}[\emptyset]\right) =
F\sur{\cycs{q}\cycs{r}}{0}[\emptyset].
\]
If we forget the partially defined contractions in $\onsOCpre$, $F$ 
becomes a morphism of cyclic hybrids so it extends, 
by Theorem~\ref{PROCUnivProp}, into a unique morphism $\tF$ of modular
hybrids that makes the above diagram commutative.
\end{proof}

\section{Modular \envelope\ of a suboperad}
\label{Dnes_se_pojedu_podivat_do_Ribeauville}

The central technical result of this article,
Proposition~\ref{THMModOCus} of the previous section,  
describes the modular \envelope\
$\ModHybFun(\onsOC)$ of the modular hybrid $\onsOC$. We need
a similar result also for the \KaPe\ and stable subhybrids
\begin{equation}
\label{vcera_jsem_jel_na_brsulickach}
\onsOC_\KP \hookrightarrow \onsOCst \hookrightarrow \oOC, 
\end{equation}
but we  do not want to repeat the long technical proof of
Proposition~\ref{THMModOCus} for them.
We prove instead that the morphisms
\[
\ModHybFun(\onsOC_\KP) 
\to \ModHybFun(\onsOCst) \to \ModHybFun(\onsOC)
\]
of modular hybrids induced by the 
inclusions~(\ref{vcera_jsem_jel_na_brsulickach}) are injective and 
describe explicitly
the modular \envelope{s} $\ModHybFun(\onsOC_\KP)$ and $
\ModHybFun(\onsOCst)$ as subhybrids of $\ModHybFun(\onsOC)$.

The content of this section will therefore be some results about the
induced maps between modular \envelope{s}.   
To save the reader from unnecessary technicalities, we formulate and prove them
only for the `classical' cyclic operads and the `classical' modular
\envelope\ functor \hbox{$\Mod : \CycOp \to \ModOp$}
of~\cite[page~382]{markl:la}. 
It will be clear that obvious analogs of these results hold also for
non-$\Sigma$ cyclic operads and cyclic hybrids.

Let thus $\oC$ be a cyclic operad and $\oB \subset \oC$ its
cyclic suboperad. We are going to investigate the induced map
$\varpi : \Mod(\oB) \to \Mod(\oC)$. The following example shows that,
in some situations, $\varpi$ need not be a monomorphism.

\begin{example}
\label{Chtel_jsem_jet_na_bruslicky_ale_prsi.}
Let $\oC$ be the free cyclic operad generated by the two-point set 
\[
\big\{(u,v),(v,u)\big\} \subset
\oC\big(\stt {u,v}\big)
\] 
with the obvious action of the group
$\Aut\big(\stt {u,v}\big)$.
Denote by and $\oB \subset
\oC$ the cyclic suboperad consisting of $\circ$-compositions of at
least two elements of $\oC$. Let~finally 
\[
a := (u,x')\ooo {x'}{x''}
(v,x'') \in \oB\big(\stt {u,v}\big)
\ \hbox { and } \ b := (x',u)\ooo{x'}{x''}(x'',v) \in \oB\big(\stt {u,v}\big).
\] 
It follows from the axioms of
cyclic operads that
\begin{align*}
\oxi{uv}(a)& = \oxi{uv}\big( (u,x')\ooo {x'}{x''} (v,x'')\big) =  
\oxi{x'x''}\big( (u,x')\ooo {u}{v} (v,x'')\big)
\\
& \hskip 3em  = \oxi{uv}\big((x',u) \ooo{x'}{x''}(x'',v)\big) = \oxi{uv}(b)
\end{align*}
in $\Mod(\oC)$, 
while it is simple to check that $\oxi{uv}(a)\not = \oxi{uv}(b)$
in $\Mod(\oB)$. So the induced map $\varpi : \Mod(\oB) \to \Mod(\oC)$ is not a
monomorphism.

The main idea of the example can be illustrated as follows. Represent the
generator of $\oC$ by the arrow
\begin{center}
\scalebox{1}
{
\begin{pspicture}(0,-0.18296875)(2.6428125,0.1296875)
\psline[linewidth=0.04cm,arrowsize=0.1291667cm 2.0,arrowlength=1.4,arrowinset=0]{<-}(0.2809375,-0.11546875)(1.7809376,-0.11546875)
\rput(0.1,-0.11){$u$}
\rput(2,-0.11){$v$}
\end{pspicture} 
}
\end{center}
pointing from $v$ to $u$. In this graphical representation,
\begin{center}
\scalebox{1} 
{
\begin{pspicture}(0,-0.23296875)(9.942813,0.23296875)
\psline[linewidth=0.04cm,arrowsize=0.1291667cm 2.0,arrowlength=1.4,arrowinset=0]{<-}(1.6809375,-0.0)(3.1809375,-0.0)
\psline[linewidth=0.04cm,arrowsize=0.1291667cm 2.0,arrowlength=1.4,arrowinset=0]{<-}(3.6809375,-0.0)(2.0809374,-0.0)
\rput(1.437,0){$u$}
\rput[b](3.9,-0.2){$\ v\ ,$}
\rput[b](0.42234376,-0.1){$a = $}
\rput[b](5.5,-0.1){$b = $}
\rput(-.5,0){
\rput(6.9,-0){$u$}
\rput[b](9.45,-.20){$\ v\ ,$}
\psline[linewidth=0.04cm,arrowsize=0.1291667cm 2.0,arrowlength=1.4,arrowinset=0]{<-}(8.180938,-0)(7.1809373,-0)
\psline[linewidth=0.04cm,arrowsize=0.1291667cm
  2.0,arrowlength=1.4,arrowinset=0]{<-}(8.180938,-0)(9.180938,-0)
}
\end{pspicture} 
}
\end{center}
so we have in $\Mod(\oB)$
\begin{center}
\scalebox{1}
{
\begin{pspicture}(0,-0.07)(11.300938,1.25)
\psline[linewidth=0.04cm,arrowsize=0.1291667cm 2.0,arrowlength=1.4,arrowinset=0]{<-}(2.4809375,0.95)(3.9809375,0.95)
\psline[linewidth=0.04cm,arrowsize=0.1291667cm 2.0,arrowlength=1.4,arrowinset=0]{<-}(4.4809375,0.95)(2.8809376,0.95)
\psline[linewidth=0.04cm,arrowsize=0.1291667cm 2.0,arrowlength=1.4,arrowinset=0]{<-}(9.680938,0.95)(8.680938,0.95)
\psline[linewidth=0.04cm,arrowsize=0.1291667cm 2.0,arrowlength=1.4,arrowinset=0]{<-}(9.680938,0.95)(10.680938,0.95)
\psellipse[linewidth=0.03,linestyle=dashed,dash=0.16cm 0.16cm,dimen=outer](9.680938,0.95)(1.1,0.3)
\psellipse[linewidth=0.03,linestyle=dashed,dash=0.16cm 0.16cm,dimen=outer](3.4809375,0.95)(1.0,0.3)
\psarc[linewidth=0.04](4.4809375,0.45){0.5}{270.0}{90.0}
\psarc[linewidth=0.04](10.780937,0.45){0.5}{270.0}{90.0}
\psarc[linewidth=0.04](8.580937,0.45){0.5}{90.0}{270.0}
\psarc[linewidth=0.04](2.4809375,0.45){0.5}{90.0}{270.0}
\psline[linewidth=0.04cm](2.4809375,-0.05)(4.4809375,-0.05)
\psline[linewidth=0.04cm](8.580937,-0.05)(10.780937,-0.05)
\rput(0.8,0.46){$\circ_{uv}(a)=$}
\rput[t](5.2,0.46){,}
\rput(11.5,0.46){.}
\rput(6.913906,0.46){$\circ_{uv}(b)=$}
\end{pspicture} 
}
\end{center}
The dashed ovals indicate that the arrows representing the generators
cannot be separated in $\Mod(\oB)$. The ovals however
can be erased in $\Mod(\oC)$ and the arrowheads 
moved around bringing both pictures in the above display into
\begin{center}
\scalebox{1} 
{
\begin{pspicture}(0,0.08)(2.72,1.12)
\psline[linewidth=0.04cm,arrowsize=0.1291667cm 2.0,arrowlength=1.4,arrowinset=0]{<-}(1.3,1.1)(2.2,1.1)
\psarc[linewidth=0.04](2.2,0.6){0.5}{270.0}{90.0}
\psarc[linewidth=0.04](0.5,0.6){0.5}{90.0}{270.0}
\psline[linewidth=0.04cm](0.4,0.1)(1.5,0.1)
\psline[linewidth=0.04cm,arrowsize=0.1291667cm 2.0,arrowlength=1.4,arrowinset=0]{<-}(1.3,0.1)(2.2,0.1)
\psline[linewidth=0.04cm](0.4,1.1)(1.5,1.1)
\rput(3,0.6){.}
\end{pspicture} 
}
\end{center}
\end{example}

The central technical result of this section reads

\begin{proposition} 
\label{THMSuffCondOnInjectivity}
Let $\oB\subset\oC$ be cyclic operads. Assume that
for every $w' \in\oC\big(\stt{p',q'} \sqcup R\big)$ and 
$w''\in\oC\big(\stt{p'',q''} \sqcup
S\big)$ such that $w'\ooo{q'}{q''}w''\in\oB\big(\stt{p',p''} \sqcup R \sqcup
S\big)$ either 
\begin{itemize}
\item[(i)]
there is a bijection $\rho: \stt{p',p''} \sqcup R \sqcup S \to 
\stt{q',q''} \sqcup R \sqcup S$ fixing $R \sqcup S$ such that 
\[
w'\ooo{p'}{p''}w'' = \rho(w'\ooo{q'}{q''}w''),
\]
\item[(ii)]
 or there are $w'_1\in\oB(\stt{p',q'} \sqcup R)$ 
and $w''_1\in\oB(\stt{p'',q''} \sqcup S)$ 
such that 
\[
w'\ooo{p'}{p''}w'' = w'_1\ooo{p'}{p''}w''_1 \
\hbox { and } \ w' \ooo{q'}{q''}w'' = w'_1\ooo{q'}{q''}w''_1.
\]
\end{itemize}
Then the induced map $\varpi : \Mod(\oB) \to \Mod(\oC)$ is injective.
\end{proposition}

The assumption of Proposition~\ref{THMSuffCondOnInjectivity} is in
Example~\ref{Chtel_jsem_jet_na_bruslicky_ale_prsi.} violated by
$w' := (p',q') \in \oC\big( \{p',q'\} \big)$ and 
$w'' := (p'',q'') \in \oC\big( \{p'',q''\} \big)$, $S = R := \emptyset$.
Proposition~\ref{THMSuffCondOnInjectivity} will follow from 
Proposition~\ref{PROOperationsInModEnv} whose formulation and proof we
postpone to the end of this section. We will also need

\begin{definition}
Let $\oB$ be a cyclic suboperad of a modular operad $\oC$.
The \emph{$\xi$-closure} of $\oB$ in $\oC$ is defined as
\[
\xi_{\oC}(\oB) := \big\{\oxi{p'_1p''_1}\cdots\oxi{p'_np''_n}(x) \in \oC\ \big|\ 
n\in\bbN,\ x\in\oB,\ p'_1,p''_1,\ldots,p'_n,p''_n
\textrm{ are some inputs of }x \big\}.
\]
\end{definition}

The terminology 
is inspired by the old-fashioned notation $\xi_{uv}$ for $\oxi{uv}$.
Notice that $\xi_{\oC}(\oB)$ 
is the smallest modular suboperad of $\oC$ containing $\oB$, so it is
a modular \envelope\ of $\oB$ {\em in\/}~$\oC$ or {\em relative\/} to
$\oC$. From this point of view, $\Mod(\oB)$ is the {\em absolute\/}
modular completion of $\oB$. 
The following statement describes a situation when absolute and
relative completions~agree.

\begin{proposition} 
\label{CorConseqInjOnMods}
If $\oB\subset\oC$ are cyclic operads such 
that the map $\varpi : 
\Mod(\oB)\to\Mod(\oC)$ is injective, then 
\[
\Mod(\oB)
\cong  \xi_{\Mod(\oC)}(\oB) .
\]
In particular, $\Mod(\oB)
\cong  \xi_{\Mod(\oB)}(\oB)$.
\end{proposition}

\begin{proof}
By the universal property of 
$\Mod(\oB)$ applied to the inclusions $\oB \hookrightarrow
\xi_{\Mod(\oC)}(\oB)$ 
and $\oB\hookrightarrow \Mod(\oC)$, there is a modular operad morphism 
$i:\Mod(\oB)\to  \xi_{\Mod(\oC)}(\oB)$ such that the~diagram
\[
\xymatrix@C1em{\Mod(\oB)  \ar[rr]^{\varpi}\ar^i[rd] &&  \Mod(\oC)
\\
&\ \xi_{\Mod(\oC)}(\oB) \  \ar@{^{(}->}[ru]  &
}
\]
commutes.
Since $\varpi: 
\Mod(\oB) \to \Mod(\oC)$ is injective by assumption, so is $i$.
As $\xi_{\Mod(\oC)}(\oB)$ is the smallest  modular suboperad 
of $\Mod(\oC)$ containing $\oB$, $i$ must be an isomorphism. The
second isomorphism of the proposition is the particular case when
$\varpi$ is the identity morphism $\id : \oB \to \oB$.
\end{proof}

Proposition~\ref{THMSuffCondOnInjectivity} is a consequence of

\begin{proposition} 
\label{PROOperationsInModEnv}
Every element $x\in\Mod(\oC)$ in the modular \envelope\ of a cyclic operad $\oC$ 
is of the~form 
\begin{equation}
\label{Vcera_na_dvou_koncertech_a_kdyz_nebude_prset_pujdu_dnes_zase.}
x=\oxi{p'_1p_1''}\cdots\oxi{p'_np_n''} (y)
\end{equation}
where $y\in\oC$ and $ p'_1,p''_1,\ldots,p'_n,p''_n$, $n\in \bbN$, 
are some of its inputs.
On elements in this form, consider the following `moves:'
\begin{enumerate}
\item[(i)] 
Let $w'\in\oC\big(\stt{p',q'} \sqcup R\big)$ 
and $w''\in\oC\big(\stt{p'',q''} \sqcup S\big)$.
Then replace 
\[
\oxi{p'_1p_1''}\cdots\oxi{p'_{n-1}p''_{n-1}}\oxi{p'p''}(w'\ooo{q'}{q''}w'')
\ \hbox { by } \
\oxi{p'_1p_1''}\cdots\oxi{p'_{n-1}p''_{n-1}}\oxi{q'q''}(w'\ooo{p'}{p''}w'').
\]
\item[(ii)] 
Let  $p',p'', q',q''$ be some of the inputs of $y$ and $\rho$ a bijection
mapping $p',p''$ to $q',q''$ in this order which restricts to the
identity on the remaining inputs of $y$.  
Then replace 
\[
\oxi{p_1'p_1''}\cdots\oxi{p_{n-1}'p''_{n-1}}\oxi{p'p''}(y) 
\ \hbox { by } \
\oxi{p_1'p_1''}\cdots\oxi{p_{n-1}'p''_{n-1}}\oxi{q'q''} (\rho y) .
\]
\item[(iii)]  
For an arbitrary permutation $\sigma \in \Sigma_n$ replace
\[
\oxi{p_1'p_1''}\cdots\oxi{p'_np_n''}(y) 
\ \hbox { by } \
\oxi{p'_{\sigma(1)}p''_{\sigma(1)}}\cdots\oxi{p'_{\sigma(n)}p''_{\sigma(n)}}(y).
\]
\end{enumerate}
Two 
expressions~(\ref{Vcera_na_dvou_koncertech_a_kdyz_nebude_prset_pujdu_dnes_zase.})
represent the same element of
$\Mod(\oC)$ if and only if they
are related by a finite numbers of the above moves.
\end{proposition}

\begin{proof}
The modular \envelope\ $\Mod(\oC)$ is isomorphic to the quotient
$\Free(\oC)/\sim$, 
where $\Free(\oC)$ is the free modular operad generated by $\oC$ and $\sim$
is the equivalence that identifies $\circ$-operations inside $\oC$ with the
formal ones in $\Free(\oC)$.

\setcounter{footnote}{0}

As explained e.g.~in~\cite[II.1.9]{markl-shnider-stasheff:book}, 
$\Free(\oC)$ can be constructed as
an explicit colimit whose elements are represented by decorated
graphs. Since we are working in \Set, 
every  $x\in\Free(\oC)$ has well-defined underlying
graph $G(x)$. 
Choose a contractible, not necessary connected, subgraph $T$
in $G(x)$ and contract $x$
along $T$ using the cyclic operad structure of $\oC$. Denote the result
by $C_T(x)$; clearly
\[
C_T(x) \sim x.
\] 
If $T$ is in particular a maximal subtree of $G(x)$, 
then the underlying graph of $C_T(x)$ 
has one vertex, call such a  graph a \broucek.\footnote{Czech for little
  beetle.} The element $C_T(x)$ is obtained by iterated
contractions of some $y\in \oC$. To describe it in such a way
explicitly, i.e.\ as 
\[
\oxi{p'_1p''_1} \cdots \oxi{p'_np''_n} (y) ,\  n \in \bbN,
\]  
with some specific symbols $p'_1,p''_1, \ldots p'_n,p''_n$, one needs
to label the half-edges of \broucek\ and choose their order. The ambiguity
of these choices is reflected by moves (ii) and (iii) of Proposition
\ref{PROOperationsInModEnv}.  

Another ambiguity comes from different
choices of a maximal subtree of $G(x)$. Let us analyze this situation.
Assume that $T_1$ and $T_2$ are different maximal subtrees of
$G(x)$
By~\cite[Chapter~6]{ore}, $T_1$ and $T_2$ are
related by a `singular cyclic interchange.' This means that there
exists a~subgraph $H \subset G(x)$ with precisely one cycle, and two
edges $e_1,e_2$ belonging to this cycle, such that 
\[
H \setminus\{ e_2\} = T_1 \ \hbox { and } H \setminus \{e_1\} = T_2.
\] 
In this situation, $H \setminus \{e_1,e_2\}$ is the disjoint union
of two (non-maximal) trees $U$ and $V$. Let $z := C_{U\sqcup V}
(x)$. Obviously, $G(z)$ is a graph with two vertices decorated by some
$a,b \in \oC$.
Let $u_1,\ldots,u_k$ are the edges of $G(x)$ that do not belong to
$H\setminus \{e_1,e_2\}$. 
We then have, due to the interchange law between contractions and
$\circ$-operations,
\[
z = \oxi{u_1} \cdots \oxi{u_k} \oxi{e_1} (a \circ_{e_2} b) = 
\oxi{u_1} \cdots \oxi{u_k} \oxi{e_2} (a \circ_{e_1} b)
\]
modulo the relations defining $\Mod(\oC)$. In the above display, $\oxi
e$ denotes the contraction along~$e$. Finally, we observe that 
\[
\oxi{u_1} \cdots \oxi{u_k} \oxi{e_1} (a \circ_{e_2} b) \ \hbox {
  represents } C_{T_1}(x)
\]
while 
\[
\oxi{u_1} \cdots \oxi{u_k} \oxi{e_2} (a \circ_{e_1} b) \ \hbox {
  represents } C_{T_2}(x).
\]
Equality $C_{T_1}(x) =C_{T_2}(x)$ is therefore reflected by
move (i) of Proposition \ref{PROOperationsInModEnv}. 
\end{proof}

The above proof shows that move~(i) is the relevant one, the remaining
moves only account for different choices of labels.

\begin{proof}[Proof of Proposition~\ref{THMSuffCondOnInjectivity}]
Recall that each element of $\Mod(\oB)$ is of the
form~(\ref{Vcera_na_dvou_koncertech_a_kdyz_nebude_prset_pujdu_dnes_zase.}).  
So assume that 
$y,z\in    \oB$ and  that 
\begin{gather} 
\label{EQEqualityOfTwoModEnv}
\oxi{p'_1p_1''}\cdots\oxi{p'_np_n''} (y) =
\oxi{q'_1q_1''}\cdots\oxi{q'_nq_n''} (z) \hbox { in } \Mod(\oC). 
\end{gather}
All we need is to show that the same equality holds also in $\Mod(\oB)$.
By
Proposition~\ref{PROOperationsInModEnv},~(\ref{EQEqualityOfTwoModEnv})
holds if and only if there is a finite sequence of moves (i)--(iii) transforming
its left hand side into its right hand side.  
Each move is a replacement of the form
\begin{gather} \label{EQCertainOperationNamed}
\oxi{r'_1r_1''}\cdots\oxi{r'_nr_n''} (u) 
\quad \longmapsto \quad \oxi{s_1's_1''}\cdots\oxi{s_n's_n''}( v)
\end{gather}
with some $u,v \in \oC$. The proof will thus be finished if we
show that $u\in \oB$ in~(\ref{EQCertainOperationNamed}) implies that 
\hbox{$v\in \oB$}.

This is obvious if \eqref{EQCertainOperationNamed} is move (ii) or
(iii). Let us analyze move~(i), that is, see what happens if we replace
\begin{gather} \label{EQOperationOne}
\oxi{p'_1p_1''}\cdots\oxi{p'_{n-1}p''_{n-1}}\oxi{p'p''}(u)
\ \hbox { by } \
\oxi{p'_1p_1''}\cdots\oxi{p'_{n-1}p''_{n-1}}\oxi{q'q''}(v),
\end{gather}
where $u = w'\ooo{q'}{q''}w''$ and $v = w'\ooo{p'}{p''}w''$
with some 
\[
w'\in\oC\big(\stt{p',q'} \sqcup R\big) \ \hbox { and } \
w''\in\oC\big(\stt{p'',q''} \sqcup S\big)
\]
such that $u = w'\ooo{q'}{q''}w'' \in \oB\big(\stt{p',p''} \sqcup R \sqcup
S\big)$.

In case (i) of Proposition~\ref{THMSuffCondOnInjectivity},
$w'\ooo{p'}{p''}w'=\rho(w\ooo{q'}{q''}w')$, i.e.\ $v = \rho(u) 
\in  \oB\big(\stt{q',q''} \sqcup R \sqcup
S\big)$  as required. 
In case (ii),
$v =  w'_1\ooo{p'}{p''}w''_1$ for some  $w'_1\in\oB(\stt{p',q'} \sqcup R)$ 
and $w''_1\in\oB(\stt{p'',q''} \sqcup S)$,
thus again  $v \in  \oB\big(\stt{q',q''} \sqcup R \sqcup
S\big)$ and we are done as well.
\end{proof}

\section{Modular \envelope{s} of the stable and 
Kaufmann-Penner parts}
\label{sec:stable-kaufm-penn}

In this section we use the results of
Section~\ref{Dnes_se_pojedu_podivat_do_Ribeauville} and derive the stable
and Kaufmann-Penner versions of Theorem~\ref{PROCUnivProp}.

\subsection{Stable version.} 
Proposition~\ref{THMModOCStable} below guarrantees
that one may use Proposition~\ref{THMModOCus} to describe 
explicitly the modular
\envelope\ of the stable part $\onsOCst \subset \onsOC$, as done in
Remark~\ref{vcera_v_Eppingu}. 
The main result of this
subsection is Theorem~\ref{PROOCUnivPropST}.

\begin{lemma} \label{LEMModCommWithSt}
Let $\oC_\St$ be the stable part of a cyclic operad $\oC$ as in
Definition~\ref{s_Jarkou_k_Zabackovi}. Then one has the isomorphism
\[
\xi_{\Mod(\oC)}(\oC_\St)  \cong \Mod(\oC)_\St.
\]
Consequently, if the induced map $\varpi :
\oMod{\oC_\St}\to\oMod{\oC}$ is a monomorphism, then
\[
\Mod(\oC_\St) \cong \Mod(\oC)_\St.
\]
\end{lemma}

\begin{proof}
By Proposition~\ref{PROOperationsInModEnv}, 
every $x\in\oMod{\oC}$ is of the form 
\begin{gather} 
\label{EQCanFormOfEltOfModEnv}
x=\oxi{p'_1p_1''}\cdots\oxi{p'_np_n''} (y)
\end{gather}
for some $y\in\oC$ and $n\in \bbN$. 
If $x\in \xi_{\Mod(\oC)}(\oC_\St)$, we may assume that
$y\in\oC_\St\subset \Mod(\oC)_\St$. Since contractions preserve 
stable parts of modular
operads  by Lemma~\ref{Prijedu_zase_pristi_rok?}, 
$x\in \Mod(\oC)_\St$, which shows  that
\[
\xi_{\Mod(\oC)}(\oC_\St)  \subset \Mod(\oC)_\St.
\]
If $x\in \Mod(\oC)_\St$, then $y \in \Mod(\oC)_\St \cap \oC=\oC_\St$
by the second part of  Lemma~\ref{Prijedu_zase_pristi_rok?}, 
hence $x\in \xi_{\Mod(\oC)}(\oC_\St)$,
thus 
\[
\Mod(\oC)_\St \subset  \xi_{\Mod(\oC)}(\oC_\St).
\]
Having this established, 
the second part of the lemma follows from 
Proposition~\ref{CorConseqInjOnMods}.
\end{proof}

Let us turn our attention to the cyclic hybrid $\onsOC$ from
Example~\ref{h1} and its stable  version $\onsOCst$ analyzed in
Example~\ref{Dnes_naposledy_v_Myluzach_na_varhanky.}.

\begin{proposition} 
\label{THMModOCStable}
One has an isomorphism
$\ModHybFun(\onsOCst) \cong \ModHybFun(\onsOC)_\St$.
\end{proposition}

\begin{remark}
\label{vcera_v_Eppingu}
An explicit description of $\oMod{\onsOC}_\St$ and therefore, by
Proposition~\ref{THMModOCStable}, also of $\ModHybFun(\onsOCst)$,
is provided by imposing the stability assumption on the expressions 
in~(\ref{v_pondeli_do_Leicesteru}) of Proposition~\ref{THMModOCus}. Explicitly, the symbol
\begin{equation}
\label{zas_problemy_s_Jarkou}
\sur{ \sur{\sfO_1}{g_1} \cdots \sur{\sfO_a}{g_a}}{\mini g}[\zvedak C]
\end{equation}
is stable if and only if 
\begin{equation}
\label{eq:2}
\textstyle
4\big(g + \sum_{i=1}^a g_i\big) + 2b + 2|C| + |O| >4
\end{equation}
where $b := \sum_{i=1}^a b_i$ is the total number of cycles in
$\sfO_1,\ldots,\sfO_a$. 
\end{remark}

\begin{proof}[Proof of Proposition~\ref{THMModOCStable}]
We verify that the inclusion $\onsOCst \subset \onsOC$ 
satisfies condition (i) of Proposition~\ref{THMSuffCondOnInjectivity}. 
The proof will then follow from Proposition~\ref{CorConseqInjOnMods} and 
Lemma~\ref{LEMModCommWithSt}.

Let $w',w''\in\onsOC$  be such that $w'\ooo{q'}{q''}w''\in\onsOC_\St$.
It easily follows from 
the definition of the stable part if $w',w''\not \in \onsOC_\St$, 
then also $w'\ooo{q'}{q''}w''\not\in\onsOC$, so we may assume e.g.\
that $w'\in\onsOC_\St$ while $w''\not\in\onsOC_\St$.
Since $w''$ has  
at least $2$ inputs, according to~\eqref{REMListOfUnstables} it must be
\[
\hbox { either } \,
\R3 = \sur{\cycs{p''q''}}{0}[\emptyset] 
\hbox { or }\,  \R6 =  \sur{\varnothing}{0}[\stt{p'',q''}].
\]
In both cases, $w'\ooo{q'}{q''}w''$ `replaces $q'$ by $p''$'
and  $w'\ooo{p'}{p''}w''$ `replaces $p'$ by $q''$' in $w'$.     

Let us explain what we mean by this when $w'' = \, \R6$.
Then $p',p'',q'$ and $q''$ must be `closed' inputs and one has, by definition,
\[
\sur{\sfO}{g}[R \sqcup \stt{p',q'}] 
\ooo{q'}{q''} \sur{\varnothing}{0}[\stt{p'',q''}] = 
\sur{\sfO}{g}[R \sqcup \stt{p',p''}] \ \hbox { and } \ 
\sur{\sfO}{g}[R \sqcup \stt{p',q'}] 
\ooo{p'}{p''} \sur{\varnothing}{0}[\stt{p'',q''}] = 
\sur{\sfO}{g}[R \sqcup \stt{q',q''}]. 
\]
Clearly
$w'\ooo{q'}{q''}w'=\rho(w'\ooo{p'}{p''}w'')$ 
for a bijection $\rho$ mapping $\stt{p',p''}$ to $\stt{q',q''}$ 
and restricting to 
the identity on $R$. The case when $w'' = \, \R3$ can be discussed
similarly. We leave the details to the reader.
\end{proof}

We have the following stable analog of Theorem~\ref{PROCUnivProp}.

\begin{theorem} 
\label{PROOCUnivPropST}
Let\/ $\EuScript{I}$ be the 
ideal in the stable modular hybrid $\ModHybFun(\onsOC_\St) \cong
\ModHybFun(\onsOC)_\St$  
generated by the relations
\begin{equation}
\label{Zitra_uz_jedu_z_Meduz_do_Prahy.}
\sur{ \sur{\cycs{q}}{0} \sur{\cycs{r}}{0}}{\zvedak0}[\emptyset] = \sur{
  \sur{\cycs{q}\cycs{r}}{0} }{\zvedak 0}[\emptyset] \quad \textrm{and}\quad
\sur{ \sur{\cycs{q}}{0}\sur{\cycs{}}{0} }{\zvedak0}[\emptyset] = \sur{
  \sur{\cycs{q}\cycs{}}{0} }{\zvedak 0}[\emptyset].
\end{equation}
Then one has an isomorphism of stable modular hybrids
\[
\oQnsOCst \ \cong\ \ModHybFun(\onsOC_\St)/ \, \EuScript{I}.
\]
Therefore for any, not necessarily stable, modular hybrid $\oH$ and a
morphism $F:\onsOCst\to\oH$ of cyclic hybrids satisfying 
\[
\oxi{uv}F\sur{\cycs{uqvr}}{0}[\emptyset] =
F\sur{\cycs{q}\cycs{r}}{0}[\emptyset]  \quad \textrm{and}\quad
\oxi{uv}F\sur{\cycs{uqv}}{0}[\emptyset] =
F\sur{\cycs{q}\cycs{}}{0}[\emptyset] 
\]
there is a unique morphism $\hat{F}:\oQnsOCst\to\oH$ of modular hybrids
making the diagram 
\[
\xymatrix@C4em{
\onsOCst\ \ar@{^{(}->}[r] \ar_{F}[dr]  & 
\oQnsOCst \ar@{-->}^(.5){\hat{F}}[d]
\\ &
\oH 
}
\]
commutative.
\end{theorem}

The reader may wonder why we have two relations
in~(\ref{Zitra_uz_jedu_z_Meduz_do_Prahy.}) while the `unstable'
Theorem~\ref{PROCUnivProp} has only one. The explanation is that, in
the unstable case, the second relation
in~(\ref{Zitra_uz_jedu_z_Meduz_do_Prahy.}) is the same as
\[
\sur{ \sur{\cycs{q}}{0} \sur{\cycs{r}}{0}}{\zvedak0}[\emptyset]  
\ooo rs \sur{ \sur{\cycs{s}}{0}}{\zvedak 0}[\emptyset] = \sur{
  \sur{\cycs{q}\cycs{r}}{0} }{\zvedak 0}[\emptyset]
\ooo rs \sur{ \sur{\cycs{s}}{0}}{\zvedak 0}[\emptyset] ,
\]
so it belongs to the ideal generated by the first relation.
Since 
\[
\sur{ \sur{\cycs{s}}{0}}{\zvedak 0}[\emptyset]
\]
is not stable, the same reasoning does not apply to $\ModHybFun(\onsOC_\St)$.

\begin{proof}[Proof of Theorem~\ref{PROOCUnivPropST}]
The proof is a modification of the proof of
Theorem~\ref{PROCUnivProp} so we mention only the differences. First of all, 
in addition to \eqref{EQCardyOCOne}, we also need to prove that
\[
\sur{\sur{\cycs{}}{0} \sur{\cycs{qr}}{0}}{\podpera 0}[\podpera C] 
= \sur{\sur{\cycs{}\cycs{qr}}{0}}{\podpera 0}[\podpera C]
\]
modulo $\EuScript{I}$. This equality can easily be obtained by
replacing, in~(\ref{EQCardyOCOne1}) and~(\ref{EQCardyOCOne2}), $\cycs{p}$
by  $\cycs{}$.

It might also happen that some terms in~(\ref{pisu_v_Koline}) which we
used to prove~(\ref{EQCardyOCTwo}) are unstable. Let us denote the
terms constituting~(\ref{pisu_v_Koline}) by
\begin{align*}
\Rrom A := 
\sur{\sur{\, p'\sfo'_1\, \sfo'_2\cdots\,
    \sfo'_{b'}}{g_1}}{\podpera 0}[\emptyset],\
\Rrom B& := 
\sur{\sur{\, q'\sfo''_1\, \sfo''_2\cdots\,
  \sfo''_{b''}}{g_2}}{\podpera 0}[\emptyset] ,\
\Rrom C := 
\sur{\sur{\cycs{p''}}{\podpera 0} \sur{\cycs{q''r'}}{0}}{\podpera
  0}[\emptyset] 
\\
\hbox { and } \ \Rrom D& :=
\sur{\sur{\cycs{r''}}{\podpera 0}\sur{\sfO_3}{g_3} \cdots  \sur{\sfO_a}{g_a} }{g}[C].
\end{align*}
Term\, $\Rrom C$ is always stable. Term\, $\Rrom A$ is unstable if and only if
$g_1=0$ and
$\sfO_1' = \sfo_1' = \cycs{p'}$ or $\cycs{p's}$
for some symbol $s$, in which case
\[
\Rrom A := 
\sur{\sur{\cycs{p'}}{0}}{\podpera 0}[\emptyset] \ \hbox { or } \
\sur{\sur{\cycs{p's}}{0}}{\podpera 0}[\emptyset].
\]
Likewise, $\Rrom B$ is unstable if and only if
\[
\Rrom B := 
\sur{\sur{\cycs{q'}}{0}}{\podpera 0}[\emptyset] \ \hbox { or } \
\sur{\sur{\cycs{q't}}{0}}{\podpera 0}[\emptyset]
\]
for a symbol $t$.
Finally, $\Rrom D$ is unstable if and only if $a = 2$, in which case
\[
\Rrom D := 
\sur{\sur{\cycs{r''}}{0}}{\podpera 0}[\emptyset].
\]
Let us analyze all possible situations.

\noindent 
{\it Term\, $\Rrom A$ is unstable but\, $\Rrom B$ is stable.} The left hand side
of~(\ref{EQCardyOCTwo}) takes the form
\[
\sur{\sur{\sfo_1'}{0} \sur{\sfO_2}{0} \cdots  
\sur{\sfO_a}{g_a}}{\mini g}[C] 
\] 
with $\sfo'_1 = \cycs{}$ or  $\sfo_1 = \cycs{s}$.
We then instead of~(\ref{pisu_v_Koline}) take
\[
\sur{\sur{\, q'\sfo''_1\, \sfo''_2\cdots\, \sfo''_{b''}}{g_2}}{\podpera 0}[\emptyset] 
\ooo{q'}{q''}
\sur{\sur{\sfo'_1}{0} \sur{\cycs{q''r'}}{0}}{\podpera 0}[\emptyset] \ooo{r'}{r''}
\sur{\sur{\cycs{r''}}{\podpera 0}\sur{\sfO_3}{g_3} \cdots  \sur{\sfO_a}{g_a} }{g}[C]
\]
and proceed as before. The situation when\, $\Rrom B$ is unstable but\, $\Rrom
A$ is stable is similar.

\noindent 
{\it Both\, $\Rrom A$ and\, $\Rrom B$ are unstable.}
The  left hand side
of~(\ref{EQCardyOCTwo}) is of the form
\[
\sur{\sur{\sfo_1'}{0} \sur{\sfo''_1}{0} \cdots  
\sur{\sfO_a}{g_a}}{\mini g}[\zvedak C] 
\]
where $\sfo'_1 = \cycs{}$ or  $\cycs{s}$ and $\sfo''_1 
= \cycs{}$ or  $\cycs{t}$.
We then instead of~(\ref{pisu_v_Koline}) take
\[
\sur{\sur{\sfo'_1}{\podpera 0} 
\sur{\, \sfo_1''r'} {\podpera 0}}{\podpera 0}[\zvedak \emptyset] \ooo{r'}{r''}
\sur{\sur{\cycs{r''}}{\podpera 0}\sur{\sfO_3}{g_3} 
\cdots  \sur{\sfO_a}{g_a} }{g}[\zvedak C]
\]
and proceed as in the proof of
Theorem~\ref{PROCUnivProp}.

\noindent 
{\it Term\, $\Rrom D$ is unstable.} 
 Then the left hand side
of~(\ref{EQCardyOCTwo}) takes the form
\[
\sur{\sur{\sfO_1}{g_1} \sur{\sfO_2}{g_2}}{\mini g}[\zvedak C] 
\] 
In this case,  instead of~(\ref{pisu_v_Koline}), we simply take
\[
\sur{\sur{\, p'\sfo'_1\, \sfo'_2\cdots\, \sfo'_{b'}}{g_1}}{\podpera 0}[\emptyset] 
\ooo{p'}{p''}
\sur{\sur{\, q'\sfo''_1 \, \sfo''_2\cdots\, \sfo''_{b''}}{g_2}}{\podpera 0}[\emptyset] 
\ooo{q'}{q''}
\sur{\sur{\cycs{p''}}{\podpera 0} \sur{\cycs{q''}}{0}}{\podpera
  0}[\emptyset].
\]
This finishes the proof.
\end{proof}

\subsection{Kaufmann-Penner variant} 
The first result of this subsection explains how to modify
Proposition~\ref{THMModOCus} for the modular \envelope\ of
the \KaPe\ cyclic hybrid
$\onsOCKP$. Theorem~\ref{THMOCUnivPropKP} then describes $\oQnsOCKP$ as
the quotient of this modular
\envelope\ by the Cardy~condition.

\begin{proposition}
\label{leze_na_mne_chripka}
The modular \envelope\ 
$\ModHybFun(\onsOCKP)$ is the modular subhybrid of the modular \envelope\/ 
$\ModHybFun(\onsOCst)$ obtained by imposing the stability assumption~(\ref{eq:2})
on symbols 
in~(\ref{v_pondeli_do_Leicesteru}) resp.\
in~(\ref{zas_problemy_s_Jarkou}), 
and further discarding
\begin{enumerate}
\item[(i)] 
symbols 
$\sur{ \sur{\cycs{}}{0} \cdots \sur{\cycs{}}{0} }{\zvedak 0}[\zvedak \emptyset]$ with
$a\geq 3$,\footnote{Recall that $a$ is the number of nests.}
\item[(ii)] symbols \rule{0pt}{2em}
$\sur{ \sur{\cycs{}}{0} \cdots \sur{\cycs{}}{0} \ V }{\zvedak 0}[\emptyset]$,
$a\geq 2$, 
where $V\in\oQnsO$ has at least one input, and
\item[(iii)] 
\rule{0pt}{2em}
symbols $\sur{ \sur{\cycs{}}{ 0} 
\cdots \sur{\cycs{}}{0} }{\zvedak 0}[\zvedak \stt d]$, $a\geq 2$, 
where $d$ is a single closed input.
\end{enumerate}
\end{proposition}

\begin{proof}
Denote by $\oM$ the subcollection of  $\ModHybFun(\onsOCst)$
specified in the proposition. We need to prove that
$\oM \cong \ModHybFun(\onsOCKP)$. 
Our strategy will be first to show that $\oM$ is indeed a modular
subhybrid of  $\ModHybFun(\onsOCst)$,
then verify the assumptions of Proposition~\ref{THMSuffCondOnInjectivity}, 
apply Proposition~\ref{CorConseqInjOnMods} 
and finally check directly that the $\xi$-closure of $\onsOCKP$ is $\oM$.

\noindent 
{\it Verification that $\oM$ is a modular 
subhybrid of\/ $\ModHybFun(\onsOCst)$.} Let us check first that  $\oM$
is closed under the $\circ$-operations.
Assume that $x = y \ooo{p'}{p''} z$ for some $x,y,z\in\ModHybFun(\onsOCst)$.
We must show that, if $x \not \in \oM$, then either $y \not\in \oM$ 
or  $z \not\in \oM$. Denote by $a_x$, $a_y$ and $a_z$ the number of
nests in $x$, $y$ and $z$, respectively. We distinguish three
cases.

\noindent 
{\it The element $x$ is of type (i).} 
If $p',p''$ are open
inputs, then clearly the only possibility is that
\[
y=\sur{\sur{\cycs{}}{0} \cdots \sur{\cycs{}}{0}  
\sur{\cycs{p'}}{0} }{\zvedak 0}[\zvedak \emptyset] \and 
z=\sur{\sur{\cycs{}}{0} \cdots \sur{\cycs{}}{0}  
\sur{\cycs{p''}}{0} }{\zvedak 0}[\zvedak \emptyset].
\]
The numbers of nests are related by $a_x = a_y + a_z -1$, therefore $x
\not \in \oM$ if and only if 
\begin{equation}
\label{jakoby_to_koncilo}
a_y + a_z  \geq 4.
\end{equation}
On the other hand,  $y \not \in \oM$ (resp.~$z \not \in \oM$) if and
only if $a_y \geq 2$ (resp.~$a_z \geq 2$),
so~(\ref{jakoby_to_koncilo}) implies that at least one of $y$, $z$
does not belong to $\oM$.
If $p',p''$ are closed, then obviously
\[
y=\sur{\sur{\cycs{}}{0} \cdots \sur{\cycs{}}{0}  }{0}[\zvedak \stt {p'}] \and 
z=\sur{\sur{\cycs{}}{0} \cdots \sur{\cycs{}}{0}  }{0}[\zvedak \stt{p''}]. 
\]
Now  $a_x = a_y + a_z$, so $x
\not \in \oM$ if and only if 
$
a_y + a_z  \geq 3
$
and we conclude as in the open case that either $a_y \geq 2$ or $a_z
\geq 2$.

\noindent 
{\it The element $x$ is of type (ii).} 
If $p'$, $p''$ are open inputs, one has two possibilities. The first
one is that 
\[
y=\sur{\sur{\cycs{}}{0} \cdots \sur{\cycs{}}{0}  
\sur{\cycs{p'}}{0}\, V\, }{\zvedak 0}[\emptyset] \and 
z=\sur{\sur{\cycs{}}{0} \cdots \sur{\cycs{}}{0}  
\sur{\cycs{p''}}{0} }{\zvedak 0}[\emptyset]
\]
(or the r\^oles of $y$ and $z$ interchanged). Then $z \not\in \oM$ and
we are done. The second option is
\[
y=\sur{\sur{\cycs{}}{0} \cdots \sur{\cycs{}}{0}  
\, V_y\, }{\zvedak 0}[\emptyset] \and 
z= \sur{\sur{\cycs{}}{0} \cdots \sur{\cycs{}}{0}  
\, V_z\, }{\zvedak 0}[\emptyset]
\]
for some $V_y,V_z\in\oQnsO$, both having at least one input,   
such that $V_y\ooo{p'}{p''}V_z=V$. We
easily verify that $x \not \in \oM$ if and only
if~(\ref{jakoby_to_koncilo}) holds which implies, as before, that
either $y$ or $z$ does not belong to $\oM$.
In case of closed inputs, the only possibility is
\[
y=\sur{\sur{\cycs{}}{0} \cdots \sur{\cycs{}}{0}  
\, V\, }{\zvedak 0}[\zvedak \stt {p'}] \and 
z= \sur{\sur{\cycs{}}{0} \cdots \sur{\cycs{}}{0}  
}{\zvedak 0}[\zvedak \stt {p''}]
\]
(or $y$ and $z$ interchanged). We see right away that $z \not\in \oM$.

\noindent 
{\it The element $x$ is of type (iii).} 
If $p'$ and $p''$ are open inputs, then
\[
y=\sur{\sur{\cycs{}}{0} \cdots
  \sur{\cycs{}}{0}\sur{\cycs{p'}}{0}}{\zvedak 0}[\zvedak \stt d] 
\and 
z=\sur{\sur{\cycs{}}{0} \cdots
  \sur{\cycs{}}{0}\sur{\cycs{p''}}{0}}{\zvedak 0}[\zvedak \emptyset]
\]
(or vice versa). If they are closed, the only possibility is
\[
y=\sur{\sur{\cycs{}}{0} \cdots
  \sur{\cycs{}}{0}}{\zvedak 0}[\zvedak \stt {d p'}] 
\and 
z=\sur{\sur{\cycs{}}{0} \cdots
  \sur{\cycs{}}{0}}{\zvedak 0}[\zvedak \stt {p''}]
\]
(or vice versa). In both cases $z \not\in \oM$.

It remains to verify that $\oM$  is closed under contractions.
Let $x=\oxi{p'p''}y$ for some elements 
$x,y\in \ModHybFun(\onsOCst)$. We must show that $x
\not\in \oM$ implies  $y \not\in \oM$. If $x$ is of type
(i), there are only three thinkable  candidates for $y$, namely
\[
\sur{\sur{\cycs{}}{0} 
\cdots
\sur{\cycs{}}{0}\sur{\cycs{p'}}{0}\sur{\cycs{p''}}{0}}{\zvedak 0}[\zvedak\emptyset], \
\sur{\sur{\cycs{}}{0} \cdots
  \sur{\cycs{}}{0}\sur{\cycs{p'}\cycs{p''}}{0}}{\zvedak
  0}[\zvedak\emptyset] 
\ \hbox { or } \
\sur{\sur{\cycs{}}{0} \cdots
  \sur{\cycs{}}{0}}{\zvedak 0}[\zvedak\stt {p',p''}].
\]
The respective values of the contraction $\oxi{p'p''}y$  are
\[
\sur{\sur{\cycs{}}{0} \cdots
  \sur{\cycs{}}{0}}{\zvedak 1}[\zvedak\emptyset], \
\sur{\sur{\cycs{}}{0} \cdots
  \sur{\cycs{}}{1}}{\zvedak 0}[\zvedak\emptyset] \and
\sur{\sur{\cycs{}}{0} \cdots
  \sur{\cycs{}}{0}}{\zvedak 1}[\zvedak\emptyset]
\]
which excludes this possibility. The situation when $x$ is of
type~(iii) is similar. 

Assume finally that $x$ is of type (ii). Besides
the candidates for $y$ similar to the ones above, there are also
\[
\sur{\sur{\cycs{}}{0} \cdots
  \sur{\cycs{}}{0}\, V_1\, V_2}{\zvedak0}[\zvedak\emptyset]
\hbox { with } V_1\ooo{p'}{p''}V_2=V, \and
\sur{\sur{\cycs{}}{0} \cdots
  \sur{\cycs{}}{0}\, W}{\zvedak 0}[\zvedak \emptyset] 
\hbox { with } \oxi{p'p''}W=V.
\] 
For the first candidate
\[
\oxi{p'p''}y=\sur{\sur{\cycs{}}{0} \cdots
  \sur{\cycs{}}{0}\, V}{\zvedak 1}[\zvedak \emptyset]
\]
while the second candidate does not belong to $\oM$. This finishes the
verification that $\oM$ is a modular 
subhybrid of\/ $\ModHybFun(\onsOCst)$.

\noindent 
{\it Verifying assumptions of Proposition~\ref{THMSuffCondOnInjectivity}.}
Let $w',w''\in \onsOCst$ be such that $w'\ooo{q'}{q''}w''\in\onsOCKP$.
If both $w',w''\in\onsOCKP$, there is nothing to verify.
If both  $w',w''\not \in \onsOCKP$, 
then it is easy to check  
that also $w'\ooo{q'}{q''}w''\not\in  \onsOCKP$, 
so the only interesting case is when precisely one of $w'$ and $w''$ does
not belong to  $\onsOCKP$.

Assume therefore that $w'\in\onsOCKP$ but $w''\not\in\onsOCKP$.
Since $w''$ has to have at least two
inputs $p''$ and $q''$, it
must be of type $(ii)$ in the classification of
Example~\ref{KP}.  This leaves us with two possibilities.

\noindent 
{\it Case 1:}
$w'=\sur{\sfo_1 \sfo_2 \sfo_3 \,\cdots\, \sfo_b}{0}[\zvedak C]$, 
$w''=\sur{\sfo\,  \cycs {}\, \cdots\, \cycs{}}{0}[\zvedak \emptyset]$, 
$p'\in \sfo_1$, $q'\in \sfo_2$ and  $p'',q''\in \sfo$. 
If it is so, then
\[
w' \ooo {q'}{q''} w'' = \sur{\, \sfo_1 (\sfo_2\ooo {q'}{q''} \sfo )\,
 \sfo_3 \,\cdots\, \sfo_b\, \cycs {}\, \cdots\,
  \cycs{}}{0}[\zvedak C]  \and
w' \ooo {p'}{p''} w'' = \sur{(\sfo_1\ooo {p'}{p''} \sfo )\, \sfo_2
 \sfo_3 \,\cdots\, \sfo_b\, \cycs {} \,\cdots\,
  \cycs{}}{0}[\zvedak C] .
\]
If $|\sfo|\geq 3$, then the 
assumption~(ii) of  Proposition~\ref{THMSuffCondOnInjectivity} is
satisfied with
\[
w_1'=\sur{\,\sfo_1 \sfo_2\, \sfo_3 \cdots \sfo_b \, \cycs {} \cdots
  \cycs{}}{0}[C]
\ \hbox { and } \
w_1''=\sur{\, \sfo {} \, }{0}[\emptyset]
\]
where $w_1'$ absorbed all empty cycles of $w''$. If $\sfo =
\cycs{p''q''}$, such $w''_1$ is not stable. We however have
\[
w' \ooo {q'}{q''} w'' = \sur{\, \sfo_1 (\sfo_2\ooo {q'}{q''} \cycs{p''q''} )\,
 \sfo_3\, \cdots\, \sfo_b\, \cycs {} \,\cdots\,
  \cycs{}}{0}[\zvedak C], \
w' \ooo {p'}{p''} w'' = \sur{(\sfo_1\ooo {p'}{p''} \cycs{p''q''} )\, \sfo_2
 \sfo_3\, \cdots\, \sfo_b\, \cycs {} \,\cdots\,
  \cycs{}}{0}[\zvedak C] 
\]
so we notice, 
as in the proof of Proposition~\ref{THMModOCStable},
that $w'\ooo{q'}{q''}w''$ replaces $q'$ by $p''$
and  $w'\ooo{p'}{p''}w''$ replaces $p'$ by $q''$ in $w'$.  Therefore
$w'\ooo{q'}{q''}w'=\rho(w'\ooo{p'}{p''}w'')$ 
for a bijection $\rho$ mapping $\stt{p',p''}$ to $\stt{q',q''}$ 
and restricting to the identity elsewhere. 
Assumption~(i) of  Proposition~\ref{THMSuffCondOnInjectivity} is thus
satisfied. 

\noindent 
{\it Case 2:}
$w'=\sur{\sfo_1 \sfo_2 \,\cdots \,\sfo_b}{0}[\zvedak C]$, 
$w''=\sur{\sfo\,  \cycs {}\, \cdots\, \cycs{}}{0}[\zvedak \emptyset]$, 
$p',q'\in \sfo_1$, and  $p'',q''\in \sfo$. 
Then we calculate
\[
w' \ooo {q'}{q''} w'' = \sur{\,  (\sfo_1\ooo {q'}{q''} \sfo )\,
 \sfo_2\, \cdots\, \sfo_b\, \cycs {}\, \cdots\,
  \cycs{}}{0}[\zvedak C]  \and
w' \ooo {p'}{p''} w'' = \sur{(\sfo_1\ooo {p'}{p''} \sfo )\, 
 \sfo_2 \,\cdots\, \sfo_b\, \cycs {} \,\cdots\,
  \cycs{}}{0}[\zvedak C] .
\]
If $\sfo \not = \cycs{p''q''}$, the assumption~(ii) of 
Proposition~\ref{THMSuffCondOnInjectivity} is satisfied with
\[
w_1'=\sur{\,\sfo_1 \, \sfo_2 \cdots \sfo_b \, \cycs {} \cdots
  \cycs{}}{0}[\zvedak C]
\ \hbox { and } \
w_1''=\sur{\, \sfo \, }{0}[\zvedak \emptyset].
\]
If $\sfo = \cycs{p''q''}$, we argue precisely as in the first case.

This finishes the verification of assumptions of
Proposition~\ref{THMSuffCondOnInjectivity}.
Proposition~\ref{CorConseqInjOnMods} now
implies
\[
\ModHybFun(\onsOCKP) = 
\xi_{\ModHybFun(\onsOCst)}(\onsOCKP).
\]
Since we already know that $\oM$ is a modular subhybrid of 
$\ModHybFun(\onsOCst)$, the minimality of the $\xi$-closure
implies the inclusions
\[
\ModHybFun(\onsOCKP) = 
\xi_{\ModHybFun(\onsOCst)}(\onsOCKP)  \subset 
 \oM \subset \ModHybFun(\onsOCst).
\]
It therefore remains to:

\noindent 
{\it  Verify that 
$\oM \subset \xi_{\ModHybFun(\onsOCst)}(\onsOCKP)$.}
We know by~\cite{nsmod} that $\oQnsO \cong \nsoMod(\nsoAss)$.
For any $V \in   \oQnsO$ therefore exists a (non-unique) $\sfo_V \in
\nsoAss$ such that 
\begin{equation}
\label{Jarka_nachlazena}
V = \oxi {q_1'q_1''} \cdots \oxi {q_t'q_t''}(\sfo_V)
\end{equation} 
for some $q_1',q_1'',  \ldots, q_t',q_t'' \in \sfo_V$. To save
space, we will denote the iterated contraction
in~(\ref{Jarka_nachlazena}) by $\xi_V$;~(\ref{Jarka_nachlazena}) will
then read $V =\xi_V (\sfo_V)$.
With this notation, we have in $\ModHybFun(\onsOC)$ the equality
\begin{gather} 
\label{EQDecompositionVirtual}
\sur{V_1 \cdots V_a}{g}[\zvedak C] =
\oxi{p'_1p''_1}\cdots\oxi{p'_gp''_g} \, \xi_{V_1} \cdots \xi_{V_a} 
\sur{ \sur{\, \sfo_{V_1}}{0}\ \cdots  \ \sur{\,
    \sfo_{V_a}}{0}}{0}[\zvedak C \sqcup \stt{p'_1,p''_1,\ldots,p'_g,p''_g}]
\end{gather}
along with the identification
\begin{equation}
\label{nejde_internet}
\ModHybFun(\onsOC) \ni
\sur{ \sur{\, \sfo_{V_1}}{0}\ \cdots  \ \sur{\,
    \sfo_{V_a}}{0}}{0}[\zvedak C\, \sqcup\,
\stt{p'_1,p''_1,\ldots,p'_g,p''_g}] = 
\sur{ \sfo_{V_1}\ \cdots  \ 
    \sfo_{V_a}}{0}[\zvedak C\, \sqcup \,
  \stt{p'_1,p''_1,\ldots,p'_g,p''_g}]
\in \onsOC
\end{equation}
provided by the unit~(\ref{v_Koline_behem_soustredeni}).

Denote the left hand side of~(\ref{EQDecompositionVirtual}) and the
element in~(\ref{nejde_internet}) by $y$.
If $x$ is stable, then so is $y$ by
Lemma~\ref{Prijedu_zase_pristi_rok?},
so~(\ref{EQDecompositionVirtual}) in fact holds in $\ModHybFun(\onsOCst)$.
We need to show that if $x\in \oM$, then $y\in\onsOCKP$.

If $g\geq 1$, then $y$ has at least two closed inputs,  thus $y\in\onsOCKP$.
Suppose that $g=0$ and $|C|\geq 1$. If $y\not\in\onsOCKP$, 
then $|C|=1$ and $\sfo_1 = \cdots = \sfo_a =\sur{\cycs{}}{0}$, so 
$V_1=\cdots= V_a =\sur{\cycs{}}{0}$, 
hence $x\not\in\oOC_\KP$, which contradicts the assumption. 
The last case to be analyzed is  
$g=0$ and $|C|=0$. Then at least two $V_i$'s, say $V_1$ and $V_2$,
have at least one input, otherwise $x \not\in \oM$.
So the same is true for $\sfo_{V_{1}}$ and $\sfo_{V_{1}}$, thus $y\in\oOC_\KP$.
\end{proof}

To sum up, at this moment we know that the sequence of
inclusions~(\ref{vcera_jsem_jel_na_brsulickach}) induce inclusions
\[
\ModHybFun(\onsOC_\KP) 
\hookrightarrow \ModHybFun(\onsOCst) \hookrightarrow \ModHybFun(\onsOC),
\]
We also know that
\[
\xi_{\ModHybFun(\onsOCst)}(\onsOC_\KP) \cong \ModHybFun(\onsOC_\KP)
\and
\xi_{\ModHybFun(\onsOC)}(\onsOC_\KP) \cong  \ModHybFun(\onsOC)
\]
while the isomorphism $\xi_{\ModHybFun(\onsOCst)}(\onsOC_\KP) \cong 
\xi_{\ModHybFun(\onsOC)}(\onsOC_\KP)$ is immediate.

\begin{definition}
\label{za_chvili_s_Nunykem_na_veceri}
The {\em Kaufmann-Penner modular hybrid\/} $\oQnsOCKP$ is the modular
subhybrid of $\oQnsOCst$ generated by $\onsOCKP$, i.e.\ $\oQnsOCKP :=
\xi_{\oQnsOCst}(\onsOCKP)$.
\end{definition}

A more intelligent description of  $\oQnsOCKP$ will be given
in Theorem~\ref{THMPresentationOfQOCKP} below.
The linearization of the Kaufmann-Penner hybrid 
$\oQnsOCKP$ is in fact isomorphic to the
homology of the arc operad $\widetilde{{\it Arc}}$ of
\cite[page~346]{kaufmann-penner:NP06}, whence its name.  
Notice that  $\oQnsOCKP$ 
contains the \gg\ stable cyclic operad $\oComst$ from
Example~\ref{ze_ja_vul_jsem_zacal_psate_ten_grant_stable} 
as a suboperad of elements with $g = b =0$, 
i.e.\ elements of the form
\begin{equation}
\label{eq:6}
\sur{\varnothing}{0}[\korekce C] , \ C \in \Fin.
\end{equation}   
Let us prove a variant of
Theorems~\ref{PROCUnivProp} and~\ref{PROOCUnivPropST} for it.

\begin{theorem} 
\label{THMOCUnivPropKP}
Let\/ $\EuScript{I}$ denote the ideal in the modular hybrid 
$\ModHybFun(\onsOCKP)$ generated by the relation
\begin{equation}
\label{jsem_zmitan}
\sur{ \sur{\cycs{q}}{0} \sur{\cycs{r}}{0} }{\zvedak 0}[\zvedak \emptyset] = \sur{
  \sur{\cycs{q}\cycs{r}}{0} }{\zvedak 0}[\zvedak \emptyset].
\end{equation}
Then
\[
\oQnsOCKP \ \cong\ \ModHybFun(\onsOCKP) /\, \EuScript{I}.
\]
Therefore, 
for any modular hybrid $\oH$ and a morphism $F:\onsOCKP\to\oH$ such
that 
\[
\oxi{uv}F\sur{\cycs{uqvr}}{0}[\emptyset] = F\sur{\cycs{q}\cycs{r}}{0}[\emptyset],
\]
there is a unique morphism $\hat{F}:\oQnsOCKP\to\oH$ of modular hybrids
making the diagram 
\[
\xymatrix@C4em{
\onsOCKP\ \ar@{^{(}->}[r] \ar_{F}[dr]  & 
\oQnsOCKP \ar@{-->}^(.5){\hat{F}}[d]
\\ &
\oH 
}
\]
commutative.
\end{theorem}

An immediate consequence of this theorem combined with
Proposition~\ref{leze_na_mne_chripka} is that the
symbol~(\ref{dnes_s_Jarcou_na_vyhlidku}) with $g\geq 1$ belongs to
$\oQnsOCKP$ if and only if it is stable, i.e.\ if
either of $b$, $|O|$ or $|C|$ is nonzero. 

\begin{proof}[Proof of Theorem~\ref{THMOCUnivPropKP}]
It is clear that $\oQnsOCKP$ is  isomorphic to the $\xi$-closure of
$\onsOCKP$ in  $\oQnsOC$, therefore the 
isomorphism~(\ref{Kdy se s tim prestanu trapit?}) 
identifies $\ModHybFun(\onsOC_\KP)$ with
$\oQnsOCKP$. 
The proof therefore goes along the similar lines as the proof of 
Theorem~\ref{PROCUnivProp}, so we only highlight the differences.

We must again be aware that some terms in~(\ref{pisu_v_Koline}) which we
used to prove~(\ref{EQCardyOCTwo}) may not belong to
$\ModHybFun(\onsOCKP)$. In the proof of Theorem~\ref{PROOCUnivPropST}
we explained how to avoid appearances of unstable terms. The remaining terms
outside $\ModHybFun(\onsOCKP)$ will be eliminated by absorbing 
trivial nests  $\sur{\cycs{}}{0}$.

By this we mean that, for arbitrary nontrivial nests $V_1,\ldots,V_a
\neq\sur{\cycs{}}{0}$, we prove the following equality modulo~$\EuScript{I}$
\begin{gather} 
\label{EQAbsorbEmptyNest}
\sur{\sur{\cycs{}}{0}\cdots\sur{\cycs{}}{0} V_1\cdots V_a}{\zvedak g}[\zvedak C] 
= \sur{\widetilde {V}_1 V_2\cdots V_a}{\zvedak g}[\zvedak C],
\end{gather}
where 
\[
\widetilde{\zvedak V}_1:=\sur{\cycs{}\cdots\cycs{}\,
  \sfo_1\, \cdots \,\sfo_b}{\zvedak g_1}
\] 
if $V_1=\sur{\sfo_1\, \cdots \,\sfo_b}{g_1}$.
Assuming this, it suffices to prove \eqref{EQCardyOCTwo} for elements 
not containing a trivial 
nest~$\sur{\cycs{}}{0}$, which proceeds as in the proof of
Theorem~\ref{PROOCUnivPropST}.
To verify \eqref{EQAbsorbEmptyNest}, it suffices to prove~that
\begin{equation}
\label{Je_to_s_nama_jako_v_tom_filmu.}
\sur{\sur{\cycs{}}{0} V_1\cdots V_a}{\zvedak g}[\zvedak C] = \sur{\widetilde {V}_1
  V_2\cdots V_a}{\zvedak g}[\zvedak C]    \ \hbox{ modulo } \ \EuScript{I}
\end{equation}
for arbitrary $V_1, \ldots , V_a$ such that the 
left hand side of~(\ref{Je_to_s_nama_jako_v_tom_filmu.}), which
we denote by $x$, belongs to $\ModHybFun(\onsOCKP)$.
In the following calculation we denote, for $V = \sur{\sfO}{g} \in
\oQnsO$ and an independent symbol $p$, by $pV$ a nest of the form $V =
\sur{p\,\sfO}{g}$, where $p\,\sfO$ is an extension of the multicycle
$\sfO$ introduced in the proof of Proposition~\ref{THMModOCus}. 
We distinguish four cases.

\noindent 
{\it Case 1: $g\geq 1$.}
If  $x$ has exactly two nests, we use the decomposition
\begin{equation}
\label{eq:3}
\sur{\sur{\cycs{}}{0}\, V\, }{\zvedak g}[\zvedak C] = \oxi{r'r''} \left(
  \sur{\sur{\cycs{r'}}{0} \sur{\cycs{p'}}{0}}{\zvedak 0}[\zvedak \emptyset] \ooo{p'}{p''}
  \sur{\sur{\cycs{p''}}{0}\sur{\cycs{q'}}{0}}{\zvedak 0}[\zvedak \emptyset] \ooo{q'}{q''}
  \sur{q'r''V}{\zvedak g-1}[\zvedak C] \right).
\end{equation}
Applying relation~(\ref{jsem_zmitan}) 
to the middle term of the right hand side, we get
\begin{align*}
\sur{\sur{\cycs{}}{0}\,V\,}{\zvedak g}[\zvedak C] &= \oxi{r'r''} \left(
  \sur{\sur{\cycs{r'}}{0} \sur{\cycs{p'}}{0}}{\zvedak 0}[\zvedak\emptyset] 
\! \ooo{p'}{p''} \!
  \sur{
\sur{\cycs{p''} \cycs{q'}}{0} 
}{\zvedak 0}[\zvedak \emptyset]  \!\ooo{q'}{q''} \!
  \sur{q'r''V}{\zvedak g-1}[\zvedak C] \right)
=  \oxi{r'r''} \sur{\sur{\cycs{r'}\cycs{}}{0} \sur{r''V}{0}}{\zvedak
  g-1}[\zvedak  C] 
=  \sur{\widetilde {V}}{\zvedak g}[\zvedak C]. 
\end{align*}
The only term in the right hand side of~(\ref{eq:3}) 
that might not belong to $\ModHybFun(\onsOCKP)$ is the rightmost one.
This happens if and only if  $g=1$, $C=\emptyset$, $V=\sur{\cycs{}}{0}$, 
in which case we verify directly~that
\begin{align*}
\sur{\sur{\cycs{}}{0}\sur{\cycs{}}{0}}{\zvedak 1}[\zvedak \emptyset] &= \oxi{q'q''}
\left( \sur{\sur{\cycs{p'}}{0}\sur{\cycs{q'}}{0}}{\zvedak 0}[\zvedak \emptyset]
  \ooo{p'}{p''} \sur{\sur{\cycs{p''}}{0}\sur{\cycs{q''}}{0}}{\zvedak
    0}[\zvedak  \emptyset]
\right)
\\
& =\oxi{q'q''}
\left( \sur{\sur{\cycs{p'}\cycs{q'}}{0}}{\zvedak 0}[\zvedak \emptyset]
  \ooo{p'}{p''}
 \sur{\sur{\cycs{p''}\cycs{q''}}{0}}{\zvedak 0}[\zvedak \emptyset]
\right)
= 
\oxi{q'q''}\sur{\sur{\cycs{}\cycs{q'}\cycs{q''}}{0}}{\zvedak
  0}[\zvedak 
\emptyset] =
\sur{\sur{\cycs{}\cycs{}}{0}}{\zvedak 1}[\zvedak \emptyset]. 
\end{align*}
If $x$ has at least three nests, we use the decomposition
\[
\sur{\sur{\cycs{}}{0}V_1 \cdots V_a}{\zvedak g}[\zvedak C] = \oxi{r'r''} \left(
  \sur{\sur{\cycs{r'}}{0}\sur{\cycs{p'}}{0}}{\zvedak 0}[\zvedak \emptyset] \ooo{p'}{p''}
  \sur{\sur{\cycs{p''}}{0}\sur{\cycs{q'}}{0}}{\zvedak 0}[\zvedak \emptyset] \ooo{q'}{q''}
  \sur{q''V_1 \, r''V_2\, V_3\cdots V_a}{\zvedak g-1}[\zvedak C] \right),
\]
apply~(\ref{jsem_zmitan}) to the middle term in the right hand
side and proceed as before. In this case 
all terms clearly belong to $\ModHybFun(\onsOCKP)$.

\noindent 
{\it Case 2:
$g=0$ and $|C|\geq 2$.} Let  $d\in C$ and  $C':= C \setminus \stt d$.
Then we use the decomposition
\[
\sur{\sur{\cycs{}}{0}V_1\cdots V_a}{\zvedak 0}[\zvedak C] 
= \sur{\sur{\cycs{p'}}{0}}{\zvedak 0}[\zvedak \stt d] \ooo{p'}{p''}
\sur{\sur{\cycs{p''}}{0}\sur{\cycs{q'}}{0}}{\zvedak 0}[\zvedak 
\emptyset] \ooo{q'}{q''}
\sur{q'V_1 \, V_2    \cdots\, V_a}{\zvedak 0}[\zvedak C'].
\]
All terms in the right hand side obviously belong to $\ModHybFun(\onsOCKP)$.

\noindent 
{\it Case 3:
$g=0$ and $C = \stt {d}$.} We want to decompose
\begin{equation}
\label{eq:5}
\sur{\sur{\cycs{}}{0}V_1\cdots V_a}{\zvedak 0}[\zvedak \stt d] 
= \sur{\sur{\cycs{p'}}{0}}{\zvedak0}[\zvedak\stt d] \ooo{p'}{p''}
\sur{\sur{\cycs{p''}}{0}\sur{\cycs{q'}}{0}}{\zvedak 0}[\zvedak 
\emptyset] \ooo{q'}{q''}
\sur{q'V_1 \, V_2    \cdots\, V_a}{\zvedak 0}[\zvedak\emptyset].
\end{equation}
While the first two terms in the right hand side always belong to
$\ModHybFun(\onsOCKP)$, the last one may be problematic.
Let us discuss the case when $a \geq 2$ first. Since  
$x \in\ModHybFun(\onsOCKP)$, at least one of its nests must differ
from $\sur{\cycs{}}{0}$; we may assume without loss of generality it
is~$V_2$. Then the rightmost term in~(\ref{eq:5})  belongs to
$\ModHybFun(\onsOCKP)$. 

The same is true if $a=1$ and if $q'V_1$ is stable. If it is not
stable, then $V_1$ must be of the form $\sur{\cycs{r}}{0}$ and we
verify directly that
\[
\sur{\sur{\cycs{}}{0} \sur{\cycs{r}}{0}}{\zvedak 0}[\zvedak \stt d] 
=
\sur{\sur{\cycs{p'}}{0} \sur{\cycs{r}}{0}}{\zvedak 0}[\zvedak\emptyset]
\ooo{p'}{p''}\sur{\sur{\cycs{p''}}{0}}{\zvedak0}[\zvedak\stt d] 
=
\sur{\sur{\cycs{p'}\cycs{r}}{0}}{\zvedak0}[\zvedak\emptyset]
\ooo{p'}{p''}\sur{\sur{\cycs{p''}}{0}}{\zvedak0}[\zvedak\stt d] =
\sur{\sur{\cycs{}\,\cycs{r}}{0}}{\zvedak0}[\zvedak\emptyset].
\]

\noindent 
{\it Case 4: $g=0$ and $C= \emptyset $.}
Since $x\in \ModHybFun(\onsOCKP)$, at least two of its nests are
nontrivial, so 
we may assume that
\[
x=\sur{\sur{\cycs{}}{0}V_1V_2\cdots V_a}{\zvedak 0}[\zvedak \emptyset],
\] 
where $V_1,V_2\neq\sur{\cycs{}}{0}$.
If $pV_1$ is stable, we decompose $x$ as
\[
\sur{\sur{\cycs{}}{0}V_1V_2\cdots V_a}{\zvedak 0}[\zvedak \emptyset] 
= \sur{p'V_1}{\zvedak 0}[\zvedak \emptyset] \ooo{p'}{p''}
\sur{\sur{\cycs{p''}}{0}\sur{\cycs{q'}}{0}}{\zvedak 0}[\zvedak \emptyset] \ooo{q'}{q''}
\sur{\sur{\cycs{q''}}{0} V_2 \cdots V_a}{\zvedak 0}[\zvedak \emptyset]
\]
and apply~(\ref{jsem_zmitan}) to the middle term in the right hand
side as before.
If $pV_1$ is not stable, 
then $V_1$ has to be of the form $\sur{\cycs{r}}{0}$ 
and we use instead the decomposition
\[
\sur{\sur{\cycs{}}{0}\sur{\cycs{r}}{0}V_2\cdots V_a}{\zvedak 0}[\zvedak \emptyset] =
\sur{\sur{\cycs{r}}{0}\sur{\cycs{q'}}{0}}{\zvedak 0}[\zvedak \emptyset] \ooo{q'}{q''}
\sur{\sur{\cycs{q''}}{0} V_2 \cdots V_a}{\zvedak 0}[\zvedak \emptyset].
\]
This finishes our verification of~(\ref{EQAbsorbEmptyNest}).
\end{proof}

\section{Finitary presentations}
\label{dnes_prodlouzit_prukaz na_LAA}

The aim of this section is to give an explicit finitary presentation
of the Kaufmann-Penner modular hybrid $\oQnsOCKP$ and derive from it a
description of its algebras.
As the first step we express $\onsOCKP$ in terms of
generators and relations. Recall that the components of a cyclic
hybrid $\oH$ are indexed by couples consisting of a multicycle $\sfO$ and a
finite set $C$. We will call the symbol 
$\sfO \choose C$ the {\em  biarity\/}
of elements in $\oH(\sfO,C)$.

\begin{theorem} 
\label{THMPresentaionOfOC}
The cyclic hybrid $\onsOCKP$ has the following presentation.
The generators are:
\begin{enumerate}
\item[(g1)] 
an `open pair of 
pants' $\mu = \mu^{\cycs{pqr}}_\emptyset$ of biarity $\cyc{pqr}\choose
\emptyset$ with $G=0$, with the trivial action of cyclic
order-preserving automorphisms of $\cycs{pqr}$,
\item[(g2)] 
a `closed pair of pants' $\omega=\omega^\varnothing_{\stt {def}}$ of
biarity $\varnothing \choose {\stt {def}}$
with $G=\frac12$, and the trivial action of the 
group of automorphisms of $\stt {d,e,f}$, and
\item[(g3)] 
a `morphism' $\phi =\phi^{\cycs{p}}_{\stt d}$ with $G= \frac 12$ of
biarity $\cycs{p} \choose {\stt d}$,
\end{enumerate}
subject to the axioms:
\begin{enumerate}
\item[(a1)] 
associativity in open inputs:
	\begin{gather*}
	\mu^{\cycs{pqr}}_\emptyset \ooo{r}{s}
        \mu^{\cycs{stu}}_\emptyset 
= \mu^{\cycs{pru}}_\emptyset \ooo{r}{s} \mu^{\cycs{qts}}_\emptyset,
\end{gather*}
\item[(a2)] 
associativity in closed inputs:
	\begin{gather*}
	\omega^{\varnothing}_{\stt {def}} \ooo{f}{g} 
\omega^{\varnothing}_{\stt {ghi}} =
        \omega^{\varnothing}_{\stt {dfi}} \ooo{f}{g} 
\omega^{\varnothing}_{\stt {ehg}},
\end{gather*}
\item[(a3)] 
morphism property:
	\begin{gather*}
	\phi^{\cycs{p}}_{\stt g} \ooo{g}{f} 
\omega^{\varnothing}_{\stt {def}} =
        \left( \mu^{\cycs{pqr}}_{\emptyset} \ooo{q}{s}
          \phi^{\cycs{s}}_{\stt d} \right) \ooo{r}{t}
        \phi^{\cycs{t}}_{\stt e},\  and
      \end{gather*}
\item[(a4)] 
centrality:
	\begin{gather*}
	\mu^{\cyc{pqr}}_\emptyset \ooo{q}{s} \phi^{\cycs{s}}_{\stt d}
        = \mu^{\cyc{prq}}_\emptyset \ooo{q}{s}
        \phi^{\cycs{s}}_{\stt d}.
      \end{gather*}
    \end{enumerate}
In other words, $\onsOCKP$ is the quotient 
\begin{equation}
\label{pristi_tyden_vyjezdni_zasedani_Brno}
\onsOCKP \ \cong\ \Fr[\underline{c}yc]{E} /\, \EuScript{J}
\end{equation}
of the free cyclic hybrid generated by the collection
$E$ consisting of $\mu$, $\omega$ and $\phi$ as above,
modulo the ideal  $\EuScript{J}$ generated by the relations (ai)-(aiv).    
\end{theorem}

Generators (g1)--(g3) will be depicted as
\begin{equation}
\raisebox{-3.2em}{\rule{0pt}{1pt}}
\label{eq:4}
\psscalebox{1 1}
{
\begin{pspicture}(0,0)(9.465,1.3425)
\psline[linecolor=black, linewidth=0.04](1.845,0.7825)(1.845,-0.0175)(1.045,-0.8175)(1.045,-0.8175)
\psline[linecolor=black, linewidth=0.04](1.845,-0.0175)(2.645,-0.8175)(2.645,-0.8175)
\psline[linecolor=black, linewidth=0.04, linestyle=dashed, dash=0.17638889cm 0.10583334cm](5.445,0.7825)(5.445,-0.0175)(4.645,-0.8175)(4.645,-0.8175)
\psline[linecolor=black, linewidth=0.04, linestyle=dashed, dash=0.17638889cm 0.10583334cm](5.445,-0.0175)(6.245,-0.8175)(6.245,-0.8175)
\psline[arrows=>, arrowscale=2 2,arrowinset=0.0, linecolor=black, linewidth=0.04, linestyle=dashed, dash=0.17638889cm 0.10583334cm](8.645,.2)(8.645,-0.8175)(8.645,-0.8175)
\psline[linecolor=black, linewidth=0.04](8.645,0.7825)(8.645,-0.0175)(8.645,-0.0175)
\psdots[linecolor=black, dotsize=0.2](1.845,-0.0175)
\psdots[linecolor=black, dotsize=0.2](5.445,-0.0175)
\rput(1.845,1.1825){$r$}
\rput(1.045,-1.2175){$p$}
\rput(5.445,1.1825){$f$}
\rput(4.645,-1.2175){$d$}
\rput(6.245,-1.2175){$e$}
\rput(8.645,1.1825){$p$}
\rput(8.645,-1.2175){$d$}
\rput(2.645,-1.2175){$q$}
\rput(0.245,-0.0175){$\mu$\ :}
\rput(3.845,-0.0175){$\omega$\ :}
\rput(7.845,-0.0175){$\phi$\ :}
\rput(2.645,-0.0175){,}
\rput(6.245,-0.0175){,}
\rput(9.445,-0.0175){.}
\end{pspicture}
}
\end{equation}
The pictorial forms of the associativities (a1) and (a2) are the
`fusion rules'
\begin{center}
\psscalebox{1.0 1.0} 
{
\begin{pspicture}(0,-1.325)(11.78,1.325)
\psline[linecolor=black, linewidth=0.04](0.09,0.8)(0.49,0.0)(0.09,-0.8)(0.09,-0.8)
\psline[linecolor=black, linewidth=0.04](0.49,0.0)(1.29,0.0)(1.69,0.8)(1.69,0.8)
\psline[linecolor=black, linewidth=0.04](1.29,0.0)(1.69,-0.8)(1.69,-0.8)
\psline[linecolor=black, linewidth=0.04](3.29,0.8)(4.09,0.4)(4.89,0.8)(4.89,0.8)
\psline[linecolor=black, linewidth=0.04](4.09,0.4)(4.09,-0.4)(3.29,-0.8)(3.29,-0.8)
\psline[linecolor=black, linewidth=0.04](4.09,-0.4)(4.89,-0.8)(4.89,-0.8)
\psline[linecolor=black, linewidth=0.04, linestyle=dashed, dash=0.17638889cm 0.10583334cm](6.89,0.8)(7.29,0.0)(8.09,0.0)(8.09,0.0)
\psline[linecolor=black, linewidth=0.04, linestyle=dashed, dash=0.17638889cm 0.10583334cm](8.49,0.8)(8.09,0.0)(8.09,0.0)
\psline[linecolor=black, linewidth=0.04, linestyle=dashed, dash=0.17638889cm 0.10583334cm](6.89,-0.8)(7.29,0.0)(7.29,0.0)
\psline[linecolor=black, linewidth=0.04, linestyle=dashed, dash=0.17638889cm 0.10583334cm](8.49,-0.8)(8.09,0.0)(8.09,0.0)
\psline[linecolor=black, linewidth=0.04, linestyle=dashed, dash=0.17638889cm 0.10583334cm](10.09,0.8)(10.89,0.4)(10.89,-0.4)(10.89,-0.4)
\psline[linecolor=black, linewidth=0.04, linestyle=dashed, dash=0.17638889cm 0.10583334cm](10.09,-0.8)(10.89,-0.4)(10.89,-0.4)
\psline[linecolor=black, linewidth=0.04, linestyle=dashed, dash=0.17638889cm 0.10583334cm](11.69,0.8)(10.89,0.4)(10.89,0.4)
\psline[linecolor=black, linewidth=0.04, linestyle=dashed, dash=0.17638889cm 0.10583334cm](11.69,-0.8)(10.89,-0.4)(10.89,-0.4)
\psdots[linecolor=black, dotsize=0.2](0.49,0.0)
\psdots[linecolor=black, dotsize=0.2](1.29,0.0)
\psdots[linecolor=black, dotsize=0.2](4.09,0.4)
\psdots[linecolor=black, dotsize=0.2](4.09,-0.4)
\psdots[linecolor=black, dotsize=0.2](7.29,0.0)
\psdots[linecolor=black, dotsize=0.2](7.29,0.0)
\psdots[linecolor=black, dotsize=0.2](8.09,0.0)
\psdots[linecolor=black, dotsize=0.2](10.89,0.4)
\psdots[linecolor=black, dotsize=0.2](10.89,-0.4)
\rput(0.09,1.2){\relax $p$}
\rput(1.69,1.2){\relax $u$}
\rput(0.09,-1.2){\relax $q$}
\rput(1.69,-1.2){\relax $t$}
\rput(3.29,1.2){\relax $p$}
\rput(4.89,1.2){\relax $u$}
\rput(3.29,-1.2){\relax $q$}
\rput(4.89,-1.2){\relax $t$}
\rput(2.49,0.0){$=$}
\rput(5.69,0.0){and}
\rput(6.89,1.2){\relax $d$}
\rput(8.49,1.2){\relax $i$}
\rput(6.89,-1.2){\relax $e$}
\rput(8.49,-1.2){\relax $h$}
\rput(10.09,1.2){\relax $d$}
\rput(11.69,1.2){\relax $i$}
\rput(10.09,-1.2){\relax $e$}
\rput(11.69,-1.2){\relax $d$}
\rput(9.29,0.0){$=$}
\end{pspicture}
}
\end{center}
while the morphism property (a3) and the centrality (a4) are depicted as
\begin{center}
\psscalebox{1.0 1.0} 
{
\begin{pspicture}(0,-1.7725)(12.91,1.7725)
\psline[linecolor=black, linewidth=0.04, linestyle=dashed, dash=0.17638889cm 0.10583334cm](0.09,-1.2475)(0.89,-0.4475)(0.89,-0.4475)
\psline[linecolor=black, linewidth=0.04, linestyle=dashed, dash=0.17638889cm 0.10583334cm](1.69,-1.2475)(0.89,-0.4475)(0.89,-0.4475)
\psline[linecolor=black, linewidth=0.04](0.89,1.1525)(0.89,0.3525)(0.89,0.3525)
\psline[linecolor=black, linewidth=0.04](4.09,0.3525)(3.29,-0.4475)(3.29,-0.4475)
\psline[linecolor=black, linewidth=0.04](4.89,-0.4475)(4.09,0.3525)(4.09,1.1525)(4.09,1.1525)
\psline[linecolor=black, linewidth=0.04, linestyle=dashed, dash=0.17638889cm 0.10583334cm](8.09,0.3525)(8.09,0.3525)
\psline[linecolor=black, linewidth=0.04, linestyle=dashed, dash=0.17638889cm 0.10583334cm](8.89,-0.4475)(8.09,0.3525)(8.09,0.3525)
\psline[linecolor=black, linewidth=0.04, linestyle=dashed, dash=0.17638889cm 0.10583334cm](8.09,1.1525)(8.09,0.3525)
\psline[linecolor=black, linewidth=0.04, linestyle=dashed, dash=0.17638889cm 0.10583334cm](10.49,-0.4475)(11.29,0.3525)(11.29,0.3525)
\psline[linecolor=black, linewidth=0.04, linestyle=dashed, dash=0.17638889cm 0.10583334cm](11.29,1.1525)(11.29,0.3525)
\psline[linecolor=black, linewidth=0.04](7.29,-0.4475)(7.29,-0.4475)(7.29,-1.2475)(7.29,-1.2475)
\psline[linecolor=black, linewidth=0.04](12.09,-0.4475)(12.09,-1.2475)(12.09,-1.2475)
\rput(0.89,1.5525){\relax $p$}
\rput(0.09,-1.6475){\relax $d$}
\rput(1.69,-1.6475){\relax $e$}
\rput(4.09,1.5525){\relax $p$}
\rput(3.29,-1.6475){\relax $d$}
\rput(4.89,-1.6475){\relax $e$}
\rput[bl](8.09,1.5525){\relax $p$}
\rput(7.29,-1.6475){\relax $d$}
\rput(8.89,-0.8475){\relax $r$}
\rput(10.49,-0.8475){\relax $r$}
\rput(11.29,1.5525){\relax $p$}
\rput(12.09,-1.6475){\relax $d$}
\psdots[linecolor=black, dotsize=0.2](0.89,-0.4475)
\psdots[linecolor=black, dotsize=0.2](4.09,0.3525)
\psdots[linecolor=black, dotsize=0.2](8.09,0.3525)
\psdots[linecolor=black, dotsize=0.2](11.29,0.3525)
\rput(2.49,-0.0475){=}
\rput(9.69,-0.0475){=}
\rput(6.09,-0.0475){and}
\rput(12.89,-0.0475){.}
\psline[linecolor=black, linewidth=0.04, linestyle=dashed, 
dash=0.17638889cm 0.10583334cm, arrowsize=0.05291666666666667cm 4.0,arrowlength=1.4,arrowinset=0.0]{->}(0.89,-0.4475)(0.89,0.3525)
\psline[linecolor=black, linewidth=0.04, linestyle=dashed, dash=0.17638889cm 0.10583334cm, arrowsize=0.05291666666666667cm 4.0,arrowlength=1.4,arrowinset=0.0]{->}(3.29,-1.2475)(3.29,-0.4475)
\psline[linecolor=black, linewidth=0.04, linestyle=dashed, dash=0.17638889cm 0.10583334cm, arrowsize=0.05291666666666667cm 4.0,arrowlength=1.4,arrowinset=0.0]{->}(4.89,-1.2475)(4.89,-0.4475)
\psline[linecolor=black, linewidth=0.04, linestyle=dashed, dash=0.17638889cm 0.10583334cm, arrowsize=0.05291666666666667cm 4.0,arrowlength=1.4,arrowinset=0.0]{->}(8.09,0.3525)(7.29,-0.4475)
\psline[linecolor=black, linewidth=0.04, linestyle=dashed, dash=0.17638889cm 0.10583334cm, arrowsize=0.05291666666666667cm 4.0,arrowlength=1.4,arrowinset=0.0]{->}(11.29,0.3525)(12.09,-0.4475)
\end{pspicture}
}
\end{center}

\begin{proof}[Proof of Theorem~\ref{THMPresentaionOfOC}]\
We define a morphism $\pi : \Fr[\underline{c}yc] E \to \onsOCKP$ of cyclic
hybrids by
\[
\pi\left(\mu^{\cycs{pqr}}_\emptyset\right) :=
\sur{\cycs{pqr}}{\zvedak 0}[\zvedak \emptyset],\
\pi\left(\omega^\varnothing_{\stt {d\hskip .08em ef}}\right) :=
\sur{\varnothing}{\zvedak 0}[\zvedak \stt {d\hskip .08em ef}]
\and
\pi\left(\phi^{\cycs{p}}_{\stt d}\right) 
:= \sur{\cycs{p}}{\zvedak 0}[\zvedak \stt d].
\]
Let us verify that $\pi$ descends to a morphism 
\begin{equation}
\label{podzim}
\alpha : \Fr[\underline{c}yc] E/\,
\EuScript{J} \to \onsOCKP
\end{equation}
of cyclic hybrids. The compatibility with~(a1) means,
	\begin{gather*}
	\sur{\cycs{pqr}}{\zvedak 0}[\zvedak \emptyset] \ooo{r}{s}
        \sur{\cycs{stu}}{\zvedak 0}[\zvedak \emptyset] 
= \sur{\cycs{pru}}{\zvedak 0}[\zvedak \emptyset] 
\ooo{r}{s} \sur{\cycs{qts}}{\zvedak 0}[\zvedak \emptyset],
\end{gather*}
the compatibility with (a2) leads to  
	\begin{gather*}
	\sur{\varnothing}{\zvedak 0}[\zvedak \stt {def}] \ooo{f}{g} 
\sur{\varnothing}{\zvedak 0}[\zvedak \stt {ghi}] =
        \sur{\varnothing}{\zvedak 0}[\zvedak \stt {dfi}] \ooo{f}{g} 
\sur{\varnothing}{\zvedak 0}[\zvedak \stt {ehg}],
\end{gather*}
the compatibility with (a3) amounts to verifying  	
\begin{gather*}
	\sur{\cycs{p}}{\zvedak 0}[\zvedak  \stt g] \ooo{g}{f} 
\sur{\varnothing}{\zvedak 0}[\zvedak \stt {def}] =
        \left( \sur{\cycs{pqr}}{0}[\zvedak \emptyset] \ooo{q}{s}
          \sur{\cycs{s}}{0}[\zvedak \stt d] \right) \ooo{r}{t}
        \sur{\cycs{t}}{0}[\zvedak \stt e]
      \end{gather*}
and, finally, the compatibility with (a4) translates to 
	\begin{gather*}
	\sur{\cyc{pqr}}{\zvedak 0}[\zvedak \emptyset] \ooo{q}{s}
        \sur{\cycs{s}}{\zvedak 0}[\zvedak \stt d]
        = \sur{\cyc{prq}}{\zvedak 0}[\zvedak \emptyset] \ooo{q}{s}
        \sur{\cycs{s}}{\zvedak 0}[\zvedak \stt d].
      \end{gather*}
All the above equations  follow directly from the
definition of the $\circ$-operations in~$\onsOC$.

We are going to prove
that~(\ref{podzim}) is an isomorphism. Let us start with a couple of
preliminary remarks.
Free operads and operad-like structures are represented by
decorated graphs, as explained at several places, see 
e.g.~\cite[Sections~6 and 9]{markl:handbook}, ~\cite[Section~4]{nsmod}. We
assume that the reader is familiar with this description. In our case,
elements of $\Fr[\underline{c}yc] E$ are connected, simply connected graphs
with three types of vertices as in~(\ref{eq:4}), and two types of
(half)-edges: solid ones representing `open' propagators, and dashed
ones representing `closed' propagators. Moreover, half-edges adjacent to
a vertex representing the open pair of pants  are cyclically ordered.

The associativities (a1) and (a2) enable one to contract propagators
connecting two $\mu$-vertices or two $\omega$-vertices. The result
will be a graph $\Gamma$ with vertices 
\begin{equation}
\label{za_chvili_sraz_s_Vladou}
\raisebox{-2.2em}{\rule{0pt}{1pt}}
\psscalebox{1.0 1.0} 
{
\begin{pspicture}(1,0)(6.7409053,0.76607317)
\rput(6,-0){and}
\rput(7.7,-0.1){,}
\rput(2.4,-0.1){,}
\psline[linecolor=black, linewidth=0.04](1.3194216,0.76602525)(1.3194216,-0.033974763)(1.9194217,0.56602526)(1.9194217,0.56602526)
\psline[linecolor=black, linewidth=0.04](1.3194216,-0.033974763)(2.1194217,-0.033974763)(2.1194217,-0.033974763)
\psline[linecolor=black, linewidth=0.04](1.3194216,-0.033974763)(0.7194217,0.56602526)(0.7194217,0.56602526)
\psline[linecolor=black, linewidth=0.04](1.3194216,-0.033974763)(0.5194217,-0.033974763)(0.5194217,-0.033974763)
\psline[linecolor=black, linewidth=0.04](1.3194216,-0.033974763)(0.7194217,-0.6339748)(0.7194217,-0.6339748)
\psline[linecolor=black, linewidth=0.04](1.3194216,-0.033974763)(1.9194217,-0.6339748)(1.9194217,-0.6339748)
\rput(1.3194216,-0.6339748){$\cdots$}
\psdots[linecolor=black, dotsize=0.2](1.3194216,-0.033974763)
\psline[linecolor=black, linewidth=0.04, linestyle=dashed, dash=0.17638889cm 0.10583334cm](4.319422,0.76602525)(4.319422,-0.033974763)(4.319422,-0.033974763)
\psline[linecolor=black, linewidth=0.04, linestyle=dashed, dash=0.17638889cm 0.10583334cm](4.9194217,0.56602526)(4.319422,-0.033974763)(4.319422,-0.033974763)
\psline[linecolor=black, linewidth=0.04, linestyle=dashed, dash=0.17638889cm 0.10583334cm](5.1194215,-0.033974763)(4.319422,-0.033974763)(4.319422,-0.033974763)
\psline[linecolor=black, linewidth=0.04, linestyle=dashed, dash=0.17638889cm 0.10583334cm](3.5194216,-0.033974763)(4.319422,-0.033974763)(4.319422,-0.033974763)
\psline[linecolor=black, linewidth=0.04, linestyle=dashed, dash=0.17638889cm 0.10583334cm](3.7194216,0.56602526)(4.319422,-0.033974763)(4.319422,-0.033974763)
\psline[linecolor=black, linewidth=0.04, linestyle=dashed, dash=0.17638889cm 0.10583334cm](3.7194216,-0.6339748)(4.319422,-0.033974763)(4.319422,-0.033974763)
\psline[linecolor=black, linewidth=0.04, linestyle=dashed, dash=0.17638889cm 0.10583334cm](4.9194217,-0.6339748)(4.319422,-0.033974763)(4.319422,-0.033974763)
\psdots[linecolor=black, dotsize=0.2](4.319422,-0.033974763)
\rput(4.319422,-0.6339748){$\cdots$}
\rput(.5,0){
\psline[linecolor=black, linewidth=0.04, linestyle=dashed, dash=0.17638889cm 0.10583334cm, arrowsize=0.05291666666666667cm 2.0,arrowlength=1.4,arrowinset=0.0]{->}(6.720906,-0.6339732)(6.7192097,0.16602501)
\psline[linecolor=black, linewidth=0.04](6.7179375,0.76602364)(6.719422,0.16602524)(6.7196336,-0.033974536)
\psline[linecolor=black, linewidth=0.04, arrowsize=0.133cm
2.0,arrowlength=1.4,arrowinset=0.0]{->}(6.719422,-0.033974763)(6.719422,0.16602524)
}
\psarc[linecolor=black, linewidth=0.02, dimen=outer, arrowsize=0.133cm 2.0,arrowlength=1.4,arrowinset=0.0]{->}(1.3194216,-0.033974763){1.2}{139.32887}{223.95837}
\end{pspicture}
}
\end{equation}
which represents an element in the quotient $\Fr[\underline{c}yc] E$ modulo the ideal
generated by (a1) and (a2). We will call the vertices
in~(\ref{za_chvili_sraz_s_Vladou}) the $\mu$-~, $\omega$- and
$\phi$-vertices, respectively. The half-edges adjacent to a $\mu$-vertex
are cyclically ordered. When drawn in the plane, we assume they have
the implicit anti-clockwise cyclic order.

The case when $\Gamma$ has only $\mu$-vertices 
is very special, $\Gamma$ then must be a corolla formed by an~$\omega$-vertex
whose all adjacent half-edges are legs\footnote{I.e., by definition, external
  half-edges.} labelled
by a finite set $C$. The equivalence class of $\Gamma$ in
$\Fr[\underline{c}yc]{E} /\, \EuScript{J}$ is then an element of
biarity $\varnothing \choose C$. 

So assume that $\Gamma$ has at least one $\omega$- or $\phi$-vertex, which
happens if and only if it has at least one solid  half-edge.
Cutting all its internal dashed edges in the middle produces $b$
connected graphs $\Gamma_1,\ldots,\Gamma_b$; the non-negative integer 
$b  \in \bbN_+$ can easily be seen to be  the
number of boundaries of the equivalence class of $\Gamma$ in
$\Fr[\underline{c}yc]{E} /\, \EuScript{J}$. The  open legs of
$\Gamma_i$ are cyclically ordered and their labels form 
for each $1 \leq i \leq b$ a cycle $\sfo_i$. Denoting
by $C$ the set of labels of closed legs,  the equivalence
class of $\Gamma$ in $\Fr[\underline{c}yc]{E} /\, \EuScript{J}$ has
biarity  $\sfO \choose C$ with $\sfO := \sfo_1 \cdots \sfo_b$.  

In both cases, we explicitly assigned to elements of $\Fr[\underline{c}yc]{E} /\,
\EuScript{J}$ a biarity $\sfO \choose C$ preserved by the
map $\alpha$ in~(\ref{podzim}). We denote by 
$\alpha{\sfO \choose C}$ the restriction of this map
to subsets of elements with the indicated biarity.

Let $\Fr[\underline{c}yc] \omega$
be the free cyclic operad generated by $\omega$, $\EuScript{A}$
the ideal generated by the associativity (a2), and $\oComst$  the
cyclic stable commutative operad from
Example~\ref{ze_ja_vul_jsem_zacal_psate_ten_grant_stable} identified
with the cyclic suboperad of $\onsOCKP$ consisting of elements as
in~(\ref{eq:6}). 
It is clear that $\alpha{\varnothing \choose C}$ 
can be identified with the morphism
\[
\Fr[\underline{c}yc] \omega /  \EuScript{A} \to \oComst
\]
that sends $\omega$ to the generator of $\oComst$. 
This map map is an isomorphism
since $\Fr[\underline{c}yc] \omega /  \EuScript{A}$ is the standard presentation
of the  cyclic  
commutative
operad~\cite[Example~II.3.33]{markl-shnider-stasheff:book}, 
so the $\sfO = \varnothing$ case of
Theorem~\ref{THMPresentaionOfOC} is proven. Therefore, from
now on we {\em assume that $\sfO \not= \varnothing$.}
 Notice that for each biarity $\sfO \choose C$
there is either precisely one element in $\onsOCKP$ of that biarity, or none.
To prove that $\alpha{\sfO \choose C}$ is an isomorphism, it is
therefore enough to establish

\begin{lemma}
\label{vcera_v_Arse}
Let us denote by $\left(\Fr[\underline{c}yc]{E} /\,
\EuScript{J}\right){\sfO \choose C}$ resp.\  
$\left(\onsOCKP\right){\sfO \choose C}$ the
subsets of elements of the indicated biarity. Then
\begin{itemize}
\item [(i)]
 $\left(\Fr[\underline{c}yc]{E} /\,
\EuScript{J}\right){\sfO \choose C}$ is either empty or a one-point
set and
\item [(ii)]
$\left(\onsOCKP\right){\sfO \choose C} \not= \emptyset$ implies that 
$\left(\Fr[\underline{c}yc]{E} /\,
\EuScript{J}\right){\sfO \choose C}\not= \emptyset$.
\end{itemize}
\end{lemma}

Our strategy of the proof will be to modify the graph
$\Gamma$, bringing it in a `canonical' form~(\ref{doutnik_odhalen}),
and show that this form is uniquely determined by the biarity. Let us
start the process of modification of $\Gamma$.

Since $\Gamma$ has at least one solid (half)-edge, we may use the
morphism property (a3) to eliminate all its $\omega$-vertices. 
The only dashed internal edges will then be of the form
\begin{equation}
\label{musim_na_veceri_ale_nechce_se_mi}
\psscalebox{1.0 1.0} 
{

}
\end{equation}
where the two gray cycles indicate (possibly empty) subgraphs. The only dashed
legs are of the form
\[
\psscalebox{1.0 1.0} 
{
%
}
\]
with the label $u$ belonging to  the set $C$ of closed inputs. 
The local structure of $\Gamma$
around a~$\mu$-vertex looks as in
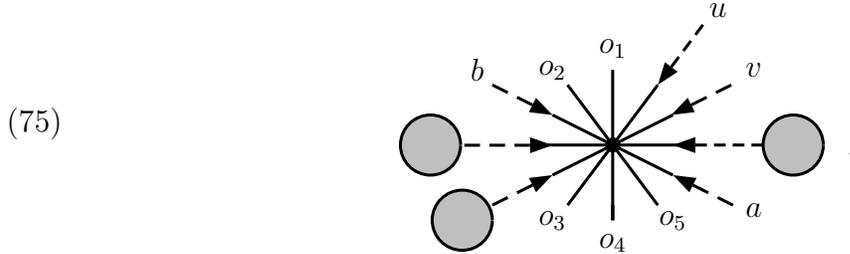
\begin{equation}
\raisebox{-4.2em}{\rule{0pt}{1pt}}
\label{dnes_sraz_s_gymplem_u_Ceskeho_lva}
\psscalebox{1.0 1.0} 
{
%
}
\end{equation}
Here the solid legs represent open inputs in the boundary cycle 
$\cycs{o_1,\ldots,o_5}$, dashed
legs closed inputs labelled by $u,a,b,v \in C$, and the gray
circles are some subgraphs. The graph~$\Gamma$ may also have 
open inputs appearing e.g.\ as
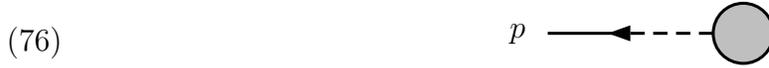
\begin{equation}
\label{eq:9}
\psscalebox{1.0 1.0} 
{
%
}
\end{equation}
that corresponds to the open boundary component 
$\cycs{p}$. Finally, empty boundary components are introduced by 
solid  edges connecting two $\phi$-vertices:
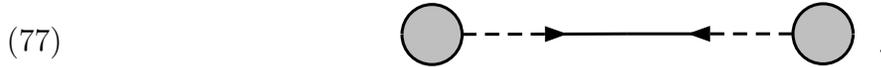
\begin{equation}
\label{eq:10}
\psscalebox{1.0 1.0} 
{
\rput(6,0){.}
%
}
\end{equation}

The centrality~(a4) implies that the position of an edge connecting a
$\mu$-vertex with a~$\phi$-vertex, call it a {\em $(\mu,\phi)$-edge\/}, 
and the positions of other
half-edges adjacent to the same $\mu$-vertex  can be interchanged, so the
$(\mu,\phi)$-edges are not subjected to the cyclic order. What we mean should be
clear from the following particular example of four half-edges
\hbox{adjacent to $\mu$}:
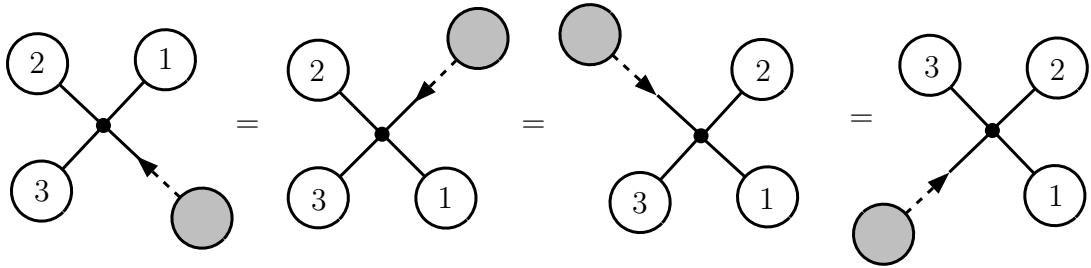
\[
\psscalebox{1 1}
{
%
}
\]
where the numbered cycles are arbitrary possibly empty subgraphs. 
The above equalities can be proved by successive applications of the
associativity~(a1) and centrality~(a4). For instance, the middle
equality follows from 
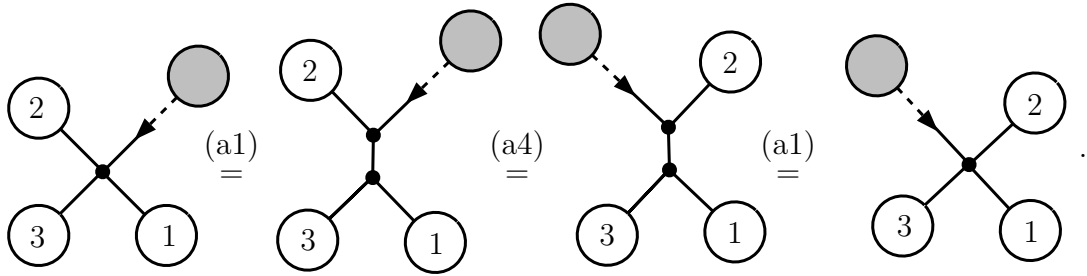
\[
\psscalebox{1.0 1.0} 
{
%
}
\]
We leave the formulation and proof for an arbitrary number
of (half)edges adjacent  to a $\mu$-vertex 
as an exercise. In particular, all $(\mu,\phi)$-edges can be
mutually interchanged. Consequently, the half
edges adjacent to the vertex
in~(\ref{dnes_sraz_s_gymplem_u_Ceskeho_lva}) can be rearranged as 
in the left picture in
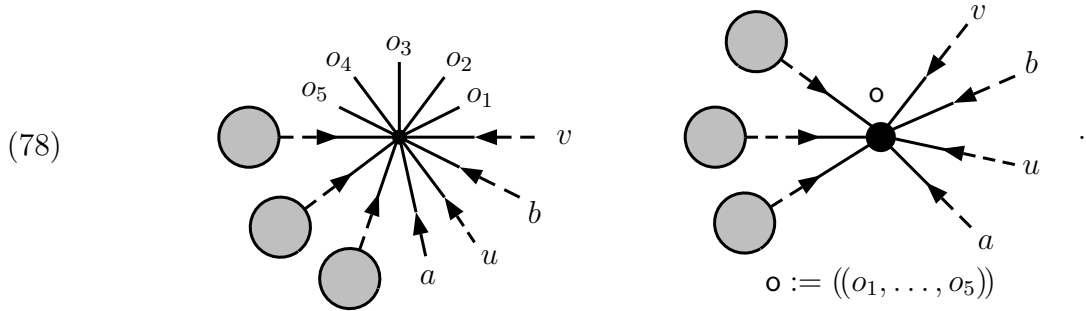
\begin{equation}
\label{za_chvili_do_Lekninu}
\raisebox{-5.2em}{\rule{0pt}{1pt}}
\psscalebox{1.0 1.0} 
{
%
}
\end{equation}
The only half-edges subject to the cyclic order are those labelled by
$o_1,\ldots,o_5$. The local structure around the above $\mu$-vertex can
therefore be encoded by a `fat' vertex \fatvert\ labelled by the corresponding
cycle, as shown in the right picture of~(\ref{za_chvili_do_Lekninu}).  
We therefore have a graph with the local structure around fat vertices
as show  below:
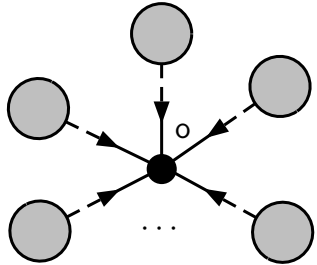
\begin{center}
\psscalebox{1.0 1.0} 
{
%
}
\end{center}
with no order imposed on the adjacent half-edges.
All edges adjacent to  these fat vertices are connected to a
$\phi$-vertex. The labelling cycle $\sfo$ might be arbitrary except for the
case when~$\fatvert$~has only one adjacent half-edge; we then 
require $\sfo$ to be non-empty,
i.e.\ we exclude
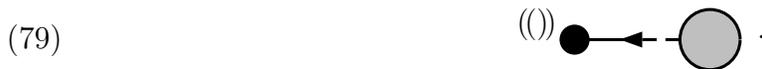
\begin{equation}
\label{dnes_na_Janotu}
\psscalebox{1.0 1.0} 
{
%
}  
\end{equation}
We also exclude the fat vertex $\fatvert^{\hbox{$\sfo$}}$ standing alone when
$|\sfo| \leq 3$.
Finally, we absorb open boundary components in~(\ref{eq:9})
and~(\ref{eq:10}) into this notation by identifying
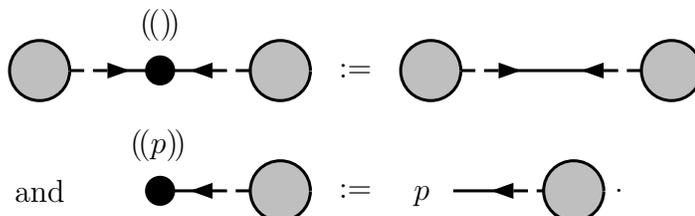
\begin{center}
\psscalebox{1.0 1.0} 
{
%
}  
\end{center}
Another tool which we use will be the {\em sliding rule\/} claiming
that the configuration
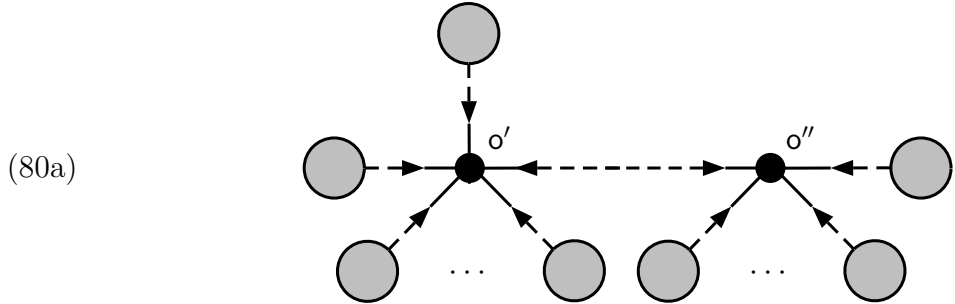
\begin{subequations}
\begin{equation}
\raisebox{-4em}{}
\label{sliding1}
\psscalebox{1.0 1.0} 
{
%
}  
\end{equation}
is, modulo $\EuScript{J}$, the same as
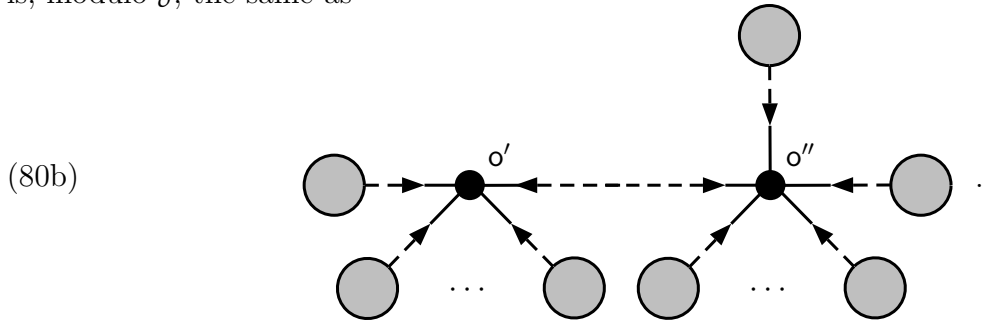
\begin{equation}
\raisebox{-4.3em}{}
\label{sliding2}
\psscalebox{1.0 1.0} 
{
%
}
\end{equation}
\end{subequations}
In words, the sliding rule claims that an arbitrary half-edge
adjacent to a fat vertex can be amputated, moved along the graph,
and attached to another fat vertex. The only restriction is that  
in doing so we must not create a  forbidden fat
vertex~(\ref{dnes_na_Janotu}).  The proof of the sliding rule is given
in Figure~\ref{podhoransky_svah}.
\begin{figure}
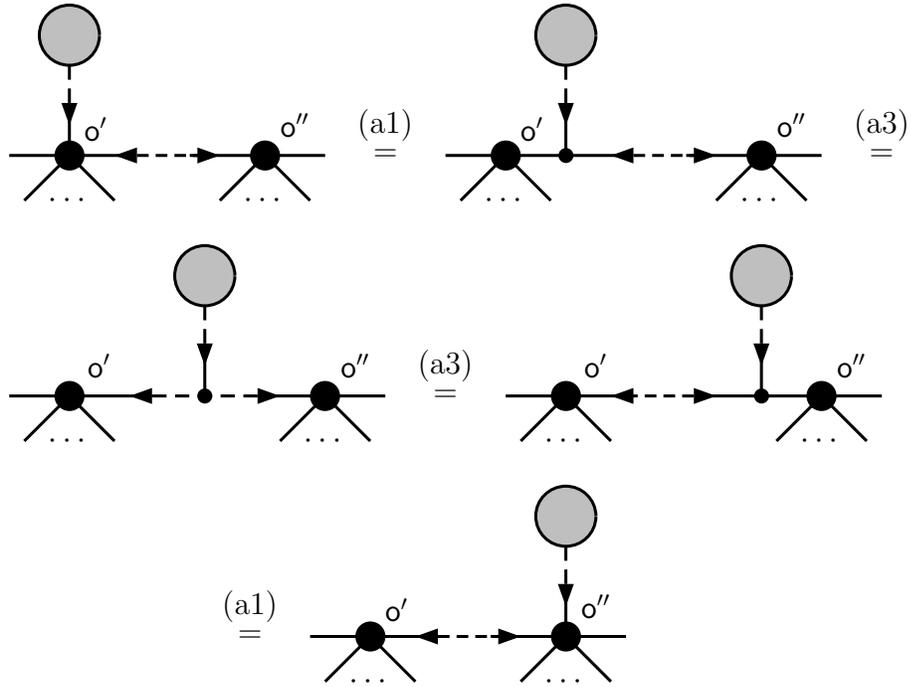

\begin{center}
\psscalebox{1.0 1.0} 
{
%
}
\end{center}
\caption{\label{podhoransky_svah}
Proof of the sliding rule: subsequent use of axioms (a1) and (a3).}
\end{figure}
It is clear that, using the sliding rule, the graph $\Gamma$ can be
brought, modulo $\EuScript{J}$, to the following linear form:
\begin{equation}
\raisebox{-2.8em}{}
\label{doutnik_odhalen}
\psscalebox{1.0 1.0} 
{
%
}
\end{equation}
where the open legs are labelled by elements of $C$ and $\sigma$ is a
permutation of the set $\stt {1,\ldots,b}$. As the next step we show that
the order of cycles is~(\ref{doutnik_odhalen}) is not
substantial. Concretely, we show that
\begin{subequations}
\begin{equation}
\raisebox{-2.3em}{}
\label{cistim_kolecko}
\psscalebox{1.0 1.0} 
{
%
}
\end{equation}
is, modulo $\EuScript{J}$, the same as
\begin{equation}
\raisebox{-2.3em}{}
\label{za_14_dni_do_Sydney}
\psscalebox{1.0 1.0} 
{
%
}
\end{equation}
\end{subequations}
Using the associativity (a1) of $\mu$ we modify~(\ref{cistim_kolecko}) into
\begin{center}
\psscalebox{1.0 1.0} 
{
%
}
\end{center}
the morphism property (a4) turns it into 
\begin{center} 
\psscalebox{1.0 1.0} 
{
%
}
\end{center}
while the associativity (a2) of the `open' multiplication $\omega$ together with
its commutativity shows that the above graph is, modulo
$\EuScript{J}$, the same as the graphs
\begin{center}
\psscalebox{1.0 1.0} 
{
%
}
\end{center}
Backtracking the above modifications we convert the graph in the 
right hand side of
the above equality into~(\ref{za_14_dni_do_Sydney}).

We are finally ready to prove Lemma~\ref{vcera_v_Arse}.
As before, $b$ denotes the number of boundaries of $\sfO$
and we assume that $b \geq 1$. Let $\sfo_1\cdots \sfo_{b'}$ be all
nontrivial cycles in $\sfO$ so that  $\sfO 
= \sfo_1\cdots \sfo_{b'} \cycs{} \cdots
\cycs{}$ with $b'' := b-b'$ trivial cycles $\cycs{}$.
We distinguish four cases.

\noindent 
{\em Case $b' \geq 2$.} Using the
commutativity~(\ref{cistim_kolecko})--(\ref{za_14_dni_do_Sydney}) we
can rearrange~(\ref{doutnik_odhalen}) so that the labels of the fat
vertices read from the left to the right are
\[
\sfo_1, \cycs{}, \ldots , \cycs{}, \sfo_2, \ldots ,\sfo_{b'}.
\]
Since both $\sfo_1$ and $\sfo_{b'}$ are nontrivial,
the forbidden vertices~(\ref{dnes_na_Janotu}) cannot occur
so that, according to the sliding
rule~(\ref{sliding1})--(\ref{sliding2}), the positions of closed legs are
not constrained. In other words, for each biarity $\sfO \choose C$ with at
  least two nontrivial cycles in $\sfO$ there exists exactly one
  isomorphism class of graphs in  $\Fr[\underline{c}yc]{E} /\, \EuScript{J}$
  with that biarity. 

\noindent 
{\em Case $b' =1$, $b'' \geq 1$.} With the aid of
commutativity~(\ref{cistim_kolecko})--(\ref{za_14_dni_do_Sydney}) 
we order the fat vertices of~(\ref{doutnik_odhalen})  from the 
left to the right into
\[
\sfo_1, \cycs{}, \ldots , \cycs{}.
\]
To avoid the forbidden ones, 
the rightmost fat vertex must be adjacent to  at least one open leg,
which may happen only when 
$C \not= \emptyset$. All remaining open legs can be then, using the sliding
rule~(\ref{sliding1})--(\ref{sliding2}), 
transferred to the rightmost fat vertex, so their positions are
irrelevant. We conclude that if  $C \not=
\emptyset$,  $\Fr[\underline{c}yc]{E} /\, \EuScript{J}$ contains 
exactly one element of biarity  ${\sfO \choose C}$ while there are  no
elements of this biarity if $C =
\emptyset$.

\noindent 
{\em Case $b' =1$, $b'' =0$.}
The graph $\Gamma$ is a corolla around a fat vertex which is clearly an
allowed one if and only if the stability $2|C| + |O| > 2$ is satisfied.  

\noindent 
{\em Case $b' =0$.} 
Since $b = b' + b'' \not= 0$,  $b'' \geq 1$
and all fat vertices in~(\ref{doutnik_odhalen}) are labelled by the
trivial cycle $\cycs{}$. To avoid forbidden fat vertices at both
extremities, we need  $|C| \geq 2$ otherwise there will be no graphs
of biarity ${\sfO \choose C}$. 
If $|C| = 2$, there is precisely one open leg at both sides
of~(\ref{doutnik_odhalen}) and, due to the obvious left-right symmetry of the
graph, the labels of these legs can be interchanges. If $|C| \geq 3$,
the sliding rule applies so
the positions of open legs are irrelevant as well.

We see that in all four cases, (i) of Lemma~\ref{vcera_v_Arse} is
satisfied. The second part can be  verified easily by 
comparing the list of elements
belonging to $\onsOCKP$ given in Example~\ref{KP} with the above
calculations. This finishes the proof of the lemma and therefore also
of the theorem.
\end{proof}

Theorems~\ref{THMOCUnivPropKP} and~\ref{THMPresentaionOfOC} together give:

\begin{theorem} 
\label{THMPresentationOfQOCKP}
The modular hybrid $\oQnsOC_\KP$ has the following presentation.
It is generators are (g1)-(g3) of Theorem~\ref{THMPresentaionOfOC} 
and the relations are (a1)-(a4) of Theorem \ref{THMPresentaionOfOC},
together with the Cardy condition
\begin{equation}
\label{pokusim_se_Jarce_pripojit_modem}
\oxi{uv} \left( 
\omega^{\cycs{uqa}}_\emptyset \ooo{a}{b}\,
\omega^{\cycs{bvr}}_\emptyset \right) = \phi^{\cycs{q}}_{\stt c}
\ooo{c}{d} \phi^{\cycs{r}}_{\stt d}.
\end{equation}
\end{theorem}

\begin{proof}
It follows from the commutativity of diagrams in
Example~\ref{DG_307_v_Akropoli} combined 
with~(\ref{pristi_tyden_vyjezdni_zasedani_Brno}) that
\begin{equation}
\label{eq:11}
\ModHybFun(\onsOCKP) \ \cong\
\ModHybFun\big(\Fr[\underline{c}yc]{E} /\, \EuScript{J}\big)
\cong  \ModHybFun\big(\Fr[\underline{c}yc]{E}\big)   /\, \EuScript{J}
\cong
 \Fr[\underline{m}od]{E} /\, \EuScript{J}, 
\end{equation}
where $ \Fr[\underline{m}od]-$ is the free modular hybrid functor, and the
collection $E$ and the ideal $\EuScript{J}$ have the same generators as
in Theorem~\ref{THMPresentaionOfOC}.

Theorem~\ref{THMOCUnivPropKP} combined with~(\ref{eq:11}) implies
that the modular hybrid $\oQnsOCKP$ is isomorphic to the quotient of
$\Fr[\underline{m}od]{E}$ by $\EuScript{J}$
and relation~(\ref{jsem_zmitan}). The proof is finished by observing that the
isomorphisms~(\ref{eq:11}) translates~(\ref{jsem_zmitan})
into~(\ref{pokusim_se_Jarce_pripojit_modem}). 
\end{proof}

In Example~\ref{dnes_vecer_s_Jarkou_a_Hankou_k_Pakousum} 
we defined algebras over cyclic
hybrids. The finitary presentation of  $\onsOCKP$ given in
Theorem~\ref{THMPresentaionOfOC} offers an explicit
description of its algebras. Recall that a {\em Frobenius algebra\/}
on a vector space $A$ equipped with a non-degenerate symmetric bilinear form
$\beta_A$ has an associative multiplication $\mu_A : A \ot A \to A$
such that the expression
\begin{equation}
\label{pozitri_do_Brna}
\beta_A \big(\mu_A(a_1,a_2),a_3\big) \in \bfk
\end{equation}
is cyclically invariant in $a_1,a_2,a_3\in A$. A Frobenius algebra is
{\em commutative\/} if~(\ref{pozitri_do_Brna}) is invariant under all permutations
of  $a_1,a_2$ and $a_3$; this forces $\mu_A$ to be commutative.

\begin{theorem} 
\label{THMOCKPAlgebras}
An algebra over the Kaufmann-Penner 
cyclic hybrid $\onsOC_\KP$ on a pair $A,B$ of finite
dimensional vector spaces equipped with symmetric non-degenerate bilinear 
forms $\beta_A,\beta_B$ is the same as
\begin{enumerate}
\item[(i)] 
a Frobenius algebra on $A$ with the associated form $\beta_A$,
\item[(ii)] 
a commutative Frobenius algebra on $B$ with the associated form
$\beta_B$, and
\item[(iii)] 
an associative algebra morphism $B\to A$ with values in the center of $A$.
\end{enumerate}
\end{theorem}

\begin{proof}
By definition, an  $\onsOC_\KP$-algebra is a morphism of cyclic hybrids 
$\alpha : \oOC_\KP \to \End_{A,B}$. Let $\mu,\omega$ and $\phi$ be the
generators of $\onsOC_\KP$ as in (g1)--(g3) of
Theorem~\ref{THMPresentaionOfOC}. Since the bilinear forms $\beta_A$
and $\beta_B$ are non-degenerate, the equations
\begin{equation}
\label{vcera_v_Arse_na_UH_a_Chadimovi}
\beta_A \big(\mu_A(a_1,a_2),a_3\big) = \alpha(\mu)(a_1 \ot a_2 \ot a_3),\
\beta_B \big(\omega_B(c_1,c_2),c_3\big)  
= \alpha(\omega)(c_1 \ot c_2 \ot c_3),\
\end{equation}
$a_i \in A$, $c_i \in B$, $i = 1,2,3$,
define bilinear maps $\mu_A : A \ot A \to A$ and  $\omega_B : B \ot B
\to B$ while $f := \alpha(\phi)$ is a linear map $B \to A$. 

It is easy to show that (a1) of Theorem~\ref{THMPresentaionOfOC}
translates to the associativity of $\mu_A$ and (a2) to 
the associativity of $\omega_B$. The symmetry, i.e.\ the commutativity
of $\omega_B$, follows from the invariance of $\omega$ under the group of
automorphisms of its inputs. Likewise, the morphism property~(a3)
implies that $f: B\to A$ is an algebra morphism while the
centrality~(a4) implies (iii) of the theorem. Finally, the symmetry
of the expressions~(\ref{pozitri_do_Brna}) 
for $\mu_A$ resp.\ $\omega_B$ follows from the
defining equations~(\ref{vcera_v_Arse_na_UH_a_Chadimovi}) 
and the cyclic symmetry of $\mu$ resp.~$\omega$. 
\end{proof}

Theorem~\ref{THMPresentationOfQOCKP} offers the following description of
algebras for the modular hybrid $\oQnsOCKP$ in the spirit of 
the classical result about 2-dimensional topological 
field theories~\cite{Kock:TFT}, see 
also~\cite[Theorem~5.4]{kaufmann-penner:NP06} and~\cite[Section~4]{lauda}.

\begin{theorem}
\label{omylem_jsem_si_koupil_bio_banany}
An algebra for the \KaPe\ modular hybrid $\oQnsOC_\KP$ on a pair 
$A,B$ of vector spaces with symmetric nondegenerate bilinear forms
$\beta_A,\beta_B$ is the same as
\begin{enumerate}
\item[(i)] 
a Frobenius algebra $(A,\mu_A,\beta_A)$,
\item[(ii)] 
a commutative Frobenius algebra $(B,\omega_B,\beta_B)$, and
\item[(iii)] 
an associative algebra morphism $f: B\to A$ with values in the center
of $A$,
\end{enumerate}
satisfying the Cardy condition
\begin{equation}
\label{v_patek_vzpominka_na_Martina}
\beta_A(\mu_A\ot\mu_A)(\id\ot\tau\ot\id)(\id\ot\id\ot\beta_A^{-1}) =
(\beta_A\ot\beta_A)
\big(\id\ot (f \ot f)\beta_B^{-1} \ot\id\big),
\end{equation}
where $\tau$ is the standard symmetry in the monoidal category of
graded vector spaces.
\end{theorem}

In~(\ref{v_patek_vzpominka_na_Martina}), $\beta_A^{-1}$ is the
inverse  of $\beta_A : A \ot A \to \bfk$, i.e.\  the unique linear map
$\beta_A^{-1} :\bfk \to  A \ot A$ satisfying
\[
(\beta_A \ot \id)(\id \ot \beta_A^{-1}) = (\id \ot
\beta_A)(\beta_A^{-1} \ot \id) = \id_A; 
\]
the inverse  $\beta_B^{-1} :\bfk \to  B \ot B$ is defined similarly.

\begin{proof}[Proof of Theorem~\ref{omylem_jsem_si_koupil_bio_banany}]
It follows
the pattern of the proof of Theorem~\ref{THMOCKPAlgebras} and we leave
the details to the reader.
A pictorial form of the Cardy
condition~(\ref{v_patek_vzpominka_na_Martina}) is
\begin{center}
\raisebox{-4em}{}
\psscalebox{1.0 1.0}
{
\begin{pspicture}(0,.4)(7.651398,1.690202)
\psarc[linecolor=black, linewidth=0.04, dimen=outer](1.231398,0.330202){0.8}{359}{180}
\psarc[linecolor=black, linewidth=0.04, dimen=outer](1.031398,-0.46979797){1.0}{126.47182}{182.05298}
\psarc[linecolor=black, linewidth=0.04, dimen=outer](2.631398,-0.46979797){1.0}{126.47182}{182.05298}
\psarc[linecolor=black, linewidth=0.04, dimen=outer](-0.16860199,-0.46979797){1.0}{0.19398086}{53.948467}
\psarc[linecolor=black, linewidth=0.04, dimen=outer](1.431398,-0.46979797){1.0}{0.19398086}{53.948467}
\psarc[linecolor=black, linewidth=0.04, dimen=outer](1.631398,-0.46979797){0.8}{174.4278}{359.9}
\psline[linecolor=black, linewidth=0.04](1.631398,-0.46979797)(1.631398,-1.069798)
\psarc[linecolor=black, linewidth=0.04, dimen=outer](4.631398,0.53020203){0.6}{1.8445666}{179.10997}
\psarc[linecolor=black, linewidth=0.04, dimen=outer](7.031398,0.53020203){0.6}{1.8445666}{179.10997}
\psarc[linecolor=black, linewidth=0.04, linestyle=dashed, dash=0.17638889cm 0.10583334cm, dimen=outer](5.831398,-1.069798){0.6}{187}{0}
\psline[linecolor=black, linewidth=0.04, linestyle=dashed, dash=0.17638889cm 0.10583334cm, arrowsize=0.05291666666666667cm 2.0,arrowlength=1.4,arrowinset=0.0]{->}(5.231398,-1.069798)(5.231398,-0.06979797)
\psline[linecolor=black, linewidth=0.04, linestyle=dashed, dash=0.17638889cm 0.10583334cm, arrowsize=0.05291666666666667cm 2.0,arrowlength=1.4,arrowinset=0.0]{->}(6.431398,-1.069798)(6.431398,-0.06979797)
\psdots[linecolor=black, dotsize=0.2](0.431398,0.330202)
\psdots[linecolor=black, dotsize=0.2](2.031398,0.330202)
\rput(1.231398,1.530202){$\beta_A$}
\rput(4.631398,1.530202){$\beta_A$}
\rput(7.031398,1.530202){$\beta_A$}
\rput(2.431398,0.53020203){$\mu$}
\rput(0.831398,0.53020203){$\mu$}
\rput(6.831398,-0.46979797){$\phi$}
\rput(5.631398,-0.46979797){$\phi$}
\rput(5.831398,-1.2697979){$\beta_B$}
\rput(3.3,-0.26979798){$=$}
\rput(8,-0.26979798){.}
\psline[linecolor=black, linewidth=0.04](1.631398,-1.469798)(1.631398,-1.669798)
\psline[linecolor=black, linewidth=0.04](0.03139801,-0.46979797)(0.03139801,-1.669798)
\psline[linecolor=black, linewidth=0.04](4.031398,0.53020203)(4.031398,-1.669798)
\psline[linecolor=black, linewidth=0.04](7.631398,0.53020203)(7.631398,-1.669798)
\psline[linecolor=black, linewidth=0.04](5.231398,0.53020203)(5.231398,-0.26979798)
\psline[linecolor=black, linewidth=0.04](6.431398,0.53020203)(6.431398,-0.26979798)
\end{pspicture}
}  
\end{center}
\end{proof}

\def\cprime{$'$}\def\cprime{$'$}

\end{document}